\documentclass{amsart}
\usepackage{graphicx}
\usepackage{latexsym}
\usepackage{hyperref}
\usepackage{amsfonts}
\usepackage[all]{xy}
\usepackage{amssymb, mathrsfs, amsfonts, amsmath}
\usepackage{color}
\usepackage{appendix}

\usepackage{enumerate}

\usepackage{eulervm}
\usepackage{geometry}

\newenvironment{\thesection}{\section}{Appendix}

\newtheorem{theorem}{Theorem}[section]

\newtheorem{corollary}[theorem]{Corollary}

\newtheorem{definition}[theorem]{Definition}
\newtheorem{example}[theorem]{Example}
\newtheorem{question}[theorem]{Question}

\newtheorem{lemma}[theorem]{Lemma}

\newtheorem{proposition}[theorem]{Proposition}
\newtheorem{remark}[theorem]{Remark}

\newcommand{\R}{\mathbb{R}}

\newcommand{\Sym}{\mathrm{Sym}}
\newcommand{\metric}{\langle \, , \, \rangle}
\newcommand{\inte}{\mathrm{Int}}
\newcommand{\disp}{\displaystyle}
\newcommand{\red}{\mathrm{red}}
\newcommand{\ra}{\rightarrow}
\newcommand{\str}{\mathrm{str}}
\newcommand{\GG}{\mathbb{G}}
\newcommand{\eps}{\varepsilon}

\newcommand{\II}{\mathrm{II}}
\newcommand{\Sph}{\mathbb{S}}
\newcommand{\di}{\mathrm{d}}
\newcommand{\weak}{\mathrm{w}}
\newcommand{\CC}{\mathbb{C}}
\newcommand{\EE}{\mathscr{E}}
\newcommand{\HH}{\mathscr{H}}

\newcommand{\Ricc}{\mathrm{Ric}}
\newcommand{\cut}{\mathrm{cut}}
\newcommand{\vol}{\mathrm{vol}}
\newcommand{\lip}{\mathrm{Lip}}
\newcommand{\cal}{\mathcal}
\newcommand{\USC}{\mathrm{USC}}

\DeclareMathOperator{\dist}{dist}

\DeclareMathOperator{\tr}{\mathrm{Tr}}

\DeclareMathOperator{\dif}{\mathrm{d}}

\begin{document}

\title[Ahlfors-Liouville and Khas'minskii properties for subequations]{Duality between Ahlfors-Liouville and Khas'minskii properties for non-linear equations}  
{\author{Luciano Mari}
\address{Scuola Normale Superiore\\56126 Pisa-Italy}
\email{luciano.mari@sns.it, mari@mat.ufc.br} \thanks{The first author was partially supported by CNPq-Brazil and by research funds of the Scuola Normale Superiore}
{\author{Leandro F. Pessoa}
\address{Departamento de Matem\'{a}tica\\Universidade Federal do Piau\'{i}-UFPI\\
64049-550, Teresina-Brazil} \email{leandropessoa@ufpi.edu.br} \thanks{The second author was partially supported by Capes and CNPq-Brazil}}

\begin{abstract}
In recent years, the study of the interplay between (fully) non-linear potential theory and geometry received important new impulse. The purpose of this work is to move a step further in this direction by investigating appropriate versions of parabolicity and maximum principles at infinity for large classes of non-linear (sub)equations $F$ on manifolds. The main goal is to show a unifying duality between such properties and the existence of suitable $F$-subharmonic exhaustions, called Khas'minskii potentials, which is new even for most of the ``standard" operators arising from geometry, and improves on partial results in the literature. Applications include new characterizations of the classical maximum principles at infinity (Ekeland, Omori-Yau and their weak versions by Pigola-Rigoli-Setti) and of conservation properties for stochastic processes (martingale completeness). Applications to the theory of submanifolds and Riemannian submersions are also discussed.
\end{abstract}

\maketitle
\tableofcontents

\section{Introduction}\label{sec_intro}

The purpose of the present paper is to describe a duality principle between some function-theoretic properties on non-compact manifolds. We are mainly interested in fully non-linear equations coming from geometry, and to introduce the problems under investigation we first need to fix some terminology. Let $X$ be a Riemannian manifold, and let  $J^2(X) \ra X$ denote the 2-jet bundle. Via the splitting 
$$
J^2(X) = \mathbf{R} \oplus T^*X \oplus \Sym^2(T^*X)
$$
induced by the Levi-Civita connection, the $2$-jet $J^2_xu$ of a function $u$ at a point $x$ can be identified with the $4$-ple 
$$
\big(x,u(x), \di u(x), \nabla^2 u(x)\big),
$$
with $\nabla^2 u$ the Hessian of $u$. Points in the $2$-jet bundle will henceforth be denoted with 
$$
(x,r,p,A), \qquad \text{with} \quad x \in X, \ \ r \in \R, \ \ p \in T^*_xX, \ \  A \in \Sym^2(T^*_xX). 
$$
According to the elegant approach pioneered by N.V. Krylov \cite{krylov} and systematically developed by R. Harvey and B. Lawson Jr. in recent years (\cite{HL_primo, HL_dir, HL_existence}), a differential inequality for a function $u$ can be viewed as a constraint on the two-jets of $u$ to lie in a region $F \subset J^2(X)$. For instance, $u \in C^ 2(X)$ enjoys the inequality $\Delta u \ge f(x,u,  \di u)$ provided that, at each point $x$, the two-jet $J^2_xu$ belongs to the subset
\begin{equation*}\label{esempidoidi}
F = \Big\{ (x,r,p,A) \in J^2(X) \ : \ \tr(A) \ge f(x,r,p)\Big\}.
\end{equation*}
The class of differential inequalities we are interested in are modelled by well-behaved subsets $F$ called \emph{subequations}. The conditions placed on $F$ to be a subequation are mild and correspond to the ellipticity and properness in the standard terminology (see for instance \cite{CIL, caffacabre} and the appendix of \cite{HL_existence}), plus a topological requirement. A function $u \in C^2(X)$ is called $F$-subharmonic if $J^2_xu \in F_x$ ($F_x$ the fiber over $x$) for each $x \in X$, and in this case we write $u \in F(X)$. Via the use of test functions in the viscosity sense, $F$-subharmonicity can be extended to USC (upper-semicontinuous), $[-\infty,+\infty)$-valued functions $u$. For $K \subset X$ closed, we will write $u \in F(K)$ if $u \in \USC(K)$ and $u \in F(\inte K)$. While $F$-subharmonic functions describe subsolutions, supersolutions are taken into account by considering the \emph{Dirichlet dual} of $F$,
$$
\widetilde F \doteq -(\sim \inte F),
$$
that is, roughly speaking, $u$ is a supersolution if $-u \in \widetilde{F}(X)$; $u$ is then called $F$-harmonic provided that $u \in F(X)$, $-u \in \widetilde F(X)$. In particular $u \in C(X)$, and $J^ 2_xu \in \partial F_x$ when $u$ is $C^2$ around $x$. For more details, we refer the reader to the next section, and also to the original sources \cite{krylov, HL_primo, HL_dir, HL_equivalence}.

Relevant subequations to which our results apply include the following examples, where $f \in C(\R)$ is non-decreasing  and $\{\lambda_j(A)\}$ denotes the increasing sequence of eigenvalues of $A \in \Sym^2(T^*X)$:\label{esempi_intro}
\begin{itemize}
\item[$(\EE 1)$] the eikonal $E = \big\{|p| \le 1\big\}$;
\item[$(\EE 2)$] $F= \big\{\tr(A) \ge f(r)\big\}$; 
\item[$(\EE 3)$] $F= \big\{\lambda_k(A) \ge f(r)\big\}$, for $k \in \{1,\ldots, \dim X\}$;
\item[$(\EE 4)$] $F= \big\{\mu_j^{(k)}(A) \ge f(r)\big\}$, where $\mu_j^{(k)}(A)$ is the $j$-th eigenvalue of the $k$-th elementary symmetric function $\sigma_k$ of $\{\lambda_j(A)\}$. Analogous examples can be given for each Dirichlet-G\"arding polynomial, see \cite{HL_garding, HL_gardingesub};
\item[$(\EE 5)$] $F= \big\{\lambda_1(A)+ \cdots + \lambda_k(A) \ge f(r)\big\}$, $k \le m$, whose $F$-subharmonics  (when $f =0$) are called $k$-plurisubharmonic functions; these subequations, which naturally appear in the theory of submanifolds, have been investigated for instance in \cite{wu,sha, HL_plurisub,HL_existence};
\item[$(\EE 6)$] the complex analogues of the last three examples;
\item[$(\EE 7)$] the quasilinear subequation describing viscosity solutions of 
\begin{equation}\label{eq_quasilinear}
\Delta_a u \doteq \mathrm{div}\big(a(|\nabla u|)\nabla u\big) \ge f(u),
\end{equation}
for $a \in C^1(\R^+)$ satisfying 
\begin{equation}\label{ipo_a}
\lambda_1(t) \doteq a(t) + ta'(t) \ge 0, \qquad \lambda_2(t) \doteq a(t)>0.
\end{equation}
Indeed, expanding the divergence we can set 
\begin{equation*}
F = \overline{\left\{ p \neq 0, \ \tr \big( T(p)A\big) > f(r) \right\}}, 
\end{equation*}
where 
\begin{equation*}
T(p) \doteq a(|p|)\metric + \frac{a'(|p|)}{|p|} p \otimes p = \lambda_1(|p|) \Pi_p + \lambda_2(|p|) \Pi_{p^\perp},
\end{equation*}
and $\Pi_p, \Pi_{p^\perp}$ are, respectively, the $(2,0)$-versions of the orthogonal projections onto the spaces $\langle p\rangle$ and $p^\perp$. This family includes the $k$-Laplacian for $1 \le k < +\infty$ and the mean curvature operator;\vspace{0.2cm}
\item[$(\EE 8)$] the normalized $\infty$-Laplacian $F= \overline{\big\{p \neq 0, \ |p|^{-2}A(p,p) > f(r)\big\}}$. 
\end{itemize}
\vspace{0.1cm}
One of the purposes behind Harvey-Lawson's approach is to develop a non-linear potential theory that could be applied to many different geometrical settings. In this respect, the following Liouville  property naturally appears in the investigation of non-compact manifolds.
\begin{definition}\label{def_Liouville} A subequation $F\subset J^2(X)$ is said to satisfy the \emph{Liouville  property} if bounded, non-negative $F$-subharmonic functions on $X$ are constant.
\end{definition}
The prototype example is the subequation $\{\tr(A) \ge 0\}$, for which the Liouville  property is one of the equivalent characterizations of the \emph{parabolicity} of $X$, a notion that is also introduced in terms of zero capacity of compact sets, non-existence of positive Green kernels on $X$, or recurrency of the Brownian motion on $X$ (see \cite[Thm. 5.1]{grigoryan} and the historical discussion therein\footnote{Here and in what follows, we refer to the Brownian motion as the diffusion process whose transition probability is given by the \emph{minimal} positive heat kernel.}). Another relevant example is the Liouville  property for the subequation $\{\tr(A) \ge \lambda r\}$, $\lambda >0$, which by \cite[Thm. 6.2]{grigoryan} is equivalent to the \emph{stochastic completeness} of $X$, that is, to the property that the Brownian motion on $X$ has infinite lifetime almost surely. As first observed by L.V. Ahlfors (see \cite[Thm. 6C]{ahlforssario}) for $\{\tr(A) \ge 0\}$ on Riemann surfaces, the Liouville property is tightly related to a maximum principle on unbounded subsets which has been recently investigated in depth (also for general linear and quasilinear subequations) in \cite{aliasmastroliarigoli,aliasmirandarigoli} under the name of the \emph{open form of the maximum principle}. We suggest the book \cite{aliasmastroliarigoli} for a thorough discussion, and also \cite{imperapigolasetti, pigolasettitroyanov}. Their approach motivates the next
  
\begin{definition}\label{def_ahlfors}
A subequation $F\subset J^2(X)$ is said to satisfy the \emph{Ahlfors property} if, having set $H=F \cup \{r \le 0\}$, for each $U \subset X$ open with non-empty boundary and for each $u \in H(\overline{U})$ bounded from above and positive somewhere, it holds
$$
\sup_{\partial U} u^+ \equiv \sup_{\overline{U}} u.
$$
\end{definition}

\begin{remark}
\emph{Note that $u \in H(\overline{U})$ means, loosely speaking, that $u$ is $F$-subharmonic on the set $\{u>0\}$. The use of $H$ is demanded since $\{u>0\}$ is not granted to be open when $u \in \USC(\overline{U})$. We will see that this generality is important for our results to hold.
}
\end{remark} 
\begin{remark}
\emph{A related type of Ahlfors property also appeared in the very recent \cite{berestyckirossi}, and we refer to \cite{mmr_berestycki} for its relationship with results in \cite{aliasmastroliarigoli}.
}
\end{remark}
It is easy to see that the Ahlfors property implies the Liouville  one, and that the two are equivalent whenever $u \equiv 0 \in \widetilde{F}(X)$ (Proposition \ref{teo_main_A_L} below). Their equivalence has been shown in \cite{ahlforssario} for $\{\tr(A) \ge 0\}$, and later extended to $\{\tr(A) \ge \lambda r\}$ in \cite{grigoryan}\footnote{In \cite[Thms. 5.1 and 6.2]{grigoryan},  the Ahlfors property for $F=\{\tr(A) \ge \lambda r\}$ ($\lambda \ge 0$) is stated in a slightly different way as the absence of $\lambda$-massive sets, but this is equivalent to the definition here because $u \equiv 0$ is $F$-harmonic.}, and for more general subequations in \cite{aliasmastroliarigoli, aliasmirandarigoli, imperapigolasetti}. However, we stress that the two properties might be non-equivalent for the classes of equations considered here. 

Conditions to guarantee the Ahlfors (Liouville) property for $\{\tr(A) \ge \lambda r\}$ have been investigated by various authors (see \cite{grigoryan}), and in particular we focus on the next criterion due to R.Z. Khas'minskii \cite{khasminski, prsmemoirs}, hereafter called the Khas'minskii test: for the properties to hold for $C^2$ functions, it is sufficient that $X$ supports an exhaustion $w$ outside a compact set $K$ that satisfies
\begin{equation}\label{khasmi_original}
\begin{aligned}
&0 < w \in C^2(X \backslash K), \qquad w(x) \ra +\infty \ \text{ as $x$ diverges,} \\
&\Delta w \le \lambda w \ \text{ on } \, X \backslash K.
\end{aligned}
\end{equation}
Here, as usual, with ``$w(x) \ra +\infty$ as $x$ diverges" we mean that the sublevels of $w$ have compact closure in $X$. As a consequence of works by Z. Kuramochi and M. Nakai \cite{kuramochi, nakai}, in the parabolic case $\lambda=0$ the Khas'minskii test is actually \emph{equivalent} to the Liouville  property. However, the proof in \cite{nakai} uses in a crucial way a number of tools from potential theory that are specific for the subequation $\{ \tr(A) \ge 0\}$, and only recently a new approach allowed to extend the equivalence to the case $\lambda>0$ and to other classes of operators (\cite{valto_reverse, marivaltorta}).\par
In this paper, we prove a duality between Ahlfors and Khas'minskii properties that applies to a broad class of subequations $F$, including those listed at page \pageref{esempi_intro}. Our motivation comes from the fact that partial forms of the duality already  appeared not only in potential theory and stochastic processes, but also when investigating maximum principles at infinity (or ``almost maximum principles"), a fundamental tool in Riemannian Geometry. We will see how the duality here gives new insights on such principles, and answers some open questions. Before stating an ``informal version" of our main Theorems \ref{teo_main} and \ref{teo_main_withgradient}, we define the Khas'minskii property: hereafter, a pair $(K,h)$ consists of 
\begin{itemize}
\item[-] a smooth, relatively compact open set $K \subset X$;
\vspace{0.1cm} 
\item[-] a function $h \in C(X\backslash K)$ satisfying $h < 0$ on $X\backslash K$ and $h(x) \ra -\infty$ as $x$ diverges. 
\end{itemize}
\begin{definition}\label{def_khasmi} A subequation $F\subset J^2(X)$ satisfies the \emph{Khas'minskii property} if, for each pair $(K,h)$, there exists a function $w$ satisfying:
\begin{equation}
\begin{array}{l}
\disp w \in F(X\backslash K), \qquad h \le w \leq 0 \quad \text{ on } \, X \backslash 
K; \\[0.2cm]
w(x) \ra -\infty \quad \text{ as $x$ diverges.}
\end{array}
\end{equation}
Such a function $w$ is called a \emph{Khas'minskii potential} for $(K,h)$.
\end{definition}
Loosely speaking, the Khas'minskii property is the possibility of constructing negative, $F$-subharmonic exhaustions that decay to $-\infty$ as slow as we wish. Note that the sign of $w$ in \eqref{def_khasmi} is opposite to that in \eqref{khasmi_original}, and that $w$ is, a priori, just upper-semicontinuous. In applications, the Khas'minskii property is quite a difficult condition to check and, often, one can just ensure a relaxed version of it. This motivates the next 
\begin{definition}\label{def_weakkhasmi}
A subequation $F\subset J^2(X)$ satisfies the \emph{weak Khas'minskii property} if there exist a relatively compact, smooth open set $K$ and a constant $C \in \R \cup \{+\infty\}$ such that, for each $x_0 \not \in \overline{K}$ and each $\eps>0$, there exists $w$ satisfying 
\begin{equation}\label{eq_WK}
\begin{aligned}
&\disp w \in F(X\backslash K), \qquad w \leq 0 \quad \text{ on } \, X \backslash K, \\
&w(x_0) \ge -\eps, \qquad \disp \limsup_{x \ra \infty} w(x) \le -C. 
\end{aligned}
\end{equation}
We call such a $w$ a \emph{weak Khas'minskii potential} for the triple $(\varepsilon,K,\{x_0\})$.
\end{definition} 
\begin{remark}
\emph{If $C= +\infty$ and $F$ is fiber-wise a (possibly truncated) cone, that is, $tJ_x \in F_x$ whenever $J_x \in F_x$ and $t \in (0,1]$, then there is no need of the condition with $\eps$ in \eqref{eq_WK}: indeed, if $w$ satisfies all the other conditions in \eqref{eq_WK}, its rescaling $\delta w \in F(X\backslash K)$, for $\delta \in (0,1]$ small enough, satisfies $\delta w(x_0) \ge -\eps$. This is the case, for instance, of \eqref{khasmi_original}, which is therefore a weak Khas'minkii property. The weak Khas'minskii property was first considered, for quasilinear operators, in \cite{prsnonlinear, prs_milan, marivaltorta} (when $C=+\infty$) while the idea of considering potentials with finite $C$ first appeared in the recent \cite{mmr_berestycki}.
}
\end{remark} 
 
Informally speaking, our main result can be stated as follows. For a precise formulation, we refer the reader to Theorems \ref{teo_main} and \ref{teo_main_withgradient}.

\begin{theorem}\label{teo_informal}
Let $F \subset J^ 2(X)$ be a ``good" subequation. Then, 
\begin{equation*}
\begin{array}{c}
\text{$F$ has the} \\[0.1cm]
\text{Khas'minskii property} 
\end{array}
 \Longleftrightarrow \begin{array}{c}
\text{$F$ has the weak} \\[0.1cm]
\text{Khas'minskii property} 
\end{array}
 \Longleftrightarrow \begin{array}{c}
\text{$\widetilde F$ has the} \\[0.1cm]
\text{Ahlfors property} 
 \end{array}.
\end{equation*}
\end{theorem}
Hereafter, when the above equivalences are satisfied, we will shortly say that \emph{AK-duality holds for $F$}.\par
An important feature is that Theorems \ref{teo_main} and \ref{teo_main_withgradient} apply to subequations $F$ which are just \emph{locally jet-equivalent} to most of the examples listed at page \pageref{esempi_intro}. Roughly speaking, a jet-equivalence can be thought as a way to plug non-constant coefficients in the subequation $F$ without, in general, affecting its properties, and considerably enlarges the range of applicability of our results (see the next sections for definitions). Moreover, Theorem \ref{teo_informal} is flexible enough to apply to good subequations which are unions-intersections of other good ones. As a relevant example for our applications, we can intersect eikonal subequations with those in $(\EE 2), \ldots, (\EE 8)$. To state the result, we first consider $f,\xi$ satisfying
\begin{equation}\label{def_f1xi1}
\begin{array}{llll}
(f1) & \ f \in C(\R), & \ f \ \text{is non-decreasing,} & \ f(0)=0, \ f(r)<0 \quad \text{ for } \, r<0; \\[0.2cm]
(\xi 1) & \ \xi \in C(\R), & \ \xi \, \text{ is non-increasing,} & \ \xi(0)=0, \ \xi(r) > 0 \quad \text{ for } \, r<0.
\end{array}
\end{equation}

We write $E_\xi$ to indicate the eikonal type subequation 
\begin{equation}\label{def_Exi}
E_\xi = \overline{\big\{|p| < \xi(r)\big\}}, \qquad \text{whose dual is} \qquad \widetilde{E_\xi} = \overline{\big\{|p| > \xi(-r)\big\}}.
\end{equation}
Then, we have\footnote{Observe that, because of $\xi(0)=0$ in $(\xi 1)$, no jets with $r>0$ belong to $E_\xi$. This is not a problem, since $E_\xi$ appears related to Khas'minskii potentials, which are non-positive, while $\widetilde{E_\xi}$ is linked to the Ahlfors property, which considers upper-level sets above zero.}

\begin{theorem}\label{cor_bonito}
Fix $f$ satisfying $(f1)$ and $\xi$ satisfying $(\xi 1)$. \vspace{0.1cm}
\begin{itemize}
\item[$i)$] If $F_f \subset J^2(X)$ is locally jet-equivalent to one of the examples in $(\EE 2), \ldots, (\EE 6)$ via locally Lipschitz bundle maps, then AK-duality holds for $F_f$ and for $F_f \cap E_\xi$. \vspace{0.1cm}
\item[$ii)$] If $F_f$ is the universal, quasilinear subequation in $(\EE 7)$, with eigenvalues $\{\lambda_j(t)\}$ in \eqref{ipo_a}, then AK-duality holds \vspace{0.1cm}
\begin{itemize}
\item[-] for $F_f$, provided that $\lambda_j(t) \in L^\infty(\R_0^+)$ for $j \in \{1,2\}$; \vspace{0.1cm}
\item[-] for $F_f \cap E_\xi$, provided that $\lambda_j(t) \in L^\infty_{\mathrm{loc}}(\R_0^+)$ for $j \in \{1,2\}$. \vspace{0.1cm}
\end{itemize}
\item[$iii)$] If $F_f$ is the universal subequation in $(\EE 8)$, then AK-duality holds for $F_f$ and for $F_f \cap E_\xi$. \vspace{0.1cm}
\end{itemize}
Furthermore, in each of $i), \ldots, iii)$, the Ahlfors property holds for $\widetilde{F_f}$ (respectively, for $\widetilde{F_f} \cup \widetilde{E_\xi}$) for some $f$ enjoying $(f1)$ (respectively, $(f,\xi)$ enjoying $(f1+\xi 1)$) if and only if it holds for all such $f$ (resp. $(f,\xi)$).
\end{theorem}

The independence on $(f,\xi)$ in the second part of the theorem will appear a number of times in what follows, and will be proved in Section \ref{sec_dependenceonf}. For instance, a prototype case of $f$ satisfying $(f1)$ is $f(r)= \lambda r$ with $\lambda>0$, as in the subequation $F = \widetilde{F} = \{\tr(A)\ge \lambda r\}$ related to the stochastic completeness of $X$, for which the value $\lambda$ plays no role.\par

\begin{remark}
\emph{We emphasize that $ii)$ includes the case of the mean curvature operator. More generally, the technical restrictions appearing in $(\EE 7)$ are necessary to apply the comparison results for viscosity solutions on manifolds that are currently available in the literature, and should be removable. This is the case, for instance, if $X$ has non-negative sectional curvature, for which the assumptions on $\lambda_j(t)$ in $ii)$ are not needed (notably, this  includes the $k$-Laplacian for $1 \le k < +\infty$). Other cases when AK-duality applies for quasilinear equations are discussed in Section \ref{sec_main_examples} and Appendix \ref{appendix_1}.
}
\end{remark}

We warn the reader that AK-duality for the case when $f$ satisfies the complementary assumption 
\begin{equation*}
\begin{array}{ll}
	(f1')  \quad &f \in C(\R), \quad f \ \text{is non-decreasing,} \\
& f=0 \ \text{ on some interval } (-\mu,0), \ \mu>0,
\end{array}
\end{equation*}
(typically, when $f\equiv 0$), depends in a more delicate way on $F$ and its properties, such as the validity of strong maximum principles in finite form. Note that $(f1')$ includes the example $\{\tr(A) \ge 0\}$ related to the parabolicity of $X$. We will address this case in Theorem \ref{teo_main_semH1}.\par
\vspace{0.2cm}

The equivalence between the Ahlfors-Liouville and the Khas'minskii property was previously known just for few special operators. In particular, 
\begin{itemize}
\item[-] Kuramochi-Nakai \cite{kuramochi,nakai} proved a similar duality for solutions of $\Delta u \ge 0$, where the Khas'minskii potentials are harmonic but do not necessarily satisfy $w \ge h$;
\item[-] Valtorta \cite{valto_reverse}, with a different approach, showed AK-duality for the $k$-Laplacian $\Delta_k u \ge 0$, $k \in (1, +\infty)$; 
\item[-] The first author and Valtorta \cite{marivaltorta}, with still another method (though inspired by \cite{valto_reverse}), settled the case of operators of the type
$$
\mathrm{div}\big(\mathcal{A}(x,u,\nabla u)\big) \ge \mathcal{B}(x,u),  
$$
where $\mathcal{A}$ behaves like a $k$-Laplacian, $k \in (1,+\infty)$.
\end{itemize}
%

\vspace{0.2cm}

As a first application of AK-duality, in Section \ref{sec_immesub} we investigate the following problem: let $\sigma : X^m \ra Y^n$ be either an isometric immersion or a Riemannian submersion, and let $F$ be a universal subequation. Suppose that $\widetilde{F}$, or $\widetilde{F} \cup \widetilde{E}$, has the Ahlfors property on $Y$. Under which condition $X$ inherits the Ahlfors property for $\widetilde{F}$, $\widetilde{F} \cup \widetilde{E}$ or some related subequations? We decide to focus on $F$ as in Examples $(\EE 3)$, $(\EE 5)$ or their complex counterparts, because of their relevance in the study of the geometry of submanifolds. Our outcome are Theorems \ref{teo_immersions} and \ref{teo_submersions} below, which answer the question under some reasonable conditions on the map $\sigma$.\par
Other applications to illustrate the usefulness of duality come from the following three examples. We begin with the classical maximum principle, stating that a function $u \in C^2(X)$ satisfies, at a maximum point $x_0$,  
\begin{equation}\label{finitemaximum}
u(x_0) = \max_X u, \quad |\nabla u(x_0)| =0, \quad \nabla^2 u(x_0) \le 0 \quad \text{ as a quadratic form.}
\end{equation}
When $X$ is non-compact and $u$ is bounded from above, the impossibility to guarantee the existence of such $x_0$ forced mathematicians to consider weakened versions of \eqref{finitemaximum}. In particular, for applications it is important to find a sequence $\{x_k\} \subset X$ for which $u(x_k) \ra \sup_X u$ and some (if possible, all) of the conditions in \eqref{finitemaximum} hold in a limit sense. The search for conditions to guarantee such maximum principles at infinity has stimulated an intense research in the past fifty years.

\subsection*{The Ekeland maximum principle, $\infty$-parabolicity and\\ geodesic completeness}
In a celebrated paper \cite{ekeland_2, ekeland}, I. Ekeland proved that any complete metric space satisfies the following property.
\begin{definition}\label{def_ekeland_classico}
A metric space $(X,\di)$ satisfies the \emph{Ekeland maximum principle} if, for each $u \in \USC(X)$ bounded from above, there exists a sequence $\{x_k\} \subset X$ with the following properties:
$$
u(x_k) > \sup_X u - \frac{1}{k}, \qquad u(y) \le u(x_k) + \frac{1}{k}\di(x_k,y) \ \text{ for each } \, y \in X.
$$
\end{definition}
Ekeland principle has a broad spectrum of applications (just to quote some of them, see \cite{ekeland_2}). It has been realized by J.D. Weston \cite{weston} and F. Sullivan \cite{sullivan} that the principle is, indeed, \emph{equivalent} to $X$ being complete. If we consider the dual eikonal equation $\widetilde{E} = \{|p| \ge 1\}$, and we compare Definitions \ref{def_ahlfors} and  \ref{def_ekeland_classico}, we can view the Ahlfors property for $\widetilde{E}$ as a sort of \emph{viscosity version of Ekeland principle.} A priori, its equivalence with the Ekeland principle in Definition \ref{def_ekeland_classico} doesn't seem obvious to us, and indeed it is an application of the above duality. Furthermore, another property turns out to be equivalent to $X$ being complete: the Ahlfors property for the $\infty$-Laplacian. Summarizing, we have the following theorem.

\begin{theorem}\label{teo_ekeland_intro}
Let $X$ be a Riemannian manifold. Then, the following statements are equivalent:
\begin{itemize}
\item[(1)] $X$ is complete.
\item[(2)] the dual eikonal $\widetilde E= \{|p|\ge 1\}$ has the Ahlfors property (viscosity, Ekeland principle).
\item[(3)] the eikonal $E=\{|p| \le 1\}$ has the Khas'minskii property.
\item[(4)] the infinity Laplacian $F_\infty \doteq \overline{\{A(p,p)>0\}}$ has the Ahlfors property.
\item[(5)] $F_\infty$ has the Liouville property.
\item[(6)] $F_\infty$ has the Khas'minskii property.
\item[(7)] $F_\infty$ has the next strengthened Liouville  property:\vspace{1ex}
\begin{quote}
Any $F_\infty$-subharmonic function $u \ge 0$ such that $|u(x)| = o\big(\varrho(x)\big)$ as $x$ diverges ($\varrho(x)$ the distance from a fixed origin) is constant.
\end{quote}
\end{itemize}
\end{theorem}

Regarding previous results on the Liouville property in $(5)$, we quote \cite{lindqvistmanfredi} where the authors proved that positive, $F_\infty$-harmonic functions $u$ on $\R^m$ are constant via a Harnack inequality. Observe that, up to reflecting, translating and taking the positive part of $u$, this is a weaker version of $(5)$. The result has been extended in \cite{CEG} to the full Liouville property for $F_\infty$ on the Euclidean space\footnote{The result in \cite{CEG} is a corollary of a more general Bernstein-type theorem, stating that any $u \in F_\infty(\R^m)$ below some affine function on $\R^m$ is necessarily affine.}. However, to the best of our knowledge the first result that relates the completeness of $X$ to the $\infty$-Laplacian is the following, recent one in \cite{pigolasetti_ensaio}; there, the authors proved that $X$ is complete if and only if, for each compact set $K$,
$$
\mathrm{cap}_\infty(K) \doteq \inf\Big\{ \|\di u\|_\infty, \ \ : \ \ u \in \lip_c(X), \ u=1 \ \text{ on } \, K\Big\}
$$
is zero. This last property is named $\infty$-parabolicity, because of its direct link with the standard definition of $k$-parabolicity via capacity (see \cite{holopainen, troyanov2}). As in the $k<+\infty$ case, the interplay with the Liouville property for $F_\infty$ depends on the characterization of solutions of
$$
\begin{cases}
\text{$u$ is $F_\infty$-harmonic on $\Omega \backslash K$,} \\
u=1 \ \text{ on } \, \partial K,  \qquad u=0 \quad \text{on } \, \partial \Omega,
\end{cases}
$$
(here, $\Omega$ open, $K \subset \Omega$ compact) as suitable minimizers for the variational problem
\begin{equation*}\label{inftylapla}
\inf_{\mathcal{L}(K,\Omega)} \|\di u\|_{L^\infty}, \quad \mathcal{L}(K,\Omega) \doteq \big\{ u \in \lip(\Omega) \ : \ u=1 \ \text{ on } \, K, \ u=0 \ \text{ on } \, \partial \Omega\big\},
\end{equation*}
a parallel first considered by G. Aronsson \cite{aronsson} and proved by R. Jensen \cite{jensen} when $X=\R^m$.

\subsection*{The Omori-Yau maximum principles}

The Omori-Yau maximum principles have been introduced by H. Omori \cite{omori} and S.T. Yau \cite{chengyau, yau}: they stated, respectively, the next criteria for the Hessian and the Laplacian.

\begin{definition}\label{def_hessianLaplacian_classico}
Let $(X, \metric)$ be a Riemannian manifold. We say that $X$ satisfies the \emph{strong Hessian (respectively, Laplacian) maximum principle} if, for each $u \in C^2(X)$ bounded from above, there exists a sequence $\{x_k\} \subset X$ with the following properties:
\begin{equation}\label{strong_HessianLaplacian}
\begin{array}{ll}
(\text{Hessian}) & \quad u(x_k) > \sup_X u - k^{-1}, \quad |\nabla u(x_k)| < k^{-1}, \quad \nabla^2 u(x_k) \le k^{-1} \metric ;\\[0.2cm]
(\text{Laplacian}) & \quad u(x_k) > \sup_X u - k^{-1}, \quad |\nabla u(x_k)| < k^{-1}, \quad \Delta u(x_k) \le k^{-1}.
\end{array}
\end{equation}
\end{definition}
%
Here, the word ``strong" refers to the presence of the gradient condition $|\nabla u(x_k)|< k^{-1}$. Note that the first two in \eqref{strong_HessianLaplacian} hold, for suitable $\{x_k\}$, on each complete manifolds in view of Ekeland principle. However, differently from Ekeland's one, the Hessian and Laplacian principles depend in a subtle way on the geometry of a complete $X$, cf. \cite{omori, yau, chenxin, rattorigolisetti}. A systematic study has been undertaken in \cite{prsmemoirs} and \cite{aliasmastroliarigoli}, where the authors described a general function-theoretic criterion (\cite[Thm. 2.4]{aliasmastroliarigoli}, see also \cite{borbely, bessalimapessoa}): the strong Laplacian maximum principle holds provided that $X$ supports a function $w$ with the following properties:
\begin{equation}\label{strongKhasminskii}
\begin{aligned}
&0 < w \in C^2(X\backslash K) \ \text{ for some compact $K$}, \quad w \ra + \infty \ \text{ as } \, x \ \text{ diverges}, \\[0.2cm]
&|\nabla w| \le G(w), \qquad \Delta w \le G(w),
\end{aligned}\hspace{-1em}
\end{equation} 
for some $G$ satisfying
\begin{equation}\label{ipo_G_SMP}
0< G \in C^1(\R^+), \qquad G' \ge 0, \qquad G^{-1} \not \in L^1(+\infty). 
\end{equation}
For the Hessian principle, the last condition in \eqref{strongKhasminskii} has to be replaced by $\nabla^2 w \le G(w) \metric$. The criterion is effective, since the function in \eqref{strongKhasminskii} can be explicitly found in a number of geometrically relevant applications, see \cite{prsmemoirs, aliasmastroliarigoli}: for instance, letting $\varrho(x)$ denote the distance from a fixed origin $o$, when $X$ is complete \eqref{strongKhasminskii} holds if the Ricci curvature satisfies
\begin{equation}\label{condi_ricciPRS}
\Ricc_x(\nabla \varrho, \nabla \varrho) \ge - G^2\big(\varrho(x)\big) \quad \text{ outside of } \, \cut(o).
\end{equation}
However, $w$ might be independent of the curvatures of $X$, and actually there are cases when \eqref{strongKhasminskii} is met but the sectional curvature of $X$ goes very fast to $-\infty$ along some sequence \cite[p. 13]{prsmemoirs}. For the Hessian principle, \eqref{condi_ricciPRS} has to be replaced by an analogous decay for all of the sectional curvatures of $X$:
\begin{equation}\label{condi_sectPRS}
 \mathrm{Sect}(\pi_x) \ge - G^2\big(\varrho(x)\big) \quad \begin{array}{l}
\forall \, x \not\in \cut(o), \\[0.1cm]
\forall \, \pi_x \le T_x X \ \text{ $2$-plane containing $\nabla \rho$.}
\end{array}
\end{equation}
Note that \eqref{ipo_G_SMP} includes the case when $G(t)$ is constant, considered in \cite{omori, yau}. With simple manipulation, we can extract from \eqref{strongKhasminskii} a weak Khas'minskii property analogous to \eqref{khasmi_original}: first, we observe (as in \cite{kimlee}) that up to replacing $w$ with 
$$
\int_0^w \frac{\di s}{G(s)},
$$
without loss of generality we can choose $G(w) \equiv 1$ in \eqref{strongKhasminskii}; next, since $w$ is an exhaustion, for $\lambda>0$ fixed it holds $\Delta w \le 1 \le \lambda w$ on $X\backslash K$, if $K$ is large enough. In other words, from \eqref{strongKhasminskii} we can produce a new $w$ satisfying $|\nabla w| \le 1$, $\Delta w \le \lambda w$ outside some compact set. Reflecting and rescaling $w$, we obtain weak Khas'minskii potentials as in Definition \ref{def_khasmi} for the subequation 
\begin{equation}\label{strongkhasm}
\big\{ \tr(A) \ge \lambda r\big\} \cap \big\{|p| \le 1\big\}, \qquad \text{with} \ \lambda>0.
\end{equation}
Observe that the presence of a divergent $w$ with bounded gradient guarantees that $X$ is complete, in other words, \eqref{strongKhasminskii} implies both the completeness of $X$ and the strong Laplacian maximum principle. However, whether or not the Hessian-Laplacian principles imply a corresponding Khas'minskii conditions, or at least the geodesic completeness of $X$, was unknown. To investigate the interplay between the three properties, we need to define an appropriate viscosity version of the Omori-Yau principles. First, we rephrase Definition~\ref{def_hessianLaplacian_classico} as follow, say in the Hessian case: $X$ satisfies the Hessian principle if, for each $\eps>0$, there exists no $u \in C^2(X)$ bounded above and such that, for some $\gamma < \sup_X u$, at all points $x \in \{u> \gamma\}$ the $2$-jet $J^2_xu$ belongs to either
\begin{align}\label{class_stronglapla}
\big\{ |p| \ge \eps \big\}, \qquad \text{or} \qquad \big\{ \lambda_m(A) \ge \eps \big\}.\ \text{\footnotemark}
\end{align}
\footnotetext{As before, $\{\lambda_j(A)\}$ is the sequence of eigenvalues of $A$ in increasing order.}
Because of the translation and rescaling properties of the subequations in \eqref{class_stronglapla}, to check the principle we can take $\eps=1$ and $\gamma=0$ without loss of generality. In view of Definition \ref{def_ahlfors}, we propose the following

\begin{definition}\label{def_OmoriYau_visco} Let $X$ be a Riemannian manifold.\vspace{0.1cm} 
\item[$\qquad$] Consider the subequation $F \!=\! \{\lambda_1(A) \!\ge\! -1\}$, whose dual is $\widetilde{F}\!=\! \{\lambda_m(A) \!\ge\! 1\}$. We say that $X$ satisfies the \emph{viscosity, strong Hessian principle} if the Ahlfors property holds for $\widetilde{F} \cup \widetilde{E}$.\vspace{0.1cm}
\item[$\qquad$] Consider the subequation $F= \{\tr(A) \!\ge\! -1\}$, whose dual is $\widetilde{F}= \{\tr(A) \ge 1\}$. 
We say that $X$ satisfies the \emph{viscosity, strong Laplacian principle} if the Ahlfors property holds for $\widetilde{F} \cup \widetilde{E}$.
\end{definition}
%

Clearly, the viscosity, Hessian (Laplacian) principle imply the corresponding $C^2$ one in \eqref{strong_HessianLaplacian}. However, it is worth to stress the following point:
\begin{quote}
\emph{the Omori-Yau principles for $C^2$-functions are granted, in all the examples that we are aware of, under \eqref{strongKhasminskii} or its Hessian counterpart, that also imply the validity of the viscosity ones (see Theorems \ref{teo_Laplacian_intro}, \ref{teo_hessianmax_intro} and Proposition \ref{cor_sufficientiSMP} below); therefore, we can safely replace Definition \ref{def_hessianLaplacian_classico} with the stronger Definition~\ref{def_OmoriYau_visco} without affecting the range of applicability of the principles.} 
\end{quote}
\begin{remark}
\emph{The possible equivalence between the classical and viscosity versions of the Omori-Yau principles is discussed in the subsection ``open problems" below. There, we also comment on an alternative definition of a viscosity, strong Hessian and Laplacian principles recently appearing in \cite{pengzhou}. 
}
\end{remark}
Passing to viscosity solutions allows us to clarify the interplay between Hessian-Laplacian principle, geodesic completeness and the Khas'minskii property. Firstly, since the Ahlfors property for $\widetilde{F}\cup \widetilde{E}$ implies that for $\widetilde{E}$, Theorem \ref{teo_ekeland_intro} has the next
\begin{corollary}\label{cor_SMPecomplete}
Any manifold $X$ satisfying the viscosity, strong Laplacian principle must be complete.
\end{corollary}
Secondly, as a particular case of Theorem \ref{cor_bonito}, there is equivalence with the Khas'minskii property. We recall that properties $(f1),(\xi 1)$ are defined in \eqref{def_f1xi1}, and that 
$$
\begin{array}{ll}
E_\xi \doteq \overline{\big\{ |p| < \xi(r)\big\}}, & \qquad \text{has dual} \qquad \widetilde{E_\xi} = \overline{\big\{ |p| > \xi(-r)\big\}}; \\[0.2cm]
F_f \doteq \big\{ \tr(A) \ge f(r)\big\}, & \qquad \text{has dual} \qquad \widetilde{F_f} = \big\{ \tr(A) \ge -f(-r)\big\}.
\end{array}
$$

\begin{theorem}\label{teo_Laplacian_intro}
The following properties are equivalent:
\begin{itemize}
\item[(1)] $X$ has the viscosity, strong Laplacian principle;\vspace{0.1cm}
\item[(2)] $\widetilde{F_f}\cup \widetilde{E_\xi}$ has the Ahlfors property for some (each) $(f,\xi)$ satisfying  $(f1+\xi 1)$;\vspace{0.1cm}
\item[(3)] $F_f \cap E_\xi$ has the Khas'minskii property for some (each) $(f,\xi)$ satisfying  $(f1+\xi 1)$;\vspace{0.1cm}
\item[(4)] $F_f \cap E_\xi$ has the weak Khas'minskii property for some (each) $(f,\xi)$ satisfying  $(f1+\xi 1)$.
\end{itemize}
\end{theorem}


\begin{remark}
\emph{Via jet-equivalence, Theorem \ref{teo_Laplacian_intro} also holds if we replace the prototype subequation $F = \{ \tr(A) \ge -1\}$ with the more general
$$
F_L = \big\{ (x,r,p,A) \ : \  \tr\big(T(x)A\big) + \langle W(x), p \rangle \ge -b(x)\big\}, 
$$
describing solutions of a general, linear elliptic inequality $Lu \ge -b(x)$ in non-divergence form with locally Lipschitz sections
$$
0 < T \ : \ X \ra \Sym^2(T^*X), \qquad W \ : \ X \ra T^*X, \qquad 0 < b \in C(X).
$$
See Example \ref{ex_linear} for details. In this way, we complement and extend recent results in \cite{albanesealiasrigoli, aliasmastroliarigoli, bessapessoa}, where the implication Khas'minskii $\Rightarrow$ strong Laplacian principle is investigated.
}
\end{remark}

As a direct corollary of Theorem \ref{teo_Laplacian_intro}, and of the corresponding Theorem~\ref{teo_hessianmax_intro} for the strong Hessian principle that will be discussed later, the mild sufficient conditions for the validity of the Omori-Yau principles described in Examples 1.13 and 1.14 of \cite{prsmemoirs} are still enough to guarantee their viscosity versions. The following result, a refined  version of these examples, will be proved in Section \ref{sec_immesub}. We premit the following 

\begin{definition}\label{def_medioricci}
Let $X^m$ be a complete $m$-dimensional manifold, and let $k \in \{1,\ldots, m-1\}$. The $k$-th (normalized) Ricci curvature is the function
$$
\begin{array}{ccccl}
\Ricc^{(k)} & : &  TX &  \longrightarrow  &  \R \\
&& v & \longmapsto & \disp \inf_{\begin{array}{c}
\mathcal{W}_k \le v^\perp \\
\dim \mathcal{W}_k = k
\end{array}} \left( \frac{1}{k} \sum_{j=1}^k \mathrm{Sect}(v \wedge e_j)\right), 
\end{array}
$$
where $\{e_j\}$ is an orthonormal basis of $\mathcal W_k$.
\end{definition}

Bounding the $k$-th Ricci curvature from below is an intermediate condition between requiring a lower bound on the sectional and on the Ricci curvatures, and indeed
$$
\begin{array}{l}
\mathrm{Sect} \ge - G^2(\rho) \quad  \Longrightarrow \quad \Ricc^{(k-1)} \ge - G^2(\rho)\|\cdot \| \quad \Longrightarrow  \\[0.2cm]
\disp \Longrightarrow \quad \Ricc^{(k)} \ge - G^2(\rho)\|\cdot \| \quad  \Longrightarrow \quad \Ricc \ge - G^2(\rho)\metric.
\end{array}
$$
The first and last implications are equivalences provided that, respectively, $k = 2$ and $k=m-1$ (here, $\Ricc$ is the normalized Ricci tensor).

\begin{proposition}\label{cor_sufficientiSMP}
Let $X^m$ be a complete manifold, $m \ge 2$. Fix $o \in X$ and let $\rho(x) = \mathrm{dist}(x,o)$. Then, 
\begin{itemize}
\item[(i)] The viscosity, strong Hessian principle holds provided that the sectional curvature of $X$ enjoys \eqref{condi_sectPRS}, for some $G$ satisfying \eqref{ipo_G_SMP}.
\item[(ii)] The viscosity, strong Laplacian principle holds provided that the Ricci curvature of $X$ enjoys \eqref{condi_ricciPRS}, for some $G$ satisfying \eqref{ipo_G_SMP}.
\item[(iii)] The viscosity, strong Laplacian principle holds on $X$ if $X$ is isometrically immersed via $\sigma : X \ra Y$ into a complete manifold $Y$ satisfying 
\begin{equation}\label{assu_medioricci}
\Ricc_x^{(m-1)}(\nabla \rho) \ge -G^2\big( \bar \rho(x)\big) \qquad \text{for each } \, x \not \in \cut(\bar o),
\end{equation}
where $\bar \rho$ is the distance in $Y$ from some fixed origin $\bar o$, $G$ satisfies \eqref{ipo_G_SMP}, and the mean curvature vector $H$ is bounded by
$$
|H(x)| \le C G\big( \bar \rho(\sigma(x)\big),
$$ 
for some constant $C>0$. 
\end{itemize}
\end{proposition}

Proposition \ref{cor_sufficientiSMP} is indeed a special case of Theorem \ref{prop_sufficientiFk}, that considers the subequation $F$ in $(\EE 5)$ for each $k \le m$. We refer the reader to Section \ref{sec_immesub} for further insight.

\subsection*{Weak maximum principles, stochastic and\\ martingale completeness}
The weak Hessian and Laplacian principles have been introduced by S.~Pigola, M. Rigoli and A.G. Setti in \cite{prs_proceeding, prsmemoirs}, starting from the observation that the gradient condition in \eqref{strong_HessianLaplacian} is unnecessary in many geometric applications. By definition, $X$ is said to satisfy the \emph{weak Hessian (respectively, Laplacian) principle} if, for each $u \in C^2(X)$ bounded above, there exists $\{x_k\}$ such that 
\begin{equation}\label{weak_HessianLaplacian}
\begin{array}{ll}
\text{(Hessian)} & \quad u(x_k) > \sup_X u - k^{-1}, \qquad \nabla^2 u(x_k) \le k^{-1} \metric ;\\[0.2cm]
\text{(Laplacian)} & \quad u(x_k) > \sup_X u - k^{-1}, \qquad \Delta u(x_k) \le k^{-1}.
\end{array}
\end{equation}
In the same works the authors prove, in a $C^2$ setting, that the weak Laplacian principle is equivalent to the Ahlfors property for $\big\{ \mathrm{Tr}(A) \ge \lambda r\big\}$ for some (each) $\lambda>0$, and therefore its validity is guaranteed under condition \eqref{khasmi_original}. Our first remark is that the weak and strong Laplacian principles are \emph{not} equivalent, as the following examples show.

\begin{example}[\textbf{weak Laplacian $\neq$ strong Laplacian}]\label{laplacian_different}
It is easy to produce an example of \emph{incomplete} manifold $X$ satisfying the weak Laplacian principle but not the strong one: for example, $X=\R^m\backslash \{0\}$ has the weak Laplacian principle because the function $w = -|x|^2 - |x|^{2-m}$ (for $m \ge 3$) or $-|x|^2+\log|x|$ (for $m=2$) is a weak Khas'minskii potential satisfying \eqref{khasmi_original}; on the other hand, $X$ does not satisfy the strong Laplacian principle, since any sequence tending to the supremum of $u(x) = e^{-|x|}$ cannot satisfy the gradient condition in \eqref{strong_HessianLaplacian}.
A very nice example of a \emph{complete}, radially symmetric surface satisfying the weak Laplacian principle but not the strong one has recently been found in \cite{borbely_counter}.
\end{example}
Passing to weak principles has some advantages. A first point is that the weak Laplacian principle is equivalent to the stochastic completeness of $X$, that is, to the fact that paths of the Brownian motion on $X$ have infinite lifetime almost surely (\cite{prs_proceeding, prsmemoirs}). Consequently, one can avail of heat equation techniques to give a sharp criterion for its validity that does not depend on curvatures: by \cite[Thm. 9.1]{grigoryan}, a complete $X$ satisfies the weak Laplacian principle whenever
\begin{equation}\label{condi_volgrigor}
\frac{r}{\log\vol(B_r)} \not \in L^1(+\infty),
\end{equation}
$B_r$ being the geodesic ball of radius $r$ centered at some fixed origin. Secondly, the absence of a gradient condition allows a natural extension of Definition \ref{weak_HessianLaplacian} to distributional solutions and to general quasilinear operators $\Delta_a$ in place of $\Delta$, including most of those considered in Example $(\EE 7)$: by \cite{prs_gafa, prsmemoirs}, a quasilinear operator $\Delta_a$ is said to satisfy the weak maximum principle if, for each $u \in C^0(X)$ bounded above and in a suitable Sobolev class (typically, $u \in W^{1,\infty}_\mathrm{loc}(X)$), and for each $\gamma < \sup_X u$,  
\begin{equation}\label{def_weakmax}
\inf_{\{u>\gamma\}} \Delta_a u \le 0 \qquad \text{in a weak sense.}
\end{equation}
The key point here is that, via refined integral estimates, property \eqref{def_weakmax} holds under mild volume growth conditions of the type in \eqref{condi_volgrigor}, see \cite{karp, prs_gafa, prsmemoirs, rigolisalvatorivignati_3, maririgolisetti}. The equivalence between distributional and viscosity solutions for $\Delta_a u \ge f(u)$ has been investigated for some families of $a(t)$ in \cite{julinjuutinen, fangzhou} (see also \cite{HL_equivalence}). 
\begin{definition}
We say that $\Delta_a$ has the viscosity, weak (respectively, strong) maximum principle if the Ahlfors property holds for  
\begin{equation}
\begin{array}{lll}
& \quad \disp \overline{\big\{ |p|>0, \ \tr( T(p)A) > \eps\big\}} & \quad \text{for the weak principle in } \eqref{def_weakmax}, \\[0.2cm]
 & \disp \overline{\big\{ |p|>0, \ \tr( T(p)A) > \eps\big\}} \cup \big\{ |p| \ge \eps\big\}
 & \quad \text{for the strong principle,} 
\end{array}\label{quasilinear_wmp}
\end{equation}
for each $\eps>0$. 
\end{definition}
Since $\Delta_a$ is generally not homogeneous, we cannot consider just $\eps=1$. However, as in Theorem \ref{teo_Laplacian_intro}, by Propositions \ref{prop_equivalenceahlfors} and \ref{prop_equivalenceahlfors_withgradient} below the Ahlfors property for \eqref{quasilinear_wmp} can be checked just on a single subequation. Again, as a particular case of Theorem \ref{cor_bonito}, we have
\begin{proposition}
Consider $F_f$ in $(\EE 7)$, and $E_\xi$ in \eqref{def_Exi}. Then, $\Delta_a$ has the viscosity, weak (strong) maximum principle if and only if the Ahlfors property holds for $\widetilde{F_f}$ (resp. $\widetilde{F_f} \cup \widetilde{E_\xi}$), for some (equivalently, each) pair $(f,\xi)$ satisfying $(f1+\xi 1)$. Furthermore, AK-duality holds under the assumptions in $ii)$ of Theorem \ref{cor_bonito}.
\end{proposition}
We conclude by considering the weak Hessian principle. In analogy with Definition \ref{def_OmoriYau_visco}, we set
\begin{definition}
$X$ is said to satisfy the viscosity, weak Hessian principle if the Ahlfors property holds for $\widetilde{F} = \{ \lambda_m(A) \ge 1\}$. 
\end{definition}
Similarly to the case of the Laplacian, there seems to be a tight relation between the Hessian principles and the theory of stochastic processes. As suggested in \cite{prs_overview, prs_milan}, a good candidate to be a probabilistic counterpart of a Hessian principle is the \emph{martingale completeness of $X$}, that is, the property that each martingale on $X$ has infinite lifetime (see Section V in \cite{emery}). However, few is known about their interplay, and the picture seems different here. For instance, in striking contrast with the case of stochastic completeness, a martingale complete manifold must be geodesically complete (\cite[Prop. 5.36]{emery}). In \cite[Prop. 5.37]{emery}, by using probabilistic tools M. Emery proved that $X$ is martingale complete provided that there exists $w \in C^2(X)$ satisfying
\begin{equation}\label{khasmi_Hessian_original}
\begin{array}{l}
\disp 0 < w \in C^2(X), \qquad w(x) \ra +\infty \ \text{ as $x$ diverges,} \\[0.2cm]
|\nabla w| \le C, \qquad \nabla^2 w \le C \metric \ \text{ on } \, X,
\end{array}
\end{equation}
for some $C>0$. By rescaling and reflecting $w$, this is again a weak Khas'minskii property. Because of  \eqref{khasmi_Hessian_original}, one might guess that the martingale completeness of $X$ is likely to be related to the strong Hessian principle. However, in view of the equivalence in the case of the Laplacian, in \cite[Question 37]{prs_overview} and \cite{prs_milan} the authors ask whether the \emph{weak} Hessian principle implies the martingale completeness of $X$, or at least its geodesic completeness, and give some partial results. Our last contribution is an answer to the above questions: perhaps surprisingly, the viscosity, weak and strong Hessian principles are equivalent, and because of AK-duality (plus extra arguments) they imply that $X$ is martingale complete. Moreover, to check the viscosity, Hessian principle it is sufficient to consider semiconcave functions. We recall
\begin{definition}
A function $u : X \ra \R$ is semiconcave if, for each $x_0 \in X$, there exist a neighborhood $U$ of $x_0$ and $v\in C^2(U)$ such that $u+v$ is concave on $U$ (i.e. $u+v$ is a concave function when restricted to geodesics).
\end{definition}
Since semiconcave functions are locally Lipschitz and $2$-times differentiable a.e., $(2)$ in Theorem \ref{teo_hessianmax_intro} below is very close to the classical, $C^2$ weak Hessian principle in \eqref{weak_HessianLaplacian}. Summarizing, we have

\begin{theorem}\label{teo_hessianmax_intro}
Consider the subequation 
$$
F_f = \{ \lambda_1(A) \ge f(r)\}, 
$$
for some $f \in C(\R)$ non-decreasing. Then, 
\begin{itemize}
\item[-] AK-duality holds both for $F_f$ and $F_f\cap E_\xi$, for some (equivalently, each) $(f,\xi)$ satisfying $(f1+\xi 1)$. 
\end{itemize}
Moreover, the following properties are equivalent:
\begin{itemize}
\item[(1)] $X$ satisfies the viscosity, weak Hessian principle;
\item[(2)] $X$ satisfies the viscosity, weak Hessian principle for semiconcave functions;
\item[(3)] $\widetilde{F_f}$ has the Ahlfors property for some (each) $f$ of type $(f1)$;
\item[(4)] $X$ satisfies the viscosity, strong Hessian principle;
\item[(5)] $\widetilde{F_f} \cup \widetilde{E_\xi}$ has the Ahlfors property, for some (each) $(f,\xi)$ satisfying $(f1+\xi 1)$;
\item[(6)] $F_f \cap E_\xi$ has the Khas'minskii property with $C^\infty$ potentials, for some (each) $(f,\xi)$ satisfying $(f1 + \xi 1)$.
\end{itemize}
In particular, each of $(1), \ldots, (6)$ implies that $X$ is geodesically complete and martingale complete.
\end{theorem}

\subsection*{Classical vs. viscosity principles, and open problems}
We conclude this introduction by proposing some questions. First, because of Theorem \ref{teo_hessianmax_intro} we know that the viscosity, Hessian principle (weak or, equivalently, strong) implies the martingale completeness of $X$. It is therefore natural to investigate the reverse implication.
\begin{question}
Is the viscosity, Hessian principle equivalent to the martingale completeness of $X$?
\end{question}
The second issue is the relationship between the viscosity versions of the maximum principles and their classical counterparts for $C^2$ functions. We saw that relaxing the regularity of the class of functions in the definition of the principles makes them powerful enough to force more rigidity and deduce new implications, as in Theorems \ref{teo_ekeland_intro}, \ref{teo_Laplacian_intro}, \ref{teo_hessianmax_intro} and Corollary \ref{cor_SMPecomplete}. However, it is natural to ask whether the viscosity and classical definitions are, indeed, equivalent. This is the case for the weak Laplacian principle, see Remark \ref{prop_weak_classicalvisco}. However, the situation for the strong Laplacian and for the Hessian cases is more delicate.
\begin{question}
Can equivalence $(1) \Leftrightarrow (2)$ in Theorem \ref{teo_hessianmax_intro} be improved to show that the viscosity, Hessian principle is equivalent to the classical, $C^2$ one? 
\end{question}
\begin{question}
Is the viscosity, strong Laplacian principle equivalent to the classical, Yau's version of it for $C^2$ functions?  
\end{question}
Both the equivalence $(1) \Leftrightarrow (2)$ in Theorem \ref{teo_hessianmax_intro} and the one in Remark~\ref{prop_weak_classicalvisco} depend on the possibility to approximate USC functions contradicting the Ahlfors property with more regular ones. In the Hessian case, we exploit the semiconcavity of $\widetilde{F}$-harmonic functions, for $\widetilde{F}= \{ \lambda_m(A) \geq 1\}$, and a natural strategy to  find $C^2$ functions contradicting the Ahlfors property arguably leads to investigate how Riemannian convolution (see \cite{greene_wu}) behaves on $\widetilde{F}$-harmonics. To apply analogous ideas for the strong Laplacian principle, one also needs to study the regularity of $(\widetilde{F}\cup \widetilde{E})$-harmonics, for $\widetilde{F}= \{\tr(A) \geq 1\}$, which might be of independent interest. It could be possible that, in this case, an approximation via the heat flow be helpful.\par
Another problem that seems to be worth investigating is the relation between our viscosity versions of the strong (Hessian, Laplacian) principles and the approach in \cite{pengzhou}. In Theorems 1.3 and 1.4 therein, a viscosity strong maximum principle is seen as a ``theorem on sums at infinity" for a pair of functions $u,v \in \USC(X)$ with $\sup_X(u+v) < +\infty$. In the particular case when $u$ is constant, Theorem 1.3 can be rephrased as follows (we remark that, given $v \in \USC(X)$, $\overline{J}^{2,+}_xv$ denotes the closure of the set of $2$-jets of test functions for $v$ at $x$).

\begin{theorem}\cite{pengzhou}\label{teo_pengzhou} 
Let $X$ be a complete manifold with sectional curvature $\mathrm{Sect} \ge -\kappa^2$, for some constant $\kappa>0$. Then, for each $v \in \USC(X)$ bounded from above, there exists a sequence $\{x_\eps\}$ and jets $J_\eps = (v(x_\eps),p_\eps, A_\eps) \in \overline{J}^{2,+}_{x_\eps} v$ such that 
$$
v(x_\eps) > \sup_X v-\eps, \qquad |p_\eps| < \eps, \qquad A_\eps \le \eps I.
$$
\end{theorem}
Taking into account the definition of $\bar J^{2,+}_x v$, Theorem \ref{teo_pengzhou} is \emph{equivalent} to say that a complete manifold with $\mathrm{Sect} \ge -\kappa^2$ has the Ahlfors property for $\{\lambda_m(A) \ge 1\} \cup \widetilde{E}$, i.e., the viscosity, strong Hessian principle. This result is therefore a particular case of (i) in Proposition \ref{cor_sufficientiSMP}. On the other hand, a version of the principle for pairs of functions $u,v$ might be useful in view of possible applications, and suggests the following
\begin{question}
Does there exist a workable version of the Ahlfors property in the form of a ``theorem on sums" at infinity? Could it be useful to prove, for instance, comparison principles at infinity? 
\end{question}

Eventually, another issue concerns the removability of the conditions on $\{\lambda_j(t)\}$ in $ii)$ of Theorem \ref{cor_bonito}. It is very likely that AK-duality holds for any subequation locally jet-equivalent to $(\EE 7)$, but the technical restrictions depend on the lack of a suitable comparison theorem in a manifold setting, as stressed in the Appendix. For this reason, it seems to us very interesting to investigate the next 

\begin{question}
Does there exist a more general theorem on sums on manifolds, or a different comparison technique for viscosity solutions, that apply to quasilinear subequations on each manifold?
\end{question}

In this respect, the beautiful result in \cite{kawohlkutev} could be helpful.

\section{Preliminaries}\label{sec_prelim}
In this section, we will review some basics of Harvey-Lawson's approach to fully non-linear equations.

\subsection{Subequations and $F$-subharmonics}

Hereafter, $X$ will be a Riemannian manifold, possibly incomplete, of dimension $m \ge 1$. Given $x_0 \in X$, the function $\varrho_{x_0}$ will denote the distance function from $x_0$, and $B_R(x_0)$ the geodesic ball of radius $R$ centered at $x_0$. In some instances, we implicitly use the musical isomorphism between $TX$ and $T^*X$ to raise and lower indices of tensor fields, when this does not cause confusion: for example, we compute $\tr(\cal T \cdot \nabla^2 u)$ for a $(2,0)$-tensor $\cal T$, without specifying that we are considering the $(1,1)$-versions of $\cal T$ and $\nabla^2 u$, and so on.

Let $J^2(X) \ra X$ be the two jet-bundle over $X$:
$$ 
J^{2}(X) \cong \mathbf{R} \oplus T^{\ast}X \oplus \Sym^{2}(T^{\ast}X),
$$ 
and denotes its points with the $4$-ples
$$
(x,r,p,A) \in X \times \R \times T^\ast_xX \times \Sym^2(T_x^\ast X).
$$
A fully non-linear equation is identified as a constraint on the $2$-jets of functions to lie in prescribed well-behaved subsets $F \subset J^2(X)$, called \emph{subequations}. The basic properties required on $F$ involve the following subsets of $J^2(X)$:
$$
\begin{array}{lcll}
N & : & \disp N_x \doteq \big\{(r,0,0) :  r \le 0\big\} & (\text{the jets of non-positive constants});\\
P & : & \disp P_x \doteq \big\{(0,0,A) : A \ge 0\big\} & (\text{the positive cone}).
\end{array} 
$$
\begin{definition}
Let $F \subset J^2(X)$. Then $F$ is called a \emph{subequation} if it satisfies
\begin{itemize}
\item[-] the \emph{positivity} condition $(P):$  $\quad F+P \subset F$;
\item[-] the \emph{negativity} condition $(N):$  $\quad F+N \subset F$;
\item[-] the \emph{topological condition} $(T):$
$$
(i) \ \  F = \overline{\inte F}, \qquad (ii) \ \ F_x = \overline{\inte F_x}, \qquad (iii) \ \ (\inte F)_x = \inte F_x.
$$
\end{itemize}
\end{definition}
In particular, a subequation is a closed subset. Condition $(P)$ is a mild ellipticity requirement, while $(N)$ parallels the properness condition, as stated in \cite{CIL}.

A function $u \in C^2(X)$ is called $F$-subharmonic (respectively, strictly $F$-subharmonic) if, for each $x \in X$, $J^2_x u \in F_x$ (respectively, in $(\inte F)_x$). The definition can be extended to the space $\USC(X)$ of upper-semicontinuous, $[-\infty, +\infty)$-valued functions via the use of test functions.
\begin{definition}
Let $u \in \USC(X)$ and $x_0 \in X$. A function $\phi$ is called a test for $u$ at $x_0$ if $\phi \in C^2$ in a neighborhood of $x_0$, $\phi \ge u$ around $x_0$ and $\phi(x_0)=u(x_0)$.
\end{definition}

\begin{definition}\label{def_Fsubarm}
$u \in \USC(X)$ is said to be $F$-subharmonic if
$$
\forall \, x \in X, \ \ \ \forall \, \phi \text{ test for $u$ at x } \quad \Longrightarrow  \quad J^2_x\phi \in F_x .
$$
The set of $F$-subharmonic functions on $X$ is denoted by $F(X)$.
\end{definition}


To define strictly $F$-subharmonics, it is convenient to fix the Sasaki metric on $J^2(X)$, that is, a metric on $J^2(X)$ which is flat on the fibers $J^2_x(X)$. The metric generates a distance, $\dist$, and the induced distance $\dist_x$ on each fiber $J^2_x(X)$ satisfies 
\begin{equation}\label{sasaki}
\dist_x(J_1,J_2) = \|J_1-J_2\| \qquad \text{for } \, J_1,J_2 \in J^2_x(X).
\end{equation}
Balls in $J^2(X)$ will always be considered with respect to $\dist$. To extend the concept of strictly $F$-subharmonicity to the USC setting, for $F \subset J^2 (X)$ and a constant $c>0$ define 
$$
F_x^c = \big\{J \in F_x \, : \, \dist_x(J, \partial F_x) \ge c\big\}. 
$$
If $F$ is a subequation, by \eqref{sasaki} the set $F^c$ satisfies $(P),(N)$ and $F^c \subset \inte F$, but $F^c$ does not necessarily satisfy $(T)$. 
\begin{definition}\label{defsubharmc2}
Let $F \subset J^{2}(X)$. A function $u \in \USC(X) $ is said to be \emph{strictly $F$-subharmonic} if, for each $x_0 \in X$, there exist a neighborhood $B \subset X$ of $x_0$ and $c>0$ such that $u \in F^c(B)$. The set of strictly $F$-subharmonic functions on $X$ will be denoted by $F^\str(X)$.
\end{definition}

%
  
Hereafter, given a function $u$ on $X$, we define $u^*,u_*$ to be the USC and LSC regularizations of $u$:
\begin{equation}\label{def_USCLSC}
u^*(x) \doteq \limsup_{y \ra x}u(y), \qquad  u_*(x) \doteq \liminf_{y \ra x}u(y) \qquad \forall \, x \in X.
\end{equation}
$F$-subharmonics and strictly $F$-subharmonics enjoy the following properties. 

\begin{proposition}[\cite{HL_dir}, Thm. 2.6 and Lemma 7.5]\label{prop_basicheF} Let $F \subset J^2(X)$ be closed and satisfying $(P)$. Then, \vspace{0.1cm}
%
\begin{itemize}
\item[(1)] if $u, v \!\in\! F(X)$ $(\text{resp, $F^\str(X)$})$, then $\max\{u,v\} \!\in\! F(X)$ $(\text{resp, $F^\str(X)$})$; 
\item[(2)] if $\{u_j\} \subset F(X)$ is a decreasing sequence, then $u \doteq \lim_j u_j \in F(X)$; 
\item[(3)] if $\{u_j\} \subset F(X)$ converges uniformly to $u$, then $u \in F(X)$; 
\item[(4)] if $\{u_\alpha\}_{\alpha \in A} \subset F(X)$ is a family of functions, locally uniformly bounded above, then the USC regularization $v^*$ of $v(x) = \sup_\alpha u_\alpha(x)$ satisfies $v^* \in F(X)$; 
\item[(5)] (stability) if $u \in F^\str(X)$ and $\psi \in C^2_c(X)$ (i.e. $\psi$ has compact support), then there exists $\delta>0$ such that $u + \delta \psi \in F^\str(X)$.
\end{itemize}
\end{proposition}

\subsection{Examples: universal Riemannian subequations}
These are subequations constructed by transplanting on $X$ an Euclidean model via the action of the orthogonal group. First, consider the two jet bundle 
$$
J^ 2( \R^m) = \R^m \times \R \times \R^m \times \Sym^2(\R^m) \doteq \R^m \times \bf{J^2}, 
$$ 
with 
$$
\bf{J^2} \doteq \R \times \R^m \times \Sym^2(\R^m),
$$
and a subset $\bf{F} \subset \bf{J^2}$ (called a model). Then, $\R^m \times \bf{F}$ is a subequation on $\R^m$ whenever $\bf{F} = \overline{\inte \bf{F}}$ and the positivity and negativity conditions are satisfied. These are called \emph{universal (Riemannian) subequations}. Now, for each $x \in X^m$, we can choose a chart $(U, \varphi)$ around $x$ and a local orthonormal frame $e = (e_1,\ldots, e_m)$ on $U$, which induces a bundle chart
\begin{equation*}\label{localbundlechart}
\begin{array}{ccc}
\disp J^2(U) = \mathbf{R} \oplus T^*U \oplus \Sym^2(T^*U) & \disp \stackrel{\Phi^e}{\longrightarrow} & \disp \varphi(U) \times \bf{J^2} \\[0.2cm]
\big(x, u(x), \di u(x), \nabla^2 u(x)\big) & \longmapsto & \big(\varphi(x), u(x), [u_j(x)], [u_{ij}(x)]\big),
\end{array}
\end{equation*}
where $[u_j(x)]=\, ^t[u_1(x),\ldots, u_m(x)]$, $u_j=\di u(e_j)$, and $[u_{ij}(x)]$ is the matrix with entries $u_{ij}= \nabla^2 u (e_i,e_j)$. If $\{\bar e_j\}$ is another local orthonormal frame on $U$, $\bar{e}_j = h_j^i e_i$ for some smooth $h : U \ra O(m)$. The bundle chart $\Phi^{\bar e}$ is thus related to $\Phi^e$ via
\begin{equation}\label{changemap}
\Phi^{\bar e} \circ (\Phi^e)^{-1} \quad : \quad \big(\varphi(x),r,p,A\big) \longmapsto \big(\varphi(x),r, h(x)p, h(x)Ah^{t}(x)\big).
\end{equation}
Now, suppose that a model $\bf{F} \subset \bf{J^2}$ is invariant by the action of $h \in O(m)$ given by $(r,p,A) \mapsto (r,hp, hAh^t)$. Then, we can define $\mathbb{F} \subset J^2(X)$ by requiring 
$$
(r,p,A) \in \mathbb{F}_x \Longleftrightarrow      \Phi^e(x,r,p,A) \in \R^m \times \bf{F}
$$
for some choice of the orthonormal frame, and the prescription is well defined (i.e. independent of $e$) because of \eqref{changemap} and the invariance of $\bf{F}$. It is easy to check that $\mathbb{F}$ is a subequation if and only if $\R^m \times \mathbf{F}$ is a subequation in $\R^m$. 
\begin{remark}
\emph{The above method can be applied in a more general setting: suppose that $X$ is endowed with a topological $G$-structure, that is, a family of local trivializations $(U,e)$ of $TX$ such that the change of frames maps $h$ are valued in a Lie group $G$. Then, all models $\bf{F} \subset \bf{J^2}$ which are invariant with respect to the action of $G$ can be transplanted to $X$ as above. For instance, if $G = U(m)$ then $\bf{F}$ can be transplanted on each almost complex, Hermitian manifold of real dimension $2m$. 
}
\end{remark}

\noindent \textbf{Examples of universal Riemannian subequations}.

\begin{itemize}
\item[$(\EE 1)$] (The eikonal). The prototype eikonal subequation is $E = \big\{|p| \le 1\big\}$. More generally, we consider the generalized eikonal $E_\xi = \overline{\big\{|p| < \xi(r)\big\}}$ where $\xi \in C(\R)$ is required to satisfy either $(\xi 1)$ in \eqref{def_f1xi1}, or 
\begin{equation}\label{def_xi0}
\begin{array}{ll}
(\xi 0) & \qquad  0 < \xi \in C(\R), \quad \xi \, \text{ is non-increasing.}
\end{array}
\end{equation}
\item[$(\EE 2)$] (The weak Laplacian principle subequation). Set $\mathbb{F} = \big\{\tr A \ge f(r)\big\}$, for $f \in C(\R)$ non-decreasing. As we saw, $\mathbb{F}$ is related to the stochastic completeness, the parabolicity of $X$ and the weak maximum principle at infinity. For $f \equiv 0$, $\mathbb{F}$ characterizes subharmonic functions.
\item[$(\EE 3)$] (The weak Hessian principle subequations). For $f \in C(\R)$ non-decreasing, and denoting with $\{\lambda_j(A)\}$ the increasing sequence of eigenvalues, consider $\mathbb{F} = \big\{\lambda_k(A) \ge f(r)\big\}$. By the monotonicity of $f$ and of the eigenvalues of $A$ ($\lambda_k(A+P) \ge \lambda_k(A)$ when $P \ge 0$), $\mathbb{F}$ satisfies $(P)$ and $(N)$, and clearly also $(T)$. Thus, $\mathbb{F}$ is a subequation, which for $j=1$ and $j=m$ is related to the martingale completeness and the weak Hessian  principle. If $f \equiv 0$, note that $\{\lambda_k(A) \ge 0\}$ is the $k$-th branch of the Monge-Amp\`ere equation $\det(A) = 0$.

\item[$(\EE 4)$] (The branches of the $k$-Hessian subequation). For $\lambda \doteq (\lambda_1,\ldots, \lambda_m) \in \R^m$ and $k \in \{1,\ldots, m\}$, consider the elementary symmetric function
$$
\sigma_k(\lambda) = \sum_{1 \le i_1<\cdots < i_k \le m} \lambda_{i_1}\lambda_{i_2}\cdots \lambda_{i_k}.
$$
Since $\sigma_k$ is invariant by permutation of coordinates of $\lambda$, we can define $\sigma_k(A)$ as $\sigma_k$ being applied to the ordered eigenvalues $\{\lambda_j(A)\}$. According to G\"arding's theory in \cite{G}, $\sigma_k(\lambda)$ is hyperbolic with respect to the vector $v = (1,\ldots, 1) \in \R^m$. Denote with 
$$
\mu_1^{(k)}(\lambda) \le \cdots \le \mu_k^{(k)}(\lambda) 
$$
the ordered eigenvalues\footnote{That is, the opposite of the roots of $\mathscr{P}(t) \doteq \sigma_k(\lambda + tv)=0$.} of $\sigma_k$. Clearly, $\mu_j^{(k)}$ is permutation invariant, thus the expression $\mu_j^{(k)}(A)$ is meaningful. A deep monotonicity result (\cite[Thm. 6.2 and Cor. 6.4]{HL_garding}, see also \cite[Thm. 5.4]{HL_gardingesub}), together with the fact that the set of eigenvalues of $P \ge 0$ is contained in the closure of the G\"arding cone of $\sigma_k$, guarantees the inequality 
\begin{equation}\label{fundamental_mono}
\mu_j^{(k)}(A+P) \ge \mu_j^{(k)}(A) \qquad \text{for each } \, P \ge 0.
\end{equation}
Consequently, for each $f \in C(\R)$ non-decreasing,
$$
\mathbb{F} = \big\{ \mu_j^{(k)}(A) \ge f(r)\big\}
$$
satisfies $(P)+(N)$. Condition $(T)$ is a consequence of \cite[Prop. 3.4]{HL_garding}, hence $\mathbb{F}$ is a subequation. Many more examples of this kind arise from hyperbolic polynomials $q(\lambda)$, with the only condition that the positive octant $\R^m_+ \doteq \{\lambda_i \ge 0, \  \forall \, i\}$ lies in the closure of the G\"arding cone of $q$ (see \cite{HL_gardingesub} for details). 
%
%

\item[$(\EE 5)$] (The $k$-plurisubharmonicity subequation). Let $\GG = \mathrm{Gr}_k(\R^m)$ be the Grassmannian of unoriented $k$-planes in $\R^m$ passing through the origin. Then, for $f \in C(\R)$ non-decreasing, the model
\begin{equation}\label{model_plurisub}
\mathbf{F}_\GG  \doteq \Big\{(r,p,A) \in \mathbf{J}^2: \tr(A_{|\xi}) \ge f(r) \ \text{ for each } \, \xi \in \GG\Big\}
\end{equation}
can be transplanted on each Riemannian manifold and, via the min-max characterization, the resulting subequation can equivalently be described as
$$
\mathbb{F} = \big\{\lambda_1(A) + \cdots + \lambda_k(A) \ge f(r) \big\}.
$$
These operators naturally arise in the study of submanifolds of Riemannian or complex ambient spaces, and we refer to \cite{HL_dir, HL_plurisub} and the references therein for a thorough discussion.

\item[$(\EE 6)$] (Complex subequations). Examples analogous to $(\EE 3),(\EE 4)$ and $(\EE 5)$ can be considered on almost complex hermitian manifolds, using the eigenvalues of the hermitian symmetric part $A^{(1,1)}$ of $A$. In this way, in examples $(\EE 3),(\EE 4)$ we recover the subequation $\{ \lambda_1(A^{(1,1)}) \ge 0\}$ describing plurisubharmonic functions, and more generally all the branches of the complex Monge-Amp\`ere equation. The complex analogue of $(\EE 5)$ gives rise to the $k$-plurisubharmonics: it corresponds to transplanting, on an Hermitian manifold of complex dimension $m$, the model $\mathbf{F}_\GG$ in \eqref{model_plurisub} with $\GG= \mathrm{Gr}_k(\CC^m)$ the Grassmannian of unoriented complex $k$-planes. In the same way, analogous examples can be treated on almost quaternionic Hermitian manifolds (see \cite{HL_dir} for details).
\item[$(\EE 7)$] (Quasilinear operators). As said in the introduction, for $f \in C(\R)$ non-decreasing the subequation reads 
\begin{equation*}\label{def_aLapla}
\mathbb{F} = \overline{\left\{p \neq 0, \ \tr \big( T(p)A\big) > f(r) \right\}}, 
\end{equation*}
where 
\begin{equation}\label{def_Tp}
T(p) \doteq \lambda_1(|p|) \Pi_p + \lambda_2(|p|) \Pi_{p^\perp}, \qquad  \begin{cases}
\lambda_1(t) = a(t) + ta'(t) \ge 0, \\
\lambda_2(t) = a(t) > 0,
\end{cases}
\end{equation}
and $\Pi_p, \Pi_{p^\perp}$ are, respectively, the $(2,0)$-version of the orthogonal projection onto the spaces $\langle p\rangle$ and $p^\perp$. Observe that \eqref{def_Tp} implies $T(p) \ge 0$ as a quadratic form, hence $\mathbb{F}$ satisfies $(P)$, and $(N),(T)$ are immediate. Hence, $\mathbb{F}$ is a subequation. Relevant examples include:
\begin{itemize}
\item[-] The $k$-Laplace operators, where $a(t)=t^{k-2}$, for $k \in [1,+\infty)$;
\item[-] The mean curvature operator, where $a(t) = (1+t^2)^{-1/2}$;
\item[-] The operator of exponentially harmonic functions, where $a(t) = \exp\{t^2\}$.
\end{itemize}
\item[$(\EE 8)$] (The normalized $\infty$-Laplacian). This operator is given by the subequation
$$
\mathbb{F} = \overline{\big\{p \neq 0, \ \ |p|^{-2}A(p,p)> f(r) \big\}},
$$
where $f \in C(\R)$ is non-decreasing.
\end{itemize}

\subsection{Plugging non-constant coefficients: affine jet-equivalence}
The next procedure allows to extend the class of subequations to which Harvey-Lawson's theory can be applied, including variable coefficient subequations, and it is based on the following

\begin{definition}
A \emph{jet-equivalence} $\Psi : J^2(X) \ra J^2(X)$ is a continuous bundle automorphism (i.e. it preserves the fibers of $J^2(X) \ra X$) that has the following form:
\begin{equation}\label{structurejetequi}
\Psi (x,r,p,A) = \big(x, r, gp, h A h^t + L(p)\big),
\end{equation}
where 
$$
\begin{array}{ccll}
g,h & : & T^*X \ra T^*X  & \qquad \text{are bundle isomorphism,} \\[0.2cm]
L & : & T^*X \ra \Sym^2(T^*X) & \qquad \text{is a bundle homomorphism.}
\end{array}
$$
An \emph{affine jet-equivalence} is a continuous bundle map $\Phi : J^2(X) \ra J^2(X)$ that can be written as $\Phi = \Psi + J$, for some jet-equivalence $\Psi : J^2(X) \ra J^2(X)$ and a section $J : X \ra J^2(X)$.
\end{definition}

In view of Example \ref{ex_2} below, the definition does not depend on the connection used to split $J^2(X)$. The set of (affine) jet-equivalences form a group, and $F, F' \subset J^2(X)$ are said to be (affine) jet-equivalent if there exists a (affine) jet-equivalence $\Psi$ with $\Psi(F)=F'$. 

\begin{remark}\label{rem_piustupido}
\emph{It is easy to check that $F$ is a subequation if and only if so is $\Psi(F)$. Suppose that $F,G$ are two subequations on $X$ which are affine jet-equivalent: $\Psi(F) = G$ for some $\Psi$. We claim that, fixing $c>0$ and a compact set $K \subset X$, there exists $\bar c>0$ such that $\Psi(F_x^c) \subset G_x^{\bar c}$ for each $x \in K$. Indeed, the claim follows because $\Psi$ is affine on fibers (hence, the stretching factor of each map $\Psi : F_x \ra G_x$ is constant), and using a compactness argument in the variable $x$. 
}
\end{remark}

\begin{example}[Change of frames]\label{ex_1}
Consider a local chart $(U, \varphi)$ on $X$, and two frames $e,\bar e$ on $U$ (possibly not orthonormal). Then, the change of frame map \eqref{changemap} can be rewritten as
\begin{equation}\label{changeofframe_2}
\begin{array}{lcccc}
\disp \Phi^{\bar e} \circ (\Phi^e)^{-1} & : & \varphi(U) \times \mathbf{J^2} & \longrightarrow & \varphi(U) \times \mathbf{J^2} \\[0.2cm]
& & (y,r,p,A) & \longmapsto & \disp \big(y,r, \bar h(y)p, \bar{h}(y)A\bar{h}^{t}(y)\big)
\end{array}
\end{equation}
with $\bar h \doteq h \circ \varphi^{-1} : \varphi(U) \ra \mathrm{GL}_m(\R)$. Clearly, \eqref{changeofframe_2} has the structure in \eqref{structurejetequi}, hence $\Phi^{\bar e} \circ (\Phi^e)^{-1}$ is a jet-equivalence of $J^2(\varphi(U))$.
\end{example}

\begin{example}[Linear subequations]\label{ex_linear}
A prototype example is that of linear subequations with continuous coefficients, possibly in non-divergence form. Let $W : X \ra T^*X$ and $\cal T : X \ra \Sym^2(T^*X)$ be continuous tensor fields, and suppose that $\cal T$ is positive definite at each point. Fix $f \in C^ 0(\R)$ non-decreasing and $b,B \in C(X)$, $b>0$ on $X$, and consider the linear operator 
$$
Lu = \tr(\cal T\cdot \nabla^2 u) + \langle W, \dif u \rangle + B.
$$
We claim that the subequation $F_L$ characterizing solutions of $Lu \ge b(x)f(u)$ is affinely jet-equivalent to the universal Riemannian subequation $F_\Delta = \{\tr(A) \ge f(r)\}$ (and jet-equivalent to it whenever $B \equiv 0$). Indeed, rewrite $F_L$ as follows
\begin{equation*}\label{def_FL}
F_L = \Bigg\{ (x,r,p,A)  :  \tr \left(b(x)^{-1}\left[\cal T(x)A \!+\! W(x) \odot p \!+\! \frac{B(x)}{m} \metric_x\right] \right) \!\ge\! f(r) \Bigg\},
\end{equation*}
where $a \odot b = \frac{1}{2}(a\otimes b + b\otimes a)$ is the symmetric product. Since $\cal T$ is positive definite, it admits a continuous, positive definite square root $H : X \ra \Sym^2(T^*X)$ (see \cite[p. 131]{stroockvaradhan}), and $H$ is Lipschitz whenever so is $\cal T$. From $\tr(\cal T A) = \tr(HAH^t)$ we deduce that the map
$$
\Psi(x,r,p,A) = \left(x,r,p, \frac{1}{b(x)}\left[H(x)AH^t(x) + W(x) \odot p + \frac{B(x)}{m}\metric_x\right] \right) 
$$ 
is an affine jet-equivalence satisfying $\Psi(F_L) = F_\Delta$. 
\end{example}

%

\begin{definition}
$F \subset J^2(X)$ is said to be \emph{locally affine jet-equivalent to a universal subequation} if, for each $x \in X$, there exist a local chart $(U,\varphi)$ around $x$ and coordinate frame $e$ on $U$ such that the decription of $F$ in the frame $e$, $\Phi^e(F)$, is affine jet-equivalent to some universal subequation $\varphi(U) \times \mathbf{F}$. Such a chart will be called a distinguished chart.
\end{definition}

By Example \ref{ex_1}, the above definition is independent of the chosen frame $e$. Moreover, the Euclidean model $\mathbf{F}\subset \mathbf{J^2}$ is uniquely defined, independently of the chart $(U, \varphi)$ (see \cite[Lemma 6.10]{HL_dir}).

\begin{example}[Universal Riemannian subequations]
By construction, universal Riemannian subequations with model $\mathbf{F}$ are locally jet-equivalent to $\R^m \times \mathbf{F}$.
\end{example}

\begin{example}[Change of connection]\label{ex_2} 
Consider a local trivialization $\Phi^e$ of $J^2(X)$ in a chart $(U, \varphi)$ with coordinate frame $e=\{\partial_j\}$:
\begin{equation}\label{Phie_22}
\Phi^e \  :  \ J^2_x u \in J^2(U) \ \longmapsto \ \big(\varphi(x),u(x), u_j(x), u_{ij}(x)\big) \in \varphi(U) \times \mathbf{J^2},
\end{equation}
where
$$ 
u_j = \partial_ju, \qquad u_{ij} = \nabla^2 u(\partial_i,\partial_j) = \partial^2_{ij} u - \Gamma^k_{ij} \partial_ku,
$$
and define the map 
\begin{equation}\label{localbundlechart_ateEuclidiano}
\begin{array}{cccc}
\Psi_e \ \ : \ \ \disp  & \disp \varphi(U) \times \bf{J^2} & \longrightarrow & \disp \varphi(U) \times \bf{J^2} \\[0.2cm]
& \disp \big(\varphi(x), u, \partial_ju, u_{ij}\big) & \longmapsto & \disp \big(\varphi(x),u,\partial_ju,\partial^2_{ij}u\big).
\end{array}
\end{equation}
Clearly, $\Psi_e$ is a jet-equivalence, and $\Psi_e\circ \Phi^e$ is the frame representation of $J^2(U)$ with respect to the flat connection. 
\end{example}

The last example shows that the definition of jet-equivalence is also independent of the connection $\nabla$ chosen to split $J^2(U)$, which makes it particularly effective. Throughout the paper, given a subequation $F \subset J^2(X)$ and a chart $(U, \varphi)$, we will say that 
$$
\mathcal{F} \doteq \Psi_e(\Phi^e(F)) \ \subset \ \varphi(U) \times \mathbf{J^2} 
$$
with $\Phi^e,\Psi_e$ as in \eqref{Phie_22} and \eqref{localbundlechart_ateEuclidiano}, is an \emph{Euclidean representation} of $F$ in the chart $(U,\varphi)$. Observe that, by construction, if $w \in C^2(U)$ and $\bar w \doteq w \circ \varphi^{-1}$,
$$
J^2_x w \in F \quad \Longleftrightarrow \quad J^2_{\varphi(x)} \bar w = \big( \varphi(x), \bar w, \partial_j \bar w, \partial^2_{ij} \bar w\big) \in \mathcal{F}. 
$$
The use of local Euclidean representations with respect to the standard ``Euclidean" coordinates on $J^2(\R^m)$ allows a direct applications of some important results like the theorem on sums (\cite{CIL}, see also \cite[Thm. C.1]{HL_dir}), without the necessity to use its more involved Riemannian counterpart. We will come back to this point later.


\subsection{Dirichlet duality}

Given $F \subset J^{2}(X)$, the \emph{Dirichlet dual of} $F$ is
$$ 
\widetilde{F} = \  \sim (-\inte F) = - (\sim \inte F). 
$$
Note that $\widetilde F$ is always closed, and that $\partial F \equiv F \cap (-\widetilde{F})$. The term ``duality" is justified by the following properties that can be readily verified.
\begin{proposition}\label{prop_basichetilde}
Let $F,F_1,F_2 \subset J^2(X)$. Then,
$$
\begin{array}{l}
F_1 \subset F_2 \Rightarrow \widetilde{F_2} \subset \widetilde{F_1}; \qquad \widetilde{F_1 \cap F_2} = \widetilde{F_1} \cup \widetilde{F_2}; \qquad \widetilde{\widetilde F} = F  \Longleftrightarrow  F = \overline{\inte F}; \\[0.2cm]
\text{$F$ is a subequation} \quad \Longleftrightarrow \quad \text{$\widetilde F$ is a subequation}.
\end{array}
$$
\end{proposition}


\begin{remark}\label{remlocalaffineequieduality}
\emph{If $F$ is locally affine jet-equivalent to a universal subequation $\mathbf{F}$, then $\widetilde F$ is locally affine jet-equivalent to $\widetilde{\mathbf{F}}$. 
}
\end{remark}
\noindent\textbf{Examples of universal subequations: duality.}\vspace{0.3cm}
$$
\begin{array}{rlcl}
(\EE 1) & \disp \quad E = \overline{\big\{ |p| < \xi(r)\big\}} & \Longrightarrow & \ \disp\widetilde E = \overline{\big\{|p| > \xi(-r)\big\}}; \\[0.2cm] 
(\EE 3) & \disp \quad F = \big\{ \lambda_k(A) \ge f(r) \big\} & \Longrightarrow & \ \disp \widetilde F = \big\{ \lambda_{m-k+1}(A) \ge -f(-r)\big\}; \\[0.2cm] 
(\EE 4) & \quad \disp F = \big\{ \mu_j^{(k)}(A) \ge f(r)\big\} & \Longrightarrow & \ \disp \widetilde{F} = \big\{ \mu^{(k)}_{k-j+1}(A) \ge -f(-r)\big\}; \\[0.2cm]
(\EE 2),(\EE 5) & \quad F = \big\{\sum_{j=1}^k \lambda_j(A) \ge f(r)\big\} & \Longrightarrow & \ \widetilde{F} = \big\{ \sum_{j=m-k+1}^m \lambda_{j}(A) \ge -f(-r)\big\}; \\[0.2cm]
(\EE 7) & \quad \disp F = \overline{\big\{\vert p\vert > 0, \tr(T(p)A) > f(r) \big\}} & \Longrightarrow & \ \disp \widetilde F = \overline{\big\{\vert p\vert > 0, \tr(T(p)A) > -f(-r) \big\}}; \\[0.2cm] 
(\EE 8) & \quad \disp F = \overline{\big\{|p|>0, \, |p|^{-2}A(p,p)> f(r)\big\}} & \Longrightarrow & \ \disp \widetilde F = \overline{\big\{|p|>0, \, |p|^{-2}A(p,p)> -f(-r)\big\}}.
\end{array}
$$

Example $(\EE 4)$ follows from the relation 
$$
\mu_j^{(k)}(-A) = - \mu_{k-j+1}^{(k)}(A),
$$
which is proved in \cite[Sec. 3]{HL_garding}. The other examples up to $(\EE 5)$ are trivial, as well as their complex analogues. To quickly show the duality in $(\EE 7)$, set 
$$
V \doteq \big\{\vert p\vert > 0, \ \tr(T(p)A) > f(r) \big\}, \quad W \doteq \big\{\vert p\vert > 0, \ \tr(T(p)A) > -f(-r) \big\}, 
$$
and note that $F = \overline{V}$. By duality, $$ \widetilde F \subset \widetilde V = \big\{\vert p\vert > 0, \  \tr(T(p)A) \geq -f(-r) \big\} \cup \big\{p=0\big\} \  \subset  \ \overline{W} \cup \big\{p=0\big\}.$$ We claim that $ \overline{W} \subset \widetilde F$. In fact, it is sufficient to show that $ W \subset \widetilde F$, i.e., $-W \subset \, \sim \inte F$. Given $(r,p,A) \in -W$, from $ \vert p\vert > 0 $ and $ \tr(T(p)A) < f(r)$ we deduce, by continuity, the existence of a neighborhood $ U'\subset -W$ containing $(r,p,A)$ with $ U' \cap F = \emptyset$. In particular, $(r,p,A) \notin \inte F$, which proves our claim. Summarizing, $ \overline{W} \subset \widetilde F \subset \overline{W}\cup \big\{p=0\big\}$. However, since $ \widetilde F $ is a subequation and $ \big\{p=0\big\} $ has empty interior in $J^{2}(X)$, we can conclude that $ \widetilde F = \overline{W}$. The same approach yields the duality in $(\EE 8)$. 

%

We are ready to define $F$-harmonics.

\begin{definition}
Let $F \subset J^2(X)$. A function $u$ is said to be $F$-harmonic if
\begin{equation*}
u \in F(X) \ \ \mbox{and} \ \ -u \in \widetilde{F}(X) .
\end{equation*}
\end{definition}
By Proposition \ref{prop_basichetilde}, if $F$ is a subequation then $u$ is $F$-harmonic if and only if $-u$ is $\widetilde F$-harmonic.

\subsection{Comparison theory} For $K \Subset X$ precompact, we define $F(K) \doteq \USC(K) \cap F(\inte K)$. The sets $F^c(K), F^\str(K)$ are defined in the same way, simply replacing $F$ with $F^c, F^\str$.
\begin{definition}
Let $F \subset J^2(X)$. We will say that $F$ on $X$ satisfies
\begin{itemize}
\item[(i)] the \emph{comparison} if for every compact subset $K\Subset X$ and $u\in F(K)$, $v\in \widetilde{F}(K)$, the zero maximum principle holds on $K$, that is,
$$
u+v \leq 0 \ \ \mbox{on} \ \ \partial K \quad \Longrightarrow \quad u+v\leq 0 \ \ \mbox{on} \ \ K;
$$
\item[(ii)] the \emph{weak comparison} if for every compact subset $K\Subset X$, $c>0$, and $u\in F^c(K)$, $v\in \widetilde{F}(K)$, the zero maximum principle holds on $K$.
\item[(iii)] the \emph{local weak comparison} if each $x\in X$ has a neighborhood $U$ such that $F$ on $U$ satisfies the weak comparison.
\end{itemize}
We also say that $F$ on $X$ satisfies the \emph{bounded} comparison (respectively, weak comparison and local weak comparison) if the statements hold when restricted to bounded functions $u$ and $v$.  
\end{definition}

\begin{remark}
\emph{Note that, while the full comparison property is symmetric in $F$ and $\widetilde F$, the weak comparison is not.
}
\end{remark}
 
The next important result is a consequence of the stability of strict $F$-subharmonics in Proposition \ref{prop_basicheF}.

\begin{theorem}\label{lwcimplieswc}\cite[Thm. 8.3]{HL_dir}
Let $F \subset J^2(X)$ satisfy $(P)$ and $(N)$. Then, 
$$
\begin{array}{c}
\text{$F$ satisfies the} \\
\text{local weak comparison} \end{array}
\quad \Longleftrightarrow \quad \begin{array}{c}
\text{$F$ satisfies the} \\
\text{weak comparison} \end{array}.
$$
\end{theorem}

The advantage of this theorem is that the local weak comparison can be checked in a local chart, for instance with the aid of the theorem on sums. The next result guarantees the weak comparison for a large class of subequations. 

\begin{theorem}\cite[Thm. 10.1]{HL_dir}\label{teo_importante!!} Let $F \subset J^2(X)$ be a subequation which is locally affinely jet-equivalent to a universal subequation, where the continuous sections $g,h,L$ in \eqref{structurejetequi} are locally Lipschitz. Then, $F$ and $\widetilde F$ satisfy the weak comparison. In particular, each universal Riemannian subequation satisfy the weak comparison\footnote{The conclusion of Theorem 10.1 in \cite{HL_dir} just states that $F$ satisfies the weak comparison, but the same theorem can be applied to $\widetilde F$ in view of Remark \ref{remlocalaffineequieduality}.}.
\end{theorem}

If the weak comparison holds, the validity of the full comparison property is granted when we can approximate functions in $F(X)$ with functions in $F^\str(X)$.

\begin{definition}
Let $F \subset J^ 2(X)$. We say that $F$ satisfies the (bounded) \emph{strict approximation} on $X$ if for each compact $K \Subset X$, each (bounded) $u \in F(K)$  can be uniformly approximated by functions in $F^\str(K)$.
\end{definition}

\begin{theorem}\cite[Thm. 9.2]{HL_dir}\label{localwcestrict}
Let $F \subset J^ 2(X)$ satisfy $(P),(N)$. Then
$$
\begin{array}{c}
\text{$F$ has the local weak comparison} \\
 + \ \ \text{(bounded) strict approximation} \end{array}
 \quad \Longrightarrow \quad \begin{array}{c}
\text{$F$ has the} \\
\text{(bounded) comparison} \end{array}.
$$
\end{theorem}

The strict approximation property is a delicate issue. A case when the strict approximation holds on a set $\Omega$ is when $\overline{\Omega}$ supports a $C^2$ function which is strictly $M$-subharmonics, where $M$ is a monotonicity cone for $F$ (see \cite{HL_dir, HL_existence}). However, in examples $(\EE 2), \ldots, (\EE 6)$ this gives topological restrictions on $\Omega$ that we would like to avoid. As we shall see, to prove Theorem \ref{cor_bonito} we will just need to show the comparison property for $F$ locally jet-equivalent to $(\EE 2), \ldots, (\EE 6)$ just when $f$ is strictly increasing. In this case, the strict approximation of $u \in F(X)$ is achieved via the functions $u_\eta \doteq u-\eta$ for positive constants $\eta>0$, provided that the dependence of $F$ on $p$ is mild enough, as quantified by the next definition which shall be compared to the classical  condition (3.14) in \cite{CIL}. For a closely related condition, in the case of subequations independent of the gradient, we refer to \cite[Section 4]{cirantpayne}.


\begin{definition}\label{def_unifcontinuous}
A function $\mathscr{F} : \R^m \times \Sym^2(\R^m) \ra \R$ is said to be \emph{uniformly continuous} if for each $\eps >0$ there exists $\delta>0$ such that, whenever
$$
(p,A), (q,B) \in \R^m \times \Sym^2(\R^m), \qquad \text{and} \qquad 
|p-q| + \|(A-B)_{+}\| < \delta,
$$
then $\mathscr{F}(q,B) \ge \mathscr{F}(p,A) -\eps$.
\end{definition}

\begin{proposition}\label{prop_unifcontinuous}
Let $\mathbf{F}= \big\{\mathscr{F}(p,A) \ge f(r)\big\}$ be one of the models in $(\EE 2), \ldots , (\EE 6)$. Then, $\mathscr{F}$ is uniformly continuous. 
\end{proposition}

\begin{proof}
In all the examples but $(\EE 4)$ and its complex analogue, the statement is immediate because there exists an absolute constant $c>0$ such that 
\begin{equation*}
\mathscr{F}(A) \le \mathscr{F}(B) + c\|(A-B)_+\|. 
\end{equation*}
Indeed, this also holds for $(\EE 4)$, as it follows using the elementary property $(2)$ in \cite[p. 1106]{HL_garding} together with the monotonicity \eqref{fundamental_mono}, and taking into account that the eigenvalues $\mu_j^{(k)}$ are evaluated with respect to the direction $(1,\ldots,1)$: 
\begin{align*}
	\mu_j^{(k)}(A) &= \mu_j^{(k)}\big(B + (A-B)\big) \\
	&\le \mu_j^{(k)}\big(B + \|(A-B)_+\|I\big) = \mu_j^{(k)}(B) + \|(A-B)_+\|.\\[-4em]
\end{align*}
\end{proof}

\begin{theorem}\label{thm_comparison_examples}
Let $F$ be locally affine jet-equivalent to a universal subequation with model $\mathbf{F} = \big\{\mathscr{F}(p,A) \ge f(r)\big\}$ via continuous bundle maps. Suppose that $\mathscr{F}$ is uniformly continuous and that $f$ is strictly increasing. Then, $F$ has the bounded, strict approximation property.
\end{theorem}

\begin{remark}
\emph{In the special case when $f \in C^1(\R)$ with $\inf_{\R}f' >0$, the boundedness assumption on $u$ is not needed.
}
\end{remark}
\begin{proof}
We consider the case of a local jet-equivalence, since adaptations to affine jet-equivalences are straightforward. Fix a compact $K\Subset X$ and let $u \in F(K)$ be bounded. We shall prove that, for each $\eta>0$, there exists $c>0$ such that $u_\eta \doteq u-\eta \in F^c(K)$. By covering $K$ with a finite number of charts $(U, \varphi)$ ensured by the local jet-equivalence condition,  it is enough to prove that $u_\eta \in F^c(\overline U)$ for some $c>0$. To this aim, take a local representation $\mathcal{F}$ of $F$ in the chart $(U, \varphi)$. By hypothesis, there exists a jet-equivalence $\Psi : \varphi(U) \times \mathbf{J^2} \ra \varphi(U) \times \mathbf{J^2}$ with
\begin{equation}\label{def_psi_novamente}
(y,r,p,A) \longmapsto  \big(y,r, g(y)p, h(y)Ah(y)^t + L_y(p)\big),
\end{equation}
satisfying 
\begin{equation*}
\Psi(\mathcal{F}) = \varphi(U) \times \mathbf{F},
\end{equation*}
that is, 
\begin{equation}\label{def_FFF}
\mathcal{F} = \Big\{ (y,r,p,A) \  :  \ \mathscr{F}\Big( g(y) p, h(y) A h(y)^t + L_y(p) \Big) \ge f(r) \Big\}.
\end{equation}
Define $\bar u \doteq u \circ \varphi^{-1}$ on $V \doteq \varphi(U)$ and note that, by construction, $\bar u \in \mathcal{F}(\overline{V})$ (we assume $\Psi$ and $\varphi$ defined on $\overline{U}$). Reasoning as in Remark \ref{rem_piustupido}, it is sufficient to prove that $\bar u_\eta \doteq \bar u -\eta \in \mathcal{F}^c(V)$.\\
%
%
%
%
Let $R$ be such that $|u| \le R$. Being $f$ strictly increasing, we can define  
$$
\varepsilon \doteq \inf\, \big\{f(r+\eta)-f(r)\  : \ r \in [-2R,2R]\big\} > 0,
$$
and since $\mathscr{F}$ is uniformly continuous, there exists $\delta>0$ such that
\begin{equation}\label{ineq_unif_con_F}
\mathscr{F}(\bar p,\bar A) > \mathscr{F}(\bar q,\bar B) -\varepsilon \quad \text{whenever} \quad \vert \bar p-\bar q\vert + \vert\vert (\bar A - \bar B)_{+}\vert\vert < \delta.
\end{equation} 
Set $M \doteq \vert \vert g \vert \vert_\infty + \vert\vert L\vert\vert_{\infty} + \vert \vert h \vert \vert^2_{\infty}$. Fix $y \in \overline{V}$ and a test function $\phi$ for $\bar u_\eta$ at $y$, and set $J^2_y\phi = (r,p,A)$. Since $\phi+\eta$ is a test for $\bar u$ at $y$, $(r+\eta,p,A) \in \mathcal{F}_y$. Thus, for $(s,q,B) \in J^2_y(V)$ satisfying
$$
\vert s- r\vert^2 + \vert p-q\vert^2 + \vert\vert (A - B)_{+}\vert\vert^2 < \min\left\{ \eta^2, \frac{\delta^2}{4M}\right\} \doteq c^2, 
$$
using \eqref{def_FFF}, \eqref{ineq_unif_con_F} and the definition of $M$ the following inequalities hold:
\begin{align*}
\mathscr{F}\big(g(y)q, h(y)Bh(y)^t + L_{y}(q)\big) &> \mathscr{F}\big(g(y)p, h(y)Ah(y)^t + L_{y}(p)\big) - \varepsilon \\
&\geq  f\big(r + \eta\big) - \varepsilon > f(s),
\end{align*} 
where the last follows since $|s-r| < \eta$. This means that the Euclidean ball $B_{y}\big(J^{2}_{y}\phi, c\big) \subset \mathcal{F}_y$, which proves that $\bar u_\eta \in \mathcal{F}^c(V)$. 
\end{proof}
%
For $F$ locally jet-equivalent to example $(\EE 7)$ or $(\EE 8)$, things are more difficult because of a worse dependence on the gradient term $p$. It is well-known that if $F$ is the universal example in $(\EE 7)$ or $(\EE 8)$ on $X = \R^m$, and $f$ is strictly increasing, comparison follows by a direct application of the theorem on sums. The same happens on any manifold with non-negative sectional curvature because of the Riemannian version of the theorem on sums in \cite{azagraferrerasanz}. However, surprisingly enough, when the curvature of $X$ is negative somewhere the Riemannian theorem on sums is not powerful enough to yield a sharp result, and needs a further uniform continuity requirement (condition $(2\flat)$ in \cite[Cor. 4.10]{azagraferrerasanz}). This accounts for the restrictions in $ii)$ and $iii)$ of Theorem \ref{cor_bonito}, proved in Section \ref{sec_main_examples}, and will be examined in Appendix \ref{appendix_1}.

When $F$ is reduced (that is, independent of the $r$-variable), comparison is a particularly subtle issue, and we refer to \cite{barlesbusca} for further insight. We here investigate just the very interesting case of the $\infty$-Laplacian (see \cite{jensen, crandall_visit}), for which we have the following result.

\begin{theorem}\label{compa_inftyLaplacian}
The comparison holds for the $\infty$-Laplace subequation $F_\infty = \overline{\{A(p,p)>0\}}$.
\end{theorem}

The result has been proved by \cite{jensen} in the Euclidean space, and a simpler, beautiful alternative argument has then been given in \cite{armstrongsmart}. Theorem \ref{compa_inftyLaplacian} follows by putting together minor modifications of results appearing in the literature. Since we have not found a reference covering the theorem in full, we sketch the proof for the convenience of the reader. 

\begin{proof}[Proof: Sketch]
For a fixed open set $\Omega \!\Subset\! X$, consider a finite covering $\{B_{\eps}(x_j)\}$ of $\overline{\Omega}$ such that, for each $j$ and each $y \in B_{4\eps}(x_j)$, the squared distance function $\varrho^2_y$ from $y$ is smooth on $B_{2\eps}(y)$ with $\nabla^2 \varrho^2_y \ge \metric$. Then, the result follows from the following steps.
\begin{itemize}
\item[-] On $B_j \doteq B_{2\eps}(x_j)$, $u \in F_\infty(\overline{B_j \cap \Omega})$ implies that $u$ satisfies the comparison with metric cones: whenever $U \Subset B_j \cap \Omega$ and, for fixed $x \in U$, there exist $a,b \in \R$ with
$$
u \le a+ b \varrho_x \qquad \text{on } \, \partial(U \backslash \{x\}),
$$
then $u \le a+b \varrho_x$ on all of $U$. This follows verbatim from the proof in \cite{CEG}, observing that $\varrho_{x}^2$ on $B_j$ has the right concavity and smoothness assumptions. See also \cite[Sec. 2]{crandall_visit}.
\item[-] If $u$ has the comparison with metric cones on $B \!\cap\! \Omega$, then $u \!\in\! \lip_{\mathrm{loc}}(B\!\cap\! \Omega)$ and, by covering, $u \in \lip_{\mathrm{loc}}(\Omega)$. This follows from Lemma 2.5 in \cite{CEG}, see also \cite{jensen, lindqvistmanfredi_2}.
\item[-] The comparison theorem in \cite{armstrongsmart}: the result is stated for $u \in F_\infty(\overline{\Omega})$, $v \in F_\infty(\overline{\Omega})$ ($F_\infty$ is self-dual) which are continuous on $\Omega$. Although stated on $\R^m$, the proof is purely metric, and comparison with cones is just used in small enough balls. Hence, everything translates effortless to manifolds.
\end{itemize}\vspace{-2em}
\end{proof}

\subsection{Boundary barriers}

Let $F \subset J^2(X)$ and consider $\Omega \subset X$ with non-empty boundary. In this section, we describe the necessary assumptions on $\partial \Omega$ to possess barriers. 

\begin{definition}
Let $x_0 \in \partial \Omega$. Given $\lambda \in \R$, a \emph{$F$-barrier at height $\lambda$} is the collection of the following data:
\begin{itemize}
\item[-] A sequence of smooth, open neighborhoods $\{U_{j}\}$ shrinking to $x_0$; 
\item[-] For each $c \in \R$, a function $\beta_{j,c} \in C^2(\overline{U}_j)\cap F^\str(U_j\cap \Omega)$ satisfying 
$$
\beta_{j,c}(x_0) = \lambda, \qquad \beta_{j,c} \le \lambda \  \text{ on }  \ U_j \cap \overline{\Omega}, \qquad \beta_{j,c} < c \ \text{ on } \ \partial U_j.
$$
\end{itemize}
A point $x_0 \in \partial \Omega$ is called \emph{$F$-regular at height $\lambda$} if it admits $F$-barriers at $\lambda$, and $F$-regular if it is $F$-regular at each height. The boundary $\partial \Omega$ is said to be $F$-regular (at height $\lambda$) if it is so for each of its points.
\end{definition}

\begin{remark}
\emph{By $(N)$, if $x_0$ is $F$-regular at height $\lambda$, then it is $F$-regular at each height $\lambda'<\lambda$.
}
\end{remark}

A sufficient condition on $\partial \Omega$ to be $F$-regular is its $F$-convexity, which is defined as follows: we first consider the reduced bundle 
$$
J^ 2_\red(X) \ra X, \qquad J^2_\red(X) \doteq T^*X \oplus \Sym^2(T^*X), 
$$
so that $J^2 (X) = \mathbf{R} \oplus J^2_\red(X)$. A subequation $F$ is called \emph{reduced} (i.e. independent of $r$) if $F = \mathbf{R} \oplus F'$ for some $F'\subset J^2_\red(X)$. Hereafter, a reduced subequation will directly be identified with $F \subset J^2_\red(X)$. 

\begin{definition}
Let $F\subset J^{2}_\red(X)$ be a reduced subequation. The \emph{asymptotic interior} $\overrightarrow{F}$ of $F$ is the set of all $J\in J^{2}_\red(X)$ for which there exist a neighborhood $\mathcal{N}(J) \subset J^{2}_\red(X)$ and a number $t_{0}>0$ such that
\begin{equation*}
t\cdot \mathcal{N}(J) \subset F \quad \forall \, t\geq t_0 .
\end{equation*}
A $C^2$-function $u$ with $J^{2}_{x}u \in \overrightarrow{F}$ for all $x$ will be called $\overrightarrow{F}$-subharmonic.
\end{definition}

By its very definition, $\overrightarrow{F}$ is an open cone satisfying $(P),(N)$.
%
%
The $F$-convexity of $\partial\Omega$ is defined in terms of possessing defining functions which are $\overrightarrow{F}$-subharmonic. We recall that a defining function for $\partial \Omega$ is a $C^2$ function $\rho$ defined in a neighborhood of $\partial\Omega$ such that 
$$ 
\partial\Omega = \{x : \rho(x) = 0\}, \qquad \dif\rho \not= 0 \ \ \mbox{on} \ \ \partial\Omega, \qquad \rho < 0 \ \ \mbox{on} \ \ \Omega . 
$$

\begin{definition}
Let $F \subset J_\red^2(X)$ be a reduced subequation, and let $\Omega \subset X$ be a smooth open set. Then $\partial\Omega$ is named \emph{$F$-convex at $x_0\in \partial\Omega$} if there exists a defining function $\rho$ with $J^2_{x_0}\rho \in \overrightarrow{F_{x_0}}$. The boundary $\partial\Omega$ is said to be \emph{$F$-convex} if this holds at every $x_0\in \partial\Omega$.
\end{definition}

\begin{remark}\label{rem_equiconve}
\emph{Let $\rho$ be a defining function for $\partial \Omega$ satisfying $J^2_{x_0}\rho \in \overrightarrow{F_{x_0}}$. By definition, using the continuity of $J^2\rho$ and the topological condition $\inte F_{x_0} = (\inte F)_{x_0}$, $t\rho$ is $F$-subharmonic in a neighborhood of $x_0$ for each $t\ge t_{x_0}$ large enough. Defining functions for $\partial \Omega$ which satisfy $J^2_{x_0}\rho \in \overrightarrow{F_{x_0}}$ can be found starting from any defining function $\bar \rho$. Indeed, by \cite[Prop. 11.6]{HL_dir}, $\partial \Omega$ is $F$-convex at $x_0$ if and only if, for each defining function $\bar \rho$ and setting $\rho_s \doteq \bar \rho + s \bar \rho^2$, $J^2_{x_0}\rho_s \in \overrightarrow{F_{x_0}}$ for sufficiently large $s$. 
}
\end{remark}

%
%

Geometrically, the $F$-convexity can be described as follows: if we choose an outward pointing normal vector $\nu$ to $\partial \Omega$ and a Fermi chart
$$
\varphi \ \ : \ \ \partial \Omega \times (-\eps,\eps) \ra B_\eps(\partial \Omega), \qquad (x,\rho) \mapsto \exp_x\big( \rho \nu(x)\big),
$$
then the coordinate $\rho$ gives a defining function for $\partial \Omega$. If we denote with $\II_{\partial \Omega}$ the second fundamental form of $\partial \Omega$ in the direction of $-\nu$:
$$
\II_{\partial \Omega}(X,Y) \doteq \langle \nabla_X \nu,Y \rangle \qquad \text{for } \, X,Y \in T\partial \Omega.
$$
With respect to the splitting $T_{x_0}X =  \langle \nu \rangle \oplus T_{x_0} \partial \Omega$ at $x_0 \in \partial \Omega$, in view of Remark \ref{rem_equiconve} the $F$-convexity condition rewrites as 
\begin{equation}\label{convegeometric}
\left( 0, (1,0), \left( \begin{array}{cc}s & 0 \\ 0 & \II_{\partial \Omega}\end{array}\right) \right) \in \overrightarrow{F_{x_0}} \quad \text{for $s$ large enough.}
\end{equation}

In order to extend the previous definition from reduced subequations to the general case we are going to associate for every subequation $F \subset J^{2}(X)$ a family of reduced subequations $F_{\lambda} \subset J^{2}_\red(X), \lambda \in \mathbb{R}$ obtained by freezing the $r$-coordinate, that is, 
$$ 
\{\lambda\}\oplus F_{\lambda} = F \cap \Big\{\{\lambda\}\oplus J^{2}_{\red}(X)\Big\}. 
$$
Note that, by $(N)$, $F_\lambda \subset F_{\lambda'}$ whenever $\lambda' \le \lambda$.

\begin{definition}
Let $F \subset J^{2}(X)$ be a subequation, and let $\Omega \subset X$ be a domain with smooth boundary. We will say that $\partial\Omega$ is \emph{$F$-convex at height $\lambda$} at a point $x_0\in \partial\Omega$ if $\partial\Omega$ is $F_{\lambda}$-convex at $x_0$. If this holds for each $\lambda \in \R$, then $\partial \Omega$ is called $F$-convex at $x_0$. The boundary $\partial\Omega$ is called \emph{$F$-convex} (at height $\lambda$) if it is so at each $x_0 \in \partial \Omega$.
\end{definition}

\begin{example}
We consider the quasilinear subequation in $(\EE 7)$. Barriers in this case are well known, and the reader is referred to the comprehensive \cite{serrin_barriers}. However, by a way of example, we can apply the definitions above to conclude that $\partial \Omega$ if $F$-convex at height $\lambda$ at $x_0$ if and only if there exists $s$ large such that, for each $t \ge t_0(s)$,
\begin{equation}\label{condiquasil}
\lambda_1(t)s + \lambda_2(t)\tr\big(\II_{\partial \Omega}\big) > f(\lambda).
\end{equation}
For instance, \eqref{condiquasil} is always satisfied when $f(\lambda) \le 0$ and $\partial \Omega$ is strictly mean convex at $x_0$.
\end{example}

\begin{example}
For the $\infty$-Laplacian $F_\infty = \overline{\{A(p,p)>0\}}$, each smooth boundary is $F_\infty$-convex.
\end{example}

For our examples, we need the following simple result.

\begin{proposition}\label{prop_condition_F_convex}
Let $\mathbb{F}$ be one of the examples in $(\EE 2), \ldots, (\EE 8)$, and suppose that $f(r) \le 0$ for $r \le 0$. Then, each Euclidean ball of $\R^ m$ is $\mathbb{F}$-convex at non-positive heights. 
\end{proposition}

\begin{proof}
It is easy to check that $\{\lambda_1(A) \ge 0\}$ is contained in each of the subequations above. Hence, 
$$
\big\{ \lambda_1(A)>0 \big\} = \overrightarrow{\big\{ \lambda_1(A)>0 \big\}} \subset \overrightarrow{F_0}.
$$
Taking an Euclidean ball $B_R(o)$, an $\overrightarrow{F_0}$-subharmonic defining function at each point of $\partial B_R(o)$ can be chosen to be $\rho \doteq r_o^2 - R^2$.
\end{proof}
%
%
The next result guarantees the existence of barriers when $\partial \Omega$ is $F$-convex.

\begin{theorem}\cite[Thm. 11.12]{HL_dir}\label{teo_exbarriera}
If $\partial \Omega$ is $F$-convex at height $\lambda$ at some $x_0 \in \partial\Omega$, then $x_0$ is $F$-regular at each height $\lambda'\leq\lambda$.
\end{theorem}

%

In what follows, we also need the following modification of Corollary~11.8 in \cite{HL_dir} to guarantee a global barrier for all of $\partial \Omega$. 

\begin{proposition}\label{prop_F_convex_global}
Let $\partial \Omega$ be $F$-convex at height $0$ and fix a global defining function $\rho$ for $\partial \Omega$. Then, if $s$ is large enough, $\rho_s \doteq \rho + s \rho^2$ is negative on a neighborhood $U$ of $\partial \Omega$ and there it satisfies the following property: $t \rho_s$ is strictly $F$-subharmonic in $U\cap \Omega$ for each $t \ge t_0$, $t_0$ large enough. Consequently, the functions $\{\beta_t\}$, $\beta_t \doteq t\rho_s$ give rise to an $F$-barrier at height $0$ for $\partial \Omega$.
\end{proposition}

\begin{proof}
As $\partial \Omega$ is $F_0$-convex by definition, by Remark \ref{rem_equiconve} there exist, for each $x \in \partial \Omega$, a neighborhood $U_x$ of $x$ and $t_x>0$ such that $t\rho_s \in F_0^\str(U_x)$ whenever $t \ge t_x$. Extracting from $\{U_x\}$ a finite subcollection covering $\partial \Omega$, we deduce the existence of a uniform $t_0$ such that $t \rho_s$ is strictly $F_0$-subharmonic in a neighborhood $U$ of $\partial \Omega$ when $t \ge t_0$. We shall just prove that $t \rho_s \in F^\str(U \cap \Omega)$. To this aim, fix $x_0 \in U \cap \Omega$ and denote with $\bar J^2\rho_s$ the reduced $2$-jet of $\rho_s$. By the strict $F_0$-subharmonicity, there exists a neighborhood $W \subset F_0$ of $\bar J^2(t\rho_s)$, which we can take of the form $W = \big\{ y \in B\big\} \times V$. Moreover, from the definition of $F_0$, $\{0\} \oplus W \subset F$. Up to reducing $B$, we can suppose that $B \Subset U \cap \Omega$ and $\rho_s < \rho_s(x_0)/2 <0$ on $\overline{B}$. Hence, by $(N)$, the open set $\{r < \rho(x_0)/2\} \times W$ is a neighborhood of $J^2(t \rho_s)$ contained in $F$, showing that $t\rho_s \in F^\str(U\cap \Omega)$.
\end{proof}

\section{The obstacle problem}\label{sec_obstacle}

Let $F \subset J^{2}(X)$ be a subequation, fix a relatively compact open set $\Omega$ and choose $g \in C(\overline \Omega)$. We consider the subequation $G^g = \{(x,r,p,A) : r \le g(x)\}$, and define the subset
$$
F^g \doteq F \cap G^g.
$$
It is easy to see that $F^g$ is a subequation. For fixed $\varphi \in C(\partial \Omega)$, a function $u \in \USC(\overline{\Omega})$ is said to solve the \emph{obstacle problem} on $\Omega$ with obstacle function $g$ and boundary data $\varphi$ if:
\begin{itemize}
\item[-] $u$ is $F^g$-harmonic on $\Omega$, and
\item[-] $u=\varphi$ on $\partial \Omega$.
\end{itemize}

As for the Dirichlet problem in \cite{HL_dir}, the solvability of the obstacle problem depends on the $F$ and $\widetilde{F}$-convexity of $\partial \Omega$, and on the validity of the comparison for $F^g$. The particular case of reduced subequations $F = \mathbf{R} \oplus F_0$, with $F_0 \subset J^{2}_{\red}(X)$, was handled in \cite{HL_existence} via jet-equivalence (see \cite{HL_P} for more details). However, to treat the general case we need to proceed a bit differently.\par

We first address the validity of weak comparison and strict approximation properties for $F^g$. 

\begin{lemma}\label{wcobstacle}
Let $F \subset J^{2}(X)$ be a subequation, fix $g \in C(X)$ and denote with $F^{g}$ the associated obstacle subequation. Fix a relatively compact open set $\Omega \subset X$. Then
$$
\begin{array}{ll}
(1) &  \disp \text{$F$ satisfies the strict approximation} \ \Longrightarrow \ \text{$F^{g}$ satisfies the strict approximation.} \\[0.2cm]
(2) &  \disp \text{$F$ satisfies the weak comparison} \ \Longrightarrow \ \text{$F^{g}$ satisfies the weak comparison.}
\end{array}
$$
\end{lemma}

\begin{proof}
The first statement is obvious: take a relatively compact set $\Omega \subset X$ and a function $u \in F^g(\overline \Omega) \subset F(\overline\Omega)$. By the strict approximation, choose $\{u_k\} \subset F^\str(\Omega)$ that approximates uniformly $u$ on $\Omega$, say with $\|u_k-u\|_\infty \le 1/k$. Then, the sequence $\{\bar u_k\}$, $\bar u_k \doteq u_k - 2/k$ satisfies 
$$
\bar u_k \in (F^g)^\str(\overline{\Omega}), \qquad \bar u_k \ra u \quad \text{uniformly on $\Omega$},
$$
showing the strict approximation for $F^g$.\par
To prove $(2)$ suppose by contradiction that $F^g$ does not satisfy the weak comparison. Then, by Theorem \ref{lwcimplieswc}, we can take a domain $U \subset X$, $c>0$, and functions $u \in (F^{g})^{c}(U)$, $v\in \widetilde{F^{g}}(U)$ with $u+v \leq 0$ on $\partial U$ but $(u+v)(x_{0}) = \max_{U}(u+v) \doteq \delta > 0$. Note that 
\begin{equation}\label{propriedades_uv}
\widetilde{F^{g}} = \widetilde{F}\cup \{r \leq - g(x)\} \qquad \text{and} \qquad v(x_0)> - u(x_0) > - g(x_0), 
\end{equation}
hence $2$-jets of test functions $\phi$ for $v$ at $x_0$ satisfy $J^ 2_{x_0} \phi \in \widetilde F_{x_0}$. However, since $\{v>-g\}$ is not necessarily open, we cannot conclude by directly applying the weak comparison for $F$. To overcome the problem, we first observe that \eqref{propriedades_uv} implies $F^g_{x_0} \not\equiv G^g_{x_0}$, and thus $(\inte G^g_{x_0}) \backslash F^g_{x_0} \neq \emptyset$ by $(T)$. Pick a $C^2$ function $f$ around $x_0$ such that
$$
J^2_{x_0}(-f) \in \inte G^g_{x_0} \backslash F_{x_0} = (\inte G^g \backslash F)_{x_0}. 
$$
By continuity, on some small ball $B = B_R(x_0) \subset U$, 
\begin{align*}
	f \in C^2(\overline B), \quad &J^2_x(-f) \in \inte G^g_x = \big \{r<g(x)\big\}, \\
	&J^2_x(-f) \not \in \inte F_x \quad \text{for each } \, x \in B. 
\end{align*}
In other words, $f \in \widetilde F(B) \subset \widetilde{F^g}(B)$ and $f>-g$ on $B$. Now, fix $\varepsilon>0$ satisfying 
\begin{equation}\label{defepsilon111}
0 < 2\varepsilon < \min\left\{\frac{\delta}{4}, v(x_{0})+g(x_{0}), f(x_{0})+g(x_{0})\right\},
\end{equation}
and observe that this choice implies $\bar{f}(x) := f(x) - (f(x_{0})+g(x_{0}) - \varepsilon) \in \widetilde{F}(B)$. Moreover, being $u \in F^g$, 
\begin{equation}\label{vebarf}
-g(x_0) < -g(x_0) + \varepsilon = \bar{f}(x_0) < v(x_0) - \varepsilon .
\end{equation}
Shrinking $B$, if necessary, from \eqref{defepsilon111} we can suppose
\begin{equation}\label{eqfbar}
0 < \bar{f}+g < 2\varepsilon \ \ \mbox{on} \ \ \overline{B}.
\end{equation} 
By the stability property in Proposition \ref{prop_basicheF}, fix $\eta > 0$ small enough that
\begin{equation*}
\bar{u} := u - \eta \varrho_{x_0}^2 \in (F^{g})^{\frac{c}{2}}(\overline{B}) \subset F^{\frac{c}{2}}(\overline{B}),
\end{equation*}
and note that $x_0$ is the unique global maximum of $\bar u + v$ on $\overline{B}$. Set $\hat{v}(x) \doteq \max \{v,\bar{f}\}$, and observe that $\hat v \in \widetilde{F^g}(B)$ because of (1) in Proposition \ref{prop_basicheF}. In fact, because of \eqref{eqfbar} the stronger $\hat{v} \in \widetilde{F}(B)$ holds, and \eqref{vebarf} implies  $$
(\bar{u}+\hat{v})(x_0) = (u+v)(x_0) = \delta > 0. 
$$
However, on $\partial B$,
\begin{align*}
\disp \bar{u}+\hat{v} & =  \max\{\bar u +v,\bar{u}+\bar{f}\} & \\[0.1cm]
&\leq  \disp \max\{u+v -\eta R^2 , u - \eta R^2 + 2\varepsilon - g\} & \quad (\mbox{by \eqref{eqfbar}}) \\[0.1cm]
&\leq  \disp \max\{\delta -\eta R^2 ,u - \eta R^2 + 2\varepsilon - g\} & \\[0.1cm]
&\leq  \disp \max\{\delta -\eta R^2 ,2\varepsilon -\eta R^2\} & \quad (\mbox{since} \ \ u \leq g) \\[0.1cm]
&= \disp \delta - \eta R^2 & \quad (\mbox{since} \ \ \varepsilon < \delta/4).
\end{align*}
Concluding, the functions $\bar{u} \in F^{c/2}(B)$ and $\bar v \doteq  \hat{v}- (\delta - \eta R^2)_+ \in \widetilde{F}(B)$ contradict the validity of the weak comparison for $F$ on $B$, as we have
$$
\begin{aligned}
&(\bar u + \bar v)(x_0) = \delta - (\delta -\eta R^2)_+ > 0, \\
&\bar u + \bar v \le \delta - \eta R^2 - (\delta -\eta R^2)_+ \le 0 \ \ \text{on } \, \partial B.
\end{aligned}\vspace{-2em}
$$\end{proof}

We are ready to solve the obstacle problem for $F^g$. Our treatment closely follows the lines in \cite{HL_dir} (see also \cite{cirantpayne}), and relies on Perron's method. Given a boundary function $\varphi \in C(\partial\Omega)$, we consider the Perron's family
\begin{equation*}\label{perronclass} 
\mathcal{F}^{g}(\varphi) = \big\{v\in \USC(\overline{\Omega}) \  : \  v_{|_{\Omega}} \in F^{g}(\Omega) \ \ \mbox{and} \ \ v_{|_{\partial\Omega}}\leq \varphi\big\} ,
\end{equation*}
and the Perron's function
$$ 
u(x) = \sup\big\{v(x): v\in \mathcal{F}^{g}(\varphi)\big\}. 
$$
To ensure that the Perron's class is non-empty, the next basic requirement is sufficient:
\begin{equation}\label{ipo_basicaperron}
\text{there exist $\psi \in F(\overline \Omega)$, $\widetilde \psi \in \widetilde{F}^\str(\overline{\Omega})$ which are bounded from below.}
\end{equation}

\begin{remark}
\emph{Observe the asymmetric role of $\psi$ and $\widetilde \psi$ in \eqref{ipo_basicaperron}, which is due to the asymmetric role of $F, \widetilde F$ in Theorem \ref{obstaclethm} below. 
}
\end{remark}

Given $u : \overline\Omega \ra \R$, we recall that the definition of USC (respectively, LSC) regularization $u^*$ (resp. $u_*$) is given in \eqref{def_USCLSC}.
 

\begin{theorem}\label{obstaclethm}
Let $F \subset J^2(X)$ be a subequation, fix $g \in C(X)$ and a smooth, relatively compact open set $\Omega \subset X$. Suppose that \eqref{ipo_basicaperron} is met, and that $\partial\Omega$ is both $F$ and $\widetilde{F}$ convex. Let $\varphi \in C(\partial \Omega)$ satisfying $\varphi \le g$ on $\partial \Omega$. Then, if $\widetilde F$ satisfies the weak comparison on $\Omega$, the Perron's function $u$ has the following properties:
\begin{enumerate}
\item[1)] $u_* = u^* = \varphi$ on $\partial\Omega$,
\item[2)] $u \equiv u^* \in F^g(\Omega)$,
\item[3)] $(-u)^* \in \widetilde{F^{g}}(\Omega)$.
\end{enumerate}
Moreover, 
\begin{itemize}
\item[4)] if $F$ satisfies both the weak comparison and the strict approximation, in \eqref{ipo_basicaperron}  we can just assume $\psi \in F(\overline \Omega), \widetilde \psi \in \widetilde F(\overline{\Omega})$, and $u$ is the unique solution of the obstacle problem for $F^g$ on $\Omega$.
\end{itemize}
\end{theorem}


\begin{proof}
First, we consider boundary barriers:

\begin{lemma}\label{lem_FetildeF}
Let $x_0 \in \partial \Omega$.
\begin{itemize}
\item[$(LF)$] If $\partial \Omega$ is $F$-convex at $x_0$, for each $\delta>0$ small enough there exists 
\begin{equation}\label{propertyunideru}
\underline{u} \in \mathcal{F}^g(\varphi), \qquad \underline{u} \ \text{continuous at } \, x_0, \qquad \underline{u}(x_0) = \varphi(x_0)-\delta.
\end{equation}
\item[$(L\widetilde F)$] If $\partial \Omega$ is $\widetilde F$-convex at $x_0$, for each $\delta>0$ small enough there exists 
\begin{equation*}\label{propertyoveru}
\overline{u} \in \widetilde{\mathcal{F}}(-\varphi) \cap \widetilde F^\str(\overline\Omega), \qquad \overline{u} \ \text{continuous at } \, x_0, \qquad \overline{u}(x_0) = -\varphi(x_0)-\delta.
\end{equation*}
\end{itemize}
\end{lemma}


\begin{proof}
We first prove $(LF)$. From $\varphi \le g$, having fixed $\delta>0$ we can choose a small ball $B$ around $x_0$ such that $\varphi(x_0) -\delta < \varphi \le g$ on $\overline{B}$. By Theorem~\ref{teo_exbarriera}, $x_0$ is regular at height $\varphi(x_0)-\delta$, so we pick a barrier $\{\beta_c\}$ defined in a neighborhood $U \subset B$ of $x_0$. Note that, by construction, $\beta_c \le \varphi(x_0)-\delta < \varphi \le g$, hence $\beta_c \in (F^g)^ \str(U \cap \Omega)$. Since $\psi \in \USC(\overline{\Omega})$ has a finite maximum, up to translating $\psi$ downwards we can assume that 
\begin{equation*}\label{propfg2}
\psi < g \ \text{ on $\overline{\Omega}$,} \ \ \psi < \varphi - \delta \ \text{ on $\ \partial\Omega$,} \ \text{ i.e.,} \ \  \psi \in \mathcal{F}^g(\varphi).
\end{equation*}
Using that $\psi$ is bounded from below, we can choose $c$ large enough that $\beta_c < \psi$ in a neighborhood of $\partial U$. Then, the function 
$$
\underline{u} \doteq  \begin{cases}
\max\{ \beta_c, \psi\} & \quad \text{on } \, U \cap \overline{\Omega}, \\
\psi & \quad \text{on } \, \Omega \backslash \overline{U},
\end{cases}
$$
satisfies all the properties in \eqref{propertyunideru} (note that the open set $\{\psi < \beta_c\}$ contains $x_0$, thus $\underline{u} \equiv \beta$ around $x_0$). The proof of $(L\widetilde{F})$ is analogous, and indeed it is exactly Proposition $\widetilde{F}$ in \cite{HL_dir}. We remark that property $\overline{u} \in \widetilde F^\str(\overline{\Omega})$ follows because $\widetilde{\Psi} \in \widetilde{F}^\str(\Omega)$ in \eqref{ipo_basicaperron}.
\end{proof}

\begin{corollary}\label{cor_FetildeF}
Let $x_0 \in \partial \Omega$.
\begin{itemize}
\item[-] If $\partial \Omega$ is $F$-convex at $x_0$, then $u_*(x_0) \ge \varphi(x_0)$.
\item[-] If $\partial \Omega$ is $\widetilde{F}$-convex at $x_0$, then $u^*(x_0) \le \varphi(x_0)$.
\end{itemize}
\end{corollary}

\begin{proof}
To prove the first claim, $(LF)$ in Lemma \ref{lem_FetildeF} and the definition of Perron's envelope give
$$
u_*(x_0) = \liminf_{x \ra x_0} u(x) \ge \liminf_{x \ra x_0} \underline{u}(x_0) = \varphi(x_0)-\delta,
$$
and the sought follows by letting $\delta \downarrow 0$. Regarding the second claim, we first compare any $v \in \mathcal{F}^g(\varphi)$ with the function $\overline{u}$ in Lemma \ref{lem_FetildeF}: 
\begin{equation}\label{impopoint}
\overline{u} \in \widetilde{F}^\str(\overline{\Omega}), \qquad v \in F^g(\overline{\Omega}) \subset F(\overline \Omega), \qquad \overline{u} + v \le 0 \ \  \text{ on } \, \partial\Omega. 
\end{equation}
Since $\widetilde F$ satisfies the weak comparison, $v + \overline{u} \le 0$ on $\overline{\Omega}$. Taking supremum over $v$, $u \le - \overline{u}$ on $\Omega$, and because of the continuity of $\overline{u}$ at $x_0$ we deduce
$$
u^*(x_0) \doteq \limsup_{x \ra x_0} u(x) \le - \limsup_{x \ra x_0} \overline{u}(x) = \varphi(x_0)+ \delta.
$$
The desired estimate follows again by letting $\delta \downarrow 0$.
\end{proof}

\noindent \textbf{Proof of $1)$ of Theorem \ref{obstaclethm}}. It follows directly by Corollary \ref{cor_FetildeF}: $u^*(x_0) \le \varphi(x_0) \le u_*(x_0)$. \\[0.2cm]
\noindent \textbf{Proof of $2)$}. First, since $v \leq g \in C(\Omega)$ for all $v \in \mathcal{F}^{g}(\varphi)$, the family $\mathcal{F}^{g}(\varphi)$ is locally uniformly bounded above, and hence $u^* \in F^g(\overline{\Omega})$ by (4) in Proposition \ref{prop_basicheF}. Coupling with $1)$, $u^* \in \mathcal{F}^g(\varphi)$ and thus $u^* \le u$ by definition. Therefore, $u \equiv u^*$.\\[0.2cm]
\noindent \textbf{Proof of $3)$}. Suppose, by contradiction, that $(-u)^* = -u_*\notin \widetilde{F^{g}}(\Omega)$, and choose $x_{0} \in \Omega$, $\varepsilon >0$ and a test function $\phi \in C^{2}$ near $x_0$ such that
\begin{equation}\label{eqFg1}
\begin{cases}
\disp \phi \geq -u_* + \varepsilon \varrho_{x_0}^2 \quad \text{on } \, B = B_R(x_0) \Subset \Omega, \\
\phi(x_0) = -u_*(x_0),
\end{cases} \quad \text{but } \quad J^{2}_{x_{0}}\phi \notin \widetilde{F^{g}}_{x_{0}}.
\end{equation}
This means that $-J^{2}_{x_{0}}\phi \in \inte {F^g}_{x_{0}}$. Using $(T)$ and the continuity of $J_x^2\phi$, up to reducing $R$ there exists $\delta_0>0$ small enough that $-\phi + \delta \in (F^{g})^\str(B)$ for each $\delta<\delta_0$. Choosing $\delta < \varepsilon R^2$, \eqref{eqFg1} implies that 
\begin{equation*}
(-\phi + \delta)(x_0) = u_*(x_0) + \delta \quad \text{ and } \ -\phi + \delta < u_* \quad \text{near } \, \partial B. 
\end{equation*}
Let $x_k \ra x_0$ satisfying $u(x_k) \ra u_*(x_0)$. Then, for $k$ large, 
\begin{equation*}
\qquad (-\phi + \delta)(x_k) > u(x_k) + \frac{\delta}{2}, \qquad \text{while } \ -\phi + \delta < u \quad \text{near } \, \partial B. 
\end{equation*}
On the other hand, from $1)$ and $2)$ we deduce that the function
\begin{equation*}
u' \doteq 
\begin{cases}
	u & \mbox{on} \ \ \overline{\Omega}\backslash B, \\
	\max\big\{u,-\phi + \delta\big\}  &\mbox{on} \ \  \overline B ,
\end{cases}
\end{equation*}
satisfies $u' \in F^g(\overline\Omega)$ and $u' = \varphi$ on $\partial \Omega$. Hence, $u' \in \mathcal{F}^g(\varphi)$ and thus $u' \le u$. This contradicts the inequality $u'(x_k)= -\phi(x_k)+\delta > u(x_k)$.\\[0.2cm]
\noindent \textbf{Proof of $4)$}. If comparison holds for $F$, it holds for $\widetilde F$. The only point where we used the weak comparison for $\widetilde F$ in items $1),2),3)$ is to conclude $u + \overline{u} \le 0$ from \eqref{impopoint}. Under the validity of the full comparison for $\widetilde F$, the same conclusion can be reached even when $\overline{u} \in \widetilde F(\overline \Omega)$, which is granted under the weaker requirement $\widetilde \psi \in \widetilde F(\overline \Omega)$ in \eqref{ipo_basicaperron}. Furthermore, since $F$ satisfies both the weak comparison and the strict approximation properties, by Lemma \ref{wcobstacle} so does $F^g$, hence $F^g$ satisfies the comparison because of Theorem \ref{localwcestrict}. Uniqueness then follows immediately: if $u,v$ solve the obstacle problem, 
$$
u \in F^g(\overline \Omega), \qquad -v \in \widetilde{F^g}(\overline\Omega), \qquad u -v = \varphi-\varphi = 0 \ \ \text{ on } \, \partial \Omega,
$$
hence $u \le v$ on $\Omega$ by comparison. Reversing the role of $u,v$ we get $u\equiv v$.
\end{proof}

\section{Ahlfors, Liouville and Khas'minskii properties}\label{sec_teoprinci}

This section is devoted to prove the AK-duality between the Ahlfors, Khas'minskii and weak Khas'minskii properties stated in Definitions~\ref{def_ahlfors}, \ref{def_khasmi}, \ref{def_weakkhasmi}. We also show their relation with the Liouville  property in Definition \ref{def_Liouville}.

Our first main result concerns subequations $F$ on $X$ which are locally jet-equivalent to a universal one $\mathbb{F}$, for which we require the next conditions:
\begin{itemize}
\item[$(\HH 1)$] $\quad$ negative constants are strictly $\mathbb{F}$-subharmonics; 
\item[$(\HH 2)$] $\quad$ there exists an Euclidean ball $\mathbb{B}$ which is $\mathbb{F}$-convex at height $0$; 
\item[$(\HH 3)$] $\quad$ $F$ has the bounded, strict approximation on $X$;
\item[$(\HH 4)$] $\quad$ the function $u \equiv 0$ is $\mathbb{\widetilde{F}}$-subharmonic.
\end{itemize}

We first investigate when a subequation which is locally jet-equivalent to any of the examples in $(\EE 2), \ldots, (\EE 8)$ satisfy $(\HH 1),\ldots, (\HH 4)$. Each of the examples is described by a subequation of type 
\begin{equation}\label{cond_F_f_model}
\mathbb{F}_f \doteq \overline{\big\{\mathscr{F}(p,A) > f(r) \big\}}, 
\end{equation}
for some $\mathscr{F} : \mathbf{J^2} \ra \R$ and for $f \in C(\R)$ non-decreasing. Furthermore, $\mathscr{F}(0,0)=0$.\par
\begin{itemize}
\item Assumptions $(\HH 1+\HH 4)$ are equivalent to $f$ satisfy $(f1)$ in \eqref{def_f1xi1}. 
\item $(\HH 2)$ holds for examples $(\EE 2), \ldots, (\EE 8)$ provided that $f(r) \le 0$ for $r<0$, as a consequence of Proposition \ref{prop_condition_F_convex}. 
\item $(\HH 3)$ holds for examples $(\EE 2), \ldots, (\EE 6)$ provided that $f$ is strictly increasing on $\R$ and $\mathscr{F}$ is uniformly continuous, by Theorem \ref{thm_comparison_examples}. If $f$ is not strictly increasing, the strict approximation is much more delicate: for a counterexample to $(\HH 3)$ in $(\EE 5)$, see Example 12.8 in~\cite{HL_dir}.
\end{itemize}
We begin with the next simple lemma.

\begin{lemma}\label{lem_maxprinc}
Let $Y$ be a Riemannian manifold, and let $F\subset J^2(Y)$ be a subequation. If negative constants are in $F^\str$, then $\widetilde{F^0}$ satisfies the maximum principle: functions $u \in \widetilde{F^0}(Y)$ cannot achieve a local positive maximum.
\end{lemma}

\begin{proof}
At any local maximum point $x_0 \in Y$ with $u(x_0)=c>0$, using $\phi\equiv c$ as a test function we would get $J^2_{x_0}\phi \in (\widetilde{F^0})_{x_0} = \widetilde{F}_{x_0} \cup \{r \le 0\}$. Being $c>0$, necessarily $J^2_{x_0}\phi \in \widetilde F_{x_0}$, that is, $-J^2_{x_0}\phi \not\in \inte F_{x_0}$, contradicting the assumption $-\phi \in F^\str(Y)$.
\end{proof}

Seeking to clarify the role of each assumption for the AK-duality, we first describe the interplay between the Liouville and Ahlfors properties.

\begin{proposition}\label{teo_main_A_L}
Let $F \subset J^{2}(X)$ be a subequation. Then,
$$
\widetilde{F} \ \ \text{has the Ahlfors property} \ \ (A) \quad \Longrightarrow \quad \widetilde{F} \ \ \text{has the Liouville property} \ \ (L).
$$
If further $0 \in \widetilde{F}(X)$, then the two properties are equivalent.
\end{proposition}
\begin{proof}
$(A) \Rightarrow (L)$.\\
Suppose, by contradiction, that there exists a bounded, non-negative and non-constant function $u \in \widetilde{F}(X)$, and fix a compact set $K$ such that $\max_K u < \sup_X u$. Then, $u$ would contradict the Ahlfors property on $U \doteq X \backslash K$.\\[0.2cm] 
\noindent $(L) \Rightarrow (A)$, if $0 \in \widetilde{F}(X)$.\\
If, by contradiction, the Ahlfors property fails to hold, we can find $U \subset X$ with $\partial U \neq \emptyset$, and $u \in \widetilde{F^0}(\overline{U})$ satisfying 
$$
\sup_{\partial U} u^+ < \sup_{\overline{U}}u - 2\eps,
$$
for some $\eps>0$ (note that $H \doteq \widetilde{F} \cup \{r \le 0\} = \widetilde{ F \cap \{r \le 0\}} \doteq \widetilde{F^0}$). The sublevel set $\{u<\sup_{\partial U}u^+ +\eps\}$ is an open neighborhood of $\partial U$, thus by $(N)$ and $0\in \widetilde{F}(X)$, the function 
$$
v = \max\Big\{ u- \sup_{\partial U} u^+ - \eps, 0\Big\}
$$
satisfies $v \in \widetilde{F}(\overline{U})$, $v \ge 0$, $v \not\equiv 0$ and $v \equiv 0$ in a neighborhood of $\partial U$. Extending $v$ with zero on $X \backslash U$ would produce a non-constant, non-negative, bounded $\widetilde{F}$-subharmonic function on $X$, contradicting the Liouville property.
\end{proof}

We are ready for our main theorem.

\begin{theorem}\label{teo_main}
Let $F \subset J^{2}(X)$ be locally jet-equivalent to a universal subequation $\mathbb{F}$ via locally Lipschitz bundle maps. Assume $(\HH 1), (\HH 2), (\HH 3)$. Then, AK-duality holds for $F$,i.e.,
\begin{equation*}
\begin{array}{c}
F \ \mbox{ satisfies $(K)$} \\
(\mbox{Khas'minskii prop.})
\end{array}
 \Longleftrightarrow  
\begin{array}{c}
F \ \mbox{ satisfies $(K_\weak)$} \\
(\mbox{weak Khas'minskii prop.})
\end{array}
 \Longleftrightarrow  
\begin{array}{c}
\widetilde{F} \ \mbox{ satisfies $(A)$} \\
(\mbox{Ahlfors prop.}) 
\end{array}
\end{equation*}
\end{theorem}


\begin{proof}
$(K) \Rightarrow (K_\weak)$. Obvious.\\[0.2cm]
$(K_\weak) \Rightarrow (A)$. For this implication, we need properties $(\HH 1),(\HH 3)$.\\
Suppose, by contradiction, that there exist an open subset $U \subset X$ and a function $u \in \widetilde{F^0}(\overline{U})$ bounded from above and satisfying 
\begin{equation}\label{ipoteses}
\sup_{\partial U} u^+ < \sup_{\overline{U}} u  .
\end{equation}
By $(\HH 1)$ and Lemma \ref{lem_maxprinc}, $u_U \doteq \sup_{\overline U} u$ is not attained. Therefore, coupling with \eqref{ipoteses} and since $u \in \USC(\overline{U})$, for each relatively compact, open set $K$ intersecting $U$ it holds
$$
u_K \doteq \max\Big\{ \max_{\overline{U \cap K}} u, \sup_{\partial U} u^+\Big\} < u_U.
$$
Let now $K \subset X$ be the compact subset granted by $(K_\weak)$, and choose $\eps>0$ small enough that $4\eps < C$,  $C$ being the constant in $(K_\weak)$. Up to enlarging $K$, we can suppose $K \cap U \neq \emptyset$ and
$$
\mu \doteq u_U - u_K < 3\eps.
$$
Choose $x_0 \in U \backslash \overline{K}$ in such a way that 
\begin{equation}\label{u_mi_inequality}
u_{U} - \mu/2 < u(x_0).
\end{equation}
Consider the weak Khas'minskii potential $w$ of $(\mu/2,K, \{x_0\})$, and fix a relatively compact, open set $\Omega$ containing $\overline{K} \cup \{x_0\}$ and such that $w \le -C + \eps$ on $X \backslash \Omega$. Using also $w \le 0$, we deduce the following inequalities: 
$$
\begin{array}{ll}
\quad u-u_K + w \le \mu - (C - \eps) < 4\eps -C < 0, & \text{on } \, \overline{U} \cap \partial \Omega; \\[0.1cm]
\quad \disp u -u_K + w \le u - u_{K} \le 0, & \text{on } \, \partial K \cap \overline{U}; \\[0.1cm]
\quad \disp u -u_K + w \le u - u_K \le 0, & \text{on } \, \partial U \cap (\Omega \backslash K).  
\end{array}
$$
Clearly $w \in F^0(X \backslash K)$. We therefore compare
$$
u-u_K \in \widetilde{F^0}\big((\overline{\Omega}\backslash K) \cap \overline{U}\big) \qquad \text{and} \qquad w \in F^0\big((\overline{\Omega}\backslash K) \cap \overline{U}\big). 
$$
Since $F$ is jet-equivalent to a universal subequation via locally Lipschitz bundle maps, $F$ and $\widetilde F$ satisfy the weak comparison principle by Theorem \ref{teo_importante!!}, and because of property $(\HH 3)$, Theorem \ref{localwcestrict} and Lemma \ref{wcobstacle} we conclude the validity of the full comparison property for $F^0$. Hence
$$
u-u_K + w\le 0 \qquad \text{on } \, (\overline{\Omega}\backslash K) \cap \overline U. 
$$
However, since $w(x_0) \ge -\mu/2$, by \eqref{u_mi_inequality} we get 
$$
(u - u_{K} + w)(x_0) > u_U - \frac{\mu}{2} - u_{K} - \frac{\mu}{2} = u_U - \mu - u_K = 0, 
$$
a contradiction.\\[0.2cm]
\noindent $(A) \Rightarrow (K)$. For this implication, we need properties $(\HH 1),(\HH 2)$.\\
Fix a pair $(K,h)$, indices $i,j \in \mathbb{N}$, and a smooth exhaustion $\{D_j\}$ of $X$ with $K \subset D_1$. The idea is to produce $w$ as a monotone, locally uniform limit of a decreasing sequence $\{w_i\}$ of $F$-subharmonic functions, where $w_{i+1}$ is obtained from $w_i$ by means of a sequence $u_{j} \in F(D_j \backslash K)$ (also depending on $i$) of solutions of obstacle problems. The proof is divided into seven steps, and the Ahlfors property enters crucially in Steps 5 and 6 to guarantee that $w_{i+1}$ be close enough to $w_i$ on a large compact set, in order for the limit $w$ to be locally finite. The first problem to address is that, in general, $\partial K$ does not possess barriers to guarantee the solvability of obstacle type problems (for example, if $K$ is a small convex geodesic ball). For this reason, we first need to produce convex boundaries by suitably modifying $X$ and $F$ in a small neighborhood.\\[0.2cm] 
\noindent \textbf{Step 1: producing $\mathbb{F}$-convex boundaries.}\\[0.2cm] 
Let $t$ be the radius of $\mathbb{B}$ in $(\HH 2)$. Consider a compact manifold $M$ that contains an open set isometric to an Euclidean ball $\mathbb{B}_{2t}$ of radius $2t$. For instance, one can take $M = (\Sph^m, ( \, , \, ))$ with metric given, in polar coordinates centered at some $o \in \Sph^m$, by
$$
( \, , \, ) = \di r^2 + h(r)^2\di \theta^2, \qquad \text{with} \qquad h(r) = \begin{cases}
r & \ \ \text{ on } \, \big(0, 2t\big], \\
t\sin(r/t)  & \ \ \text{ on } \, \big[3t, \pi t\big).
\end{cases}
$$
Fix $x_0 \in K$, let $(O, \varphi)$ be a distinguished chart around $x_0$, and fix $R$ small enough that 
\begin{equation}\label{conditions_R} 
B_{5R} = B_{5R}(x_0) \Subset K \cap O, \qquad 10 R \le t.
\end{equation}
Consider a connected sum $\mathcal{X} \doteq X \sharp M$ along the annuli $B_{2R} \backslash B_R \subset X$ and $\mathbb{B}_{2R}\backslash \mathbb{B}_R \subset M$, obtained by identifying the point with polar coordinates $(r,\theta) \in B_{2R} \backslash B_R$ with the one of coordinates $(3R-r, \theta) \in \mathbb{B}_{2R}\backslash \mathbb{B}_R$. Choose a metric  on $\mathcal{X}$ that coincides with that of $X$ on $X\backslash B_{3R}$, and with $( \, , \, )$ on $M\backslash \mathbb{B}_{3R}$. Hereafter, we consider $X\backslash B_{3R}$ and $M\backslash \mathbb{B}_{3R}$ as being subsets of $\mathcal{X}$, and $\mathcal{D}_{j}, \mathcal{K}$ will denote the sets
$$
\cal D_{j} \doteq \cal X \backslash (X\backslash D_j) \quad \text{and}  \quad \cal K \doteq \cal X \backslash (X\backslash K). 
$$
Note that $\cal X\backslash \cal K$ and $ \cal X \backslash \cal D_j$ are isometric copies of, respectively, $X\backslash K$ and $X\backslash D_j$. Set $\cal V = M \backslash \mathbb B_t \subset \cal X$ and observe that, with the orientation pointing inside of $\cal V$, $\partial \cal V$ is $\mathbb{F}$-convex by $(\HH 2)$.\\[0.2cm] 
\noindent \textbf{Step 2: extending the subequation.}\\[0.2cm] 
By construction, $B_{5R}$ is contained in the domain of a distinguished chart $(O, \varphi)$. Let 
$$
\Phi^e \ \ : \ \ J^2(O) \longrightarrow \varphi(O) \times \mathbf{J^2}
$$
be the local trivialization of $J^2(X)$ induced by the chart $\varphi$, and let $\Psi : \varphi(O) \times \mathbf{J^2} \ra \varphi(O) \times \mathbf{J^2}$ 
be a jet-equivalence with $\Psi\big(\Phi^e(F)\big) = \varphi(O) \times \mathbf{F}$. Denote with $a_{ij}$ the coefficients of the metric of $X$ in the frame $e = \{\partial_j\}$. Consider a cut-off function $\eta \in C^\infty_c(B_{5R})$ satisfying
$$
0 \le \eta \le 1, \qquad \eta \equiv 1 \quad \text{on } \, B_{4R}, 
$$
and define a modified metric on $X$ which is the original one on $X \backslash O$, and whose components in the local frame $\{\partial_j\}$ on $O$ are 
$$
\bar a_{ij} = \eta \delta_{ij} + (1-\eta)a_{ij}.
$$
Note that $\bar a_{ij}$ is the Euclidean metric on $B_{4R}$. Denoting with $\bar J^2(O)$ the splitting given by the metric $\bar a_{ij}$, let us consider the corresponding trivialization $\overline \Phi^e$ induced by the chart $(O, \varphi)$ and the Levi-Civita connection of $\bar a_{ij}$. Consider also the jet-equivalence 
$$
\Psi_\eta \ : \ \varphi(O) \times \mathbf{J^2} \longrightarrow \varphi(O) \times \mathbf{J^2}, \qquad \Psi_\eta = \eta \mathrm{Id} + (1-\eta)\Psi,
$$
where $\mathrm{Id}$ denotes the identity map, and observe that $\Psi_\eta$ is the identity on $B_{4R}$. Define the subequation $H$ on $J^2(X)$ by setting
$$
H = \begin{cases}
F & \quad \text{outside of } \, B_{5R}, \\
(\overline\Phi^e)^{-1}\Big( \Psi_\eta^{-1}\big( \varphi(O) \times \mathbf{F}\big)\Big) & \quad \text{on } \, B_{5R}.
\end{cases}
$$
By construction, $H$ is a subequation, and 
\begin{equation*}\label{bellali}
H = (\overline\Phi^e)^{-1}\Big(\varphi(O) \times \mathbf{F}\Big) \qquad \text{on } \, B_{4R}.
\end{equation*}
Since $\{\partial_j\}$ is an orthonormal frame on $B_{4R}$ in the metric $\bar a_{ij}$, $H$ is a universal Riemannian subequation on $B_{4R}$. Therefore, once we have performed the gluing construction to produce $\cal X$, we can then extend $H$ on the entire $\cal X$ by declaring $H$ to be the universal Riemannian subequation with model $\mathbf{F}$ on $\cal X \backslash (X\backslash B_{3R})$.

In what follows, for the ease of notation, we still write $F$ instead of $H$ to denote the extension that we have just constructed on $\cal X$.\\[0.1cm]
\textbf{Claim 1:} Negative constants are in $F^\str(\cal X)$, and $\widetilde F$ satisfies $(A)$ on $\cal X$.\\[0.1cm]
\textit{Proof of Claim 1:} The first claim is an immediate consequence of the construction performed to produce $F$ on $\cal X$, since a jet-equivalence does preserve the $r$-coordinate of jets (this is exactly the point when we need a \emph{jet-equivalence}, not just an \emph{affine jet-equivalence}). Regarding the Ahlfors property, this follows from the finite maximum principle. Indeed, suppose by contradiction that there exist an open set $\cal U \subset \cal X$ with non-empty boundary, and $u \in \widetilde{F^0}(\overline{\cal U})$ bounded from above and satisfying
$$
\sup_{\partial \cal U} u^+ < \sup_{\overline{\cal U}}u. 
$$
In view of Lemma \ref{lem_maxprinc}, the supremum $\sup_{\overline{\cal U}} u$ is not attained on any compact subset. Hence, choosing a compact set $\cal C$ with $\cal K \subset \inte \cal C$ and $\cal C\cap \cal U \neq \emptyset$, $\max_{\cal C} u^+  < \sup_{\overline{\cal U}} u$, and therefore 
$$
u \in \widetilde{F^0}(\overline{\cal U \backslash \cal C}) \quad \text{and} \quad \sup_{\partial (\cal U \backslash \cal C)} u^+ < \sup_{\overline{\cal U \backslash \cal C}} u. 
$$
As $\cal U \backslash \cal C \subset \cal X \backslash \cal K$, $\cal U \backslash \cal C$ is isometric to an open set of $X \backslash K$, hence transplanting $u$ and $\cal U \backslash \cal C$ on $X$ we would contradict the Ahlfors property.

To build the Khas'minskii potential for the pair $(K,h)$, fix $\cal V$ such that property $(\HH 2)$ holds, and extend $h$ continuously on $\cal X\backslash \cal V$ in such a way that the extension, still called $h$, is negative everywhere. We proceed inductively by constructing a decreasing sequence of USC functions $\{w_i\}$, with $w_0=0$ and satisfying the following properties for $i \ge 1$:
\begin{equation}\label{cond_induction_w}
\begin{array}{ll}
(a) &  w_i \in F(\cal X\backslash \cal V), \quad w_{i} = (w_i)_* = 0 \quad \text{ on } \, \partial \cal V; \\[0.2cm]
(b) &  w_i \geq -i \quad \text{ on } \, \cal X\backslash \cal V, \quad w_i = -i \quad \text{outside a compact set $\cal C_{i}$ containing $\cal D_{i}$;} \\[0.2cm]
(c) &  \left( 1- 2^{-i-2}\right)h < w_{i+1} \le w_i \le 0 \quad \text{on } \ \cal X \backslash \cal K, \quad \|w_{i+1}-w_i\|_{L^{\infty}(\cal D_{i}\backslash \cal K)} \leq \frac{\varepsilon}{2^{i}}.
\end{array}
\end{equation}
Since $X \backslash K$ is isometric to $\cal X \backslash \cal K$, $(c)$ implies that the sequence $\{w_i\}$ is locally uniformly convergent on $X\backslash K$ to some function $w$ with $h \le w \le 0$ on $X \backslash K$ and satisfying, for fixed $i$, $w \le -i$ outside $\cal C_{i}$. Therefore, $w(x) \ra -\infty$ as $x$ diverges. Property $w \in F(X\backslash K)$ follows by $(2)$ of Proposition \ref{prop_basicheF}, and hence $w$ is the desired Khas'minskii potential.

In order to start the inductive process, we take a sequence $\{\lambda_j\} \subset C(\cal X\backslash \cal K)$, $j \ge 2$, such that
\begin{equation}\label{def_hj}
\begin{array}{l}
\disp 0 \ge \lambda_j \ge -1, \quad \lambda_j = 0 \ \text{ on } \, \partial \cal K, \quad \lambda_j =-1 \ \text{ on } \, \cal X \backslash \cal D_{j-1}, \\[0.1cm]
\disp \text{$\{\lambda_j\}$ is an increasing sequence, and $\lambda_j \uparrow 0$ locally uniformly.}
\end{array}
\end{equation}
We start with $w_0=0$ on $\cal X$, and set $\cal C_{0} = \cal D_1$. Suppose that we have built $w_i$ on $\cal X \backslash \cal V$. 
Let $j_0$ be large enough to guarantee that $\cal C_{i} \Subset \cal D_{j_0-2}$, $\cal C_{i}$ being the set in property $(b)$.\\[0.2cm]
\noindent \textbf{Step 3: the obstacle problem.}\\[0.2cm]
Write $w=w_i$ (hereafter, for convenience we suppress the subscript $i$). To construct $w_{i+1}$, the idea is to solve obstacle problems with obstacle $w+ \lambda_j$, and to show that the solutions are close to $w$ if $j$ is large enough. However, since $w$ is generally not continuous, we need to approximate it with a family of continuous functions $\{\psi_{k}\}$. For $i=0$, define $ \psi_{k} = w_0 = 0 $ for each $k$, and for $i \ge 1$, being $w \in \USC(\cal X \backslash \cal V)$, we can choose a sequence
$$
\{\psi_{k}\} \subset C(\cal X \backslash \cal V), \qquad 0 \ge \psi_{k} \downarrow w \ \ \text{ pointwise on } \, \cal X \backslash \cal V. 
$$
Note that, by $(a)$, $\psi_k = 0$ on $\partial \cal V$, for each $k$. From $w \equiv -i$ on $\partial \cal C_{i}$, Dini's theorem guarantees that $\psi_{k} \downarrow w$ uniformly on $\partial \cal C_{i}$, thus up to a subsequence $ \psi_k < -i+k^{-1}$ on $\partial \cal C_{i}$. Extending appropriately $\psi_{k}$ outside of $\cal C_{i}$, we can suppose 
\begin{equation}\label{assu_psik}
0 \le \psi_{k} \downarrow w \ \ \text{ pointwise on } \, \cal X \backslash \cal V, \quad \psi_{k} = -i \ \ \text{ on } \, \cal X \backslash \cal D_{j_0} \ \ \text{ for each } \, k.
\end{equation}
For $j \ge j_0+1$, consider the sequence of obstacles $g_{j,k} = \psi_{k} + \lambda_j$, and note that, by construction, 
\begin{equation}\label{limitobstacle}
g_{j,k} \equiv -i-1 \quad \text{on } \ \cal X\backslash \cal D_{j-1}. 
\end{equation}
We claim that there exists an almost solution of the obstacle problem, that is, a function $u_{j,k}\in\USC(\overline{\cal D_{j} \backslash \cal V})$ such that  
\begin{equation}\label{dirichlet_step1}
\begin{cases}
	u_{j,k} \in F^{g_{j,k}}(\overline{\cal D_{j}\backslash \cal V}), & (-u_{j,k})^* \in \widetilde{F^{g_{j,k}}}(\overline{\cal D_{j}\backslash \cal V}),\\
	u_{j,k} = (u_{j,k})_* = 0 & \mbox{on} \ \ \partial \cal V ,\\
	u_{j,k} = (u_{j,k})_* = -i-1 & \mbox{on} \ \ \partial\cal D_{j}.
\end{cases}
\end{equation}
Indeed, let $u_{j,k}$ be the following Perron's function for problem \eqref{dirichlet_step1}:
\begin{equation}\label{perronclass_obstacle} 
\begin{array}{l}
\disp u_{j,k}(x) = \sup\big\{u(x): u\in \mathcal{F}_{j,k}\big\}, \\[0.2cm]
\disp \mathcal{F}_{j,k} = \big\{u\in F^{g_{j,k}}(\overline{\cal D_{j}\backslash \cal V}) \  : \  \ u_{|_{\partial\cal D_{j}}}\leq -i-1, \ \ u_{|_{\partial \cal V}}\leq 0\big\}. 
\end{array}
\end{equation}
Note that $\mathcal{F}_{j,k} \neq \emptyset$ since it contains $\Psi \doteq -i-1$. Hence, to apply the conclusions of Theorem \ref{obstaclethm}, we just need to check the continuity of $u_{j,k}$ on the boundary (property $(1)$ therein: note that $(2)$ follows from $(1)$, and $(3)$ from $(1)+(2)$). Regarding $\partial \cal D_{j}$, by construction and \eqref{limitobstacle}, $g_{j,k} \ge \Psi$ and $g_{j,k}\equiv\Psi$ outside of $\cal D_{j-1}$, hence necessarily $u_{j,k} \equiv -i-1$ in a neighborhood of $\partial \cal D_{j}$. On the other hand, $\partial \cal V$ is $F$-convex by Step 1, and thus by Proposition \ref{prop_F_convex_global} there exists a global $F$-barrier $\{\beta_c\}$ at height $0$, defined on a neighborhood $\cal U$ of $\partial \cal V$ by $\beta_c = c\rho$, where $c>0$ is large enough and $\rho$ is a global defining function for $\partial \cal V$ which is $\overrightarrow{F}$-subharmonic on $\cal U$. Now, by $(a)$ in \eqref{cond_induction_w} we may fix $c$ sufficiently large that 
\begin{equation*}
\beta_c \leq w \ \ \text{ on } \ \ \overline{\cal U} \ \ \text{ and } \ \ \beta_c \leq -i-1 \ \ \text{ on } \ \ \partial \cal U.
\end{equation*}
Then, the function
$$
\beta \doteq \begin{cases}
\max\{ \Psi, \beta_c \big\} &  \text{on } \, \cal U, \\
\Psi &  \text{on } \, \cal D_{j}\backslash \big(\cal V \cup \cal U\big)
\end{cases}
$$
is well defined and satisfies $\beta \in \mathcal{F}_{j,k}$ for each $k$. By definition of the Perron's function and since $g_{j,k} \le 0$, 
\begin{equation}\label{bonitinha}
\beta \le u_{j,k} \le 0 \qquad \text{on } \, \cal D_{j}\backslash \cal V, \ \ \text{ for each } \, k, 
\end{equation}
showing property $(1)$ in Theorem \ref{obstaclethm}. 

Next, we extend $u_{j,k}$ by setting $u_{j,k} \doteq -i-1$ on $\cal X \backslash \cal D_{j}$. As $g_{j,k} \equiv -i-1$ outside of $\cal D_{j-1}$, the extension is smooth across $\partial \cal D_{j}$ and $(1)$ of Proposition~\ref{prop_basicheF} gives 
\begin{equation}\label{unifbounds}
\begin{array}{rl}
u_{j,k} \in F(\cal X \backslash \cal V) \quad \text{and} \quad \begin{cases}
-i-1 \le u_{j,k} \le \psi_k + \lambda_j & \text{on } \, \cal X\backslash \cal V, \\
u_{j,k} \le -i &  \text{on } \, \cal X \backslash \cal D_{j_0}, \\
u_{j,k} = -i-1 & \text{on } \, \cal X \backslash \cal D_{j-1}.
\end{cases}
\end{array} 
\end{equation}
\noindent \textbf{Step 4: the limit in $k$.}\\[0.2cm]
By definition of Perron's solution, the sequence $\{u_{j,k}\}$ is monotonically decreasing in $k$, whence by $(2)$ of Proposition \ref{prop_basicheF} we get
\begin{equation}\label{conveink}
u_{j,k} \downarrow u_j \in F(\cal X \backslash \cal V) \quad \text{with} \quad  \begin{cases} 
-i-1 \le u_{j} \le w + \lambda_j \le w &  \text{on } \, \cal X\backslash \cal V, \\
u_{j} \le -i &   \text{on } \, \cal X \backslash \cal D_{j_0}, \\
u_{j} \equiv -i-1 &  \text{on } \, \cal X \backslash \cal D_{j}.
\end{cases} 
\end{equation}
Moreover, taking limits in \eqref{bonitinha} we deduce
\begin{equation}\label{bonitinha_2}
(u_j)_* = u_j = 0 \qquad \text{on } \, \partial \cal V. 
\end{equation}
Next, we set $v_{j,k} \doteq (-u_{j,k})^* -i$ and note that, by \eqref{assu_psik}, \eqref{dirichlet_step1} and \eqref{unifbounds},
\begin{equation}\label{prop_vjk}
\begin{array}{l}
-i \le v_{j,k} \le 1, \qquad v_{j,k} \ge 0 \ \  \text{ on } \, \cal X \backslash \cal D_{j_0}, \\[0.1cm]
v_{j,k} \in \ \widetilde{F^{g_{j,k}+i}}(\cal D_{j}\backslash \cal V) \ \subset \ \widetilde{F^{\lambda_j}}(\cal D_{j}\backslash \cal V)
\end{array}
\end{equation}
for each $k$.
The monotonicity of the sequence $\{u_{j,k}\}$ implies the following inequalities
$$
-u_{j,k} -i \le v_{j,k} \le (-u_j)^* -i.
$$
Taking limits in $k$, 
$$
v_{j,k} \uparrow v_j \ \ \text{ as } \, k \ra +\infty, \qquad -u_j -i \le v_j \le (-u_j)^*-i,
$$
and using $(4)$ of Proposition \ref{prop_basicheF} together with \eqref{conveink}, \eqref{bonitinha_2}, \eqref{prop_vjk} we deduce 
\begin{equation*}\label{prop_vj}
(-u_j)^*-i \equiv v_j^* \in \widetilde{F^{\lambda_{j}}}(\cal X \backslash \cal V), \quad \text{with} \quad \begin{cases}
-i \le v_j \le 1 &  \text{on } \, \cal X \backslash \cal V, \\
v_j^* = -i < 0 &  \text{on } \, \partial \cal V, \\
v_j^* \ge v_j \ge 0 &  \text{on } \, \cal X \backslash \cal D_{j_0}.
\end{cases}
\end{equation*}
%
\noindent \textbf{Step 5: the limit in $j$.}\\[0.2cm]
Again by the definition of Perron's solution, $u_{j+1,k} \ge u_{j,k}$ for each fixed $k$, and taking limits in $k$ we infer that the sequence $\{v_j\}$, hence $\{v_j^*\}$ is decreasing on $\cal X \backslash \cal V$. Therefore, 
$$
v_j^* \downarrow v \in \widetilde{F^{0}}(\cal X\backslash \cal V), \quad \text{with} \quad \begin{cases}
-i \le v \le 1 &  \text{on } \, \cal X \backslash \cal V, \\
v = -i < 0 &  \text{on } \, \partial \cal V, \\
v \ge 0 &  \text{on } \, \cal X \backslash \cal D_{j_0}.
\end{cases}
$$
Indeed, $v \in \widetilde{F^0}(\cal X \backslash \cal V)$ follows from the following elementary argument: by \eqref{def_hj} the sequence $\{\widetilde{F^{\lambda_j}}\}$ is a nested, decreasing family of closed sets converging to $\widetilde{F^0}$, hence 
$$
v^*_s \in \widetilde{F^{\lambda_j}}(\cal X\backslash \cal V), \quad \forall \, s \ge j.
$$
Taking limits in $s$ and using $(3)$ of Proposition \ref{prop_basicheF}, $v \in \widetilde{F^{\lambda_j}}(\cal X \backslash \cal V)$. The thesis follows letting $j \ra +\infty$.

We are now in the position to use the Ahlfors property to $v$, and infer that $v \le 0$ on $\cal X \backslash \cal V$. In particular, $v \equiv 0$ outside of $\cal D_{j_0}$, and by the USC-version of Dini's theorem, 
$$
v^*_j \downarrow 0 \qquad \text{locally uniformly on } \, \cal X \backslash \cal D_{j_0}.
$$
Then the definition of $v_j$ and the bound $u_j \ge -i$ outside of $\cal D_{j_0}$ yield 
\begin{equation}\label{conve_uj}
(u_{j})_* \uparrow -i \quad \text{and} \quad u_j \uparrow -i \quad \text{locally uniformly on } \, \cal X \backslash \cal D_{j_0}.
\end{equation}
%
\noindent \textbf{Step 6: the convergence of $u_j$ on $\cal D_{j_0}$.}\\[0.2cm]
For fixed $\delta>0$, in view of \eqref{def_hj} and \eqref{conve_uj} take $j_\delta > j_0 +1$ large enough that, for each $j \ge j_\delta$,
$$
- \frac{\delta}{2} \le \lambda_j \le 0 \quad \text{on } \, \cal D_{j_0}\backslash \cal V, \qquad u_j \ge -i - \delta \quad \text{on } \, \cal D_{j_0 +1}\backslash \cal D_{j_0}.
$$
Consequently, the monotonicity of $\{u_{j,k}\}$ in $k$ implies, for each $k$, the inequalities $u_{j,k} \ge -i-\delta = w - \delta$ on $\cal D_{j_0 +1}\backslash \cal D_{j_0}$. Consider the function 
$$
\bar{u}_{j,k} =  
\begin{cases}
\max\{w-\delta, u_{j,k}\} & \text{ on } \ \cal D_{j_0 +1}\backslash \cal V,\\
u_{j,k} & \text{ on } \ \overline{\cal D_{j} \backslash \cal D_{j_0 +1}} ,
\end{cases} 
$$
and note that 
$$
\begin{array}{ll}
\bar{u}_{j,k} \in F(\overline{\cal D_{j}\backslash \cal V}), & \qquad \bar{u}_{j,k} \le g_{j,k} \ \text{ on } \, \overline{\cal D}_{j}\backslash \cal V, \\[0.2cm]
\bar{u}_{j,k} = 0 \ \text{ on } \, \partial \cal V, & \qquad \bar{u}_{j,k} = u_{j,k} = -i-1  \ \text{ on } \, \partial \cal D_{j}.
\end{array}
$$ 
Thus, $\bar{u}_{j,k}$ belongs to the Perron's class $\mathcal{F}_{j,k}$ described in \eqref{perronclass_obstacle}, and consequently
$$ \bar{u}_{j,k} \le u_{j,k} \le \psi_k + \lambda_j \ \text{ on } \, \cal D_{j} \backslash \cal V .$$
By the definition of $\bar{u}_{j,k}$, and taking limits in $k$ we obtain 
$$ w - \delta \le u_j \leq w + \lambda_j \le w  \ \text{ on } \, \cal D_{j_0} \backslash \cal V .$$
Letting first $j \ra +\infty$ and then $\delta \ra 0^+$, we deduce $u_j \uparrow w$ uniformly on $\overline{\cal D}_{j_0}\backslash \cal V$ and thus, by \eqref{conve_uj}, locally uniformly on all of $\cal X \backslash \cal V $.\\[0.2cm]
\noindent \textbf{Step 7: conclusion.}\\[0.2cm]
To produce $w_{i+1}$, it is enough to select $j$ large enough to satisfy
$$
\left(1-2^{-i-2}\right)h < u_j \quad \text{and} \quad  u_{j} \ge w - \frac{\varepsilon}{2^{i+1}} \quad \text{on } \, \cal D_{i+1}\backslash \cal K.
$$ 
Note that the first follows from $(c)$ in \eqref{cond_induction_w} and the definition of $h$.
By \eqref{conveink}, defining $w_{i+1} \doteq u_j$ and $\cal C_{i+1} \doteq \cal D_{j}$, the potential $w_{i+1}$ meets all the requirements in $(a),(b),(c)$, as desired.
\end{proof}

We conclude this section by commenting on a variant of Theorem \ref{teo_main}. Precisely, we investigate the case when $(\HH 1)$ does not hold, typically, for examples $(\EE 2), \ldots, (\EE 8)$ with $f \equiv 0$. Assumption $(\HH 1)$ is just used to ensure Lemma \ref{lem_maxprinc}, that is, the strong maximum principle for functions $u \in \widetilde{F^0}(Y)$ on any manifold $Y$. Lemma \ref{lem_maxprinc} is then essential to prove Claim 1 in implication $(A) \Rightarrow (K)$. Summarizing, following the proof above we can state the following alternative version of our main theorem.

\begin{theorem}\label{teo_main_semH1}
Let $\mathbb{F} \subset J^{2}(X)$ be a universal subequation satisfying $(\HH 2)$, $(\HH 3)$ and
\begin{itemize}
\item[$(\HH 1')$] $\widetilde{\mathbb{F}}$ has the strong maximum principle on each manifold $Y$ where it is defined: non-constant, $\widetilde{\mathbb{F}^0}$-subharmonic functions on $Y$ are constant if they attain a local maximum. 
\end{itemize}
Then, AK-duality holds for $F$. 
\end{theorem}

Similarly, the result can be stated for $F$ locally jet-equivalent to $\mathbb{F}$, provided that the strong maximum principle in $(\HH 1')$ holds for each manifold $Y$ and each such $\widetilde{F} \subset J^2(Y)$. The literature on the strong maximum principle is vast, and here we limit ourselves to refer the interested reader to \cite{HL_SMP} (for universal subequations of Hessian type) and \cite{bardidalio}, and the references therein, and to \cite{PuRS, pucciserrin} for weak solutions of quasilinear operators.

\subsection{Coupling with the Eikonal}

The purpose of this subsection is to improve Theorem \ref{teo_main}, especially in the particular case of the examples in $(\EE 2),\ldots, (\EE 8)$, by producing Khas'minskii potentials which are also $E_\xi$-subharmonic, where as usual 
$$
E_\xi = \overline{\big\{ |p| < \xi(r)\big\}}, \quad \text{and $\xi$ satisfies $(\xi 1)$ in \eqref{def_f1xi1}.}
$$
The main difficulty here is that, even when $F$ satisfies all of the assumptions in Theorem \ref{teo_main} and when $F \cap E_\xi$ is a subequation, $F \cap E_\xi$ certainly does not satisfy condition $(\HH 2)$. Indeed, the asymptotic interior of $E_\xi$ is empty for each fixed $r$, which prevent any subset to be $(F \cap E_\xi)$-convex at any height.

In order to overcome this problem, we fix a non-negative function $\eta \in C(X)$, and define the ``relaxed" eikonal subequation
\begin{equation}\label{def_exieta}
E_{\xi}^\eta = \overline{\Big\{ (x,r,p,A) \in J^2(X) \ : \ |p| < \xi(r) + \eta(x)\Big\}}.
\end{equation}
Our first step is to ensure that the set $F \cap E_{\xi}^\eta$ is a subequation for each $\eta$ and each $\xi$ satisfying $(\xi 1)$ (or even $(\xi 0)$ in \eqref{def_xi0}). Clearly, $F \cap E_{\xi}^\eta$ is a closed subset with $(P),(N)$, but the topological condition is not automatic, even for subequations satisfying all the assumptions in Theorem \ref{teo_main}. As a simple example, the set
$$
F \doteq \big\{\tr(A) \ge r\big\} \cap \big\{ |p| \ge 1\big\}
$$
is a universal Riemannian subequation satisfying $(\HH 1),(\HH 2),(\HH 3)$, but $F \cap E_{\xi}^\eta$ is not a subequation (fibers over points $x$ with $\eta(x) =0$ and $\xi =1$ have empty interior). However, for Examples $(\EE 2),\ldots, (\EE 8)$ it is not difficult to check the following lemma.

\begin{lemma}\label{lem_FcapE_subequation}
Let $F$ be locally jet-equivalent to one of the examples in $(\EE 2), \ldots, (\EE 8)$. Then, $F \cap E_{\xi}^\eta$ is a subequation for each $0 \le \eta \in C(X)$ and each $\xi$ satisfying either $(\xi 1)$ or $(\xi 0)$.
\end{lemma}

\begin{remark}
\emph{In general, $F \cap E_{\xi}^\eta$ might not be locally jet-equivalent to a universal subset with model $\mathbf{F}\cap \mathbf{E}_{\xi}$\footnote{For instance, if $\eta>0$ on $X$ and $\xi \equiv 1$, the subequation $\widetilde E \cap E^\eta$ is not locally affine jet-equivalent to the subset obtained by the model $\widetilde{ \mathbf{E}} \cap \mathbf{E}$ (the first has non-empty interior, whereas the second not).}. However, it is so when the model $\mathbf{F}$ is independent of the gradient, and in this case $F \cap E_\xi^\eta$ is a subequation. Observe that, even though $\mathbf{F}$ is gradient independent, $F$ might depend on the gradient, as the bundle map $L$ in the jet-equivalence does. This is the case, for instance, of the linear subequations described in Example \ref{ex_linear}.
}
\end{remark}

Next, we need to check when $F \cap E_{\xi}^\eta$ satisfy the weak comparison and the strict approximation properties. The proof of the next two propositions are a minor modification of, respectively, \cite[Thm. 10.1]{HL_dir} and Theorem \ref{thm_comparison_examples} here, and are left to the reader.
\begin{proposition}\label{lem_weakcomp_congradient}
Let $F \subset J^2(X)$ be locally affine jet-equivalent to a universal subequation with model $\mathbf{F}$, where the sections $g,L$ are continuous and $h$ is locally Lipschitz. 
Then $F \cap E^\eta_\xi$ satisfies the weak comparison for each $0 \le \eta \in C(X)$ and each $\xi$ satisfying either $(\xi 1)$ or $(\xi 0)$.
\end{proposition}

\begin{proposition}\label{thm_comparison_examples_withgradient}
Let $F$ be locally jet-equivalent to a universal subequation with model $\mathbf{F} = \big\{\mathscr{F}(p,A) \ge f(r)\big\}$ via continuous bundle maps. Suppose that $\mathscr{F}$ is uniformly continuous (as in Definition \ref{def_unifcontinuous}), and that $f$ is strictly increasing. For $\xi$ satisfying $(\xi 1)$ or $(\xi 0)$, define $E_\xi^\eta$ as in \eqref{def_exieta} and suppose that $\xi$ is strictly decreasing. Then, $F\cap E_\xi^\eta$ has the bounded, strict approximation property.
\end{proposition}

\begin{remark}
\emph{It is worth to observe that, differently from \cite[Thm. 10.1]{HL_dir}, Proposition \ref{lem_weakcomp_congradient} does not require $g,L$ to be locally Lipschitz: this is due to the control on the gradient granted by $E_\xi^\eta$. Observe also that $F \cap E_{\xi}^\eta$ is not required to be a subequation\footnote{Indeed, the proof of Proposition \ref{lem_weakcomp_congradient} uses the classical theorem on sums (\cite{CIL}, Thm. 3.2), which is stated in Appendix 1 of \cite{HL_dir} for subequations although it just need the sets to be closed.}.
}
\end{remark}

We are ready to state our result for subequations coupled with $E_{\xi}^{\eta}$.

\begin{theorem}\label{teo_main_withgradient}
Let $F \subset J^{2}(X)$ be locally jet-equivalent to a universal subequation $\mathbb{F}$  via locally Lipschitz bundle maps.
Assume $(\HH 1), (\HH 2)$, and suppose that $F \cap E_{\xi}^\eta$ is a subequation satisfying
\begin{itemize}
\item[$(\HH 3')$] $F\cap E_{\xi}$ has the bounded, strict approximation on $X$,
\end{itemize}
for some $\xi$ enjoying $(\xi 1)$ or $(\xi 0)$. Then, AK-duality holds for $F\cap E_\xi :$
\begin{equation*}
\begin{array}{c}
\widetilde{F} \cup \widetilde{E_{\xi}} \ \mbox{ satisfies $(A)$} \\
(\mbox{Ahlfors prop.}) \end{array}
 \Longleftrightarrow  
 \begin{array}{c}
F\cap E_{\xi} \ \mbox{ satisfies $(K)$} \\
(\mbox{Khas'minskii prop.})
\end{array}
 \Longleftrightarrow  
 \begin{array}{c}
F \cap E_{\xi} \ \mbox{ satisfies $(K_\weak)$} \\
(\mbox{weak Khas'minskii prop.})
 \end{array}
\end{equation*}
Each of them implies that 
\begin{equation}\label{LforwidetildeF}
\widetilde{F} \cup \widetilde{E_{\xi}} \quad \text{satisfies} \quad (L) \quad (\text{Liouville prop.}).
\end{equation}
If, moreover, $(\HH 4)$ holds, then each of the above properties is equivalent to~\eqref{LforwidetildeF}.
\end{theorem}
\begin{proof}
Define $H \doteq F \cap E^\eta_\xi$. Implications $(K_\weak) \Rightarrow (A) \Rightarrow (L)$ and $(L) \Rightarrow (A)$, under assumption $(\HH 4)$, follow step by step the corresponding proofs in Theorem \ref{teo_main}. In particular, we observe that $(\HH 1)$ also implies the finite maximum principle for $\widetilde{H^0}$. To apply the comparison argument in $(K_\weak) \Rightarrow (A)$, note that Proposition \ref{lem_weakcomp_congradient} ensures the weak comparison for $H$. Then, because of $(\HH 3')$ and Lemma \ref{wcobstacle}, $H^0$ enjoys both the weak comparison and the strict approximation, hence the full comparison.\\
On the other hand, $(A) \Rightarrow (K)$ needs a further argument. As before, we fix $(K,h)$ and a smooth exhaustion $\{D_j\}$ of $X$ with $ K \subset D_1$, and we extend the manifold $X$ to $\cal X$ and the subequation $F$ on $\cal X$ in such way that $F$ is still locally jet-equivalent to $\mathbb{F}$. For $0 \le \eta \in C_{c}(\cal K)$ we define $ H(\cal X) \doteq (F\cap E_{\xi}^{\eta})(\cal X)$. The finite maximum principle holds for $\widetilde{H^0}(\cal X)$ since negative constants are strictly $H$-subharmonic. Therefore, $\widetilde{H}$ satisfies the Ahlfors property $(A)$ on~$\cal X$.

In order to construct the Khas'minskii potential for $(K,h)$, the only problem is the lack of barriers on $\partial \cal V$ for the obstacle problems producing $\{w_i\}$. To circumvent this fact, we need to enlarge, at each step $i$, the subequation $H$. In other words, having fixed $\cal V$ such that $(\HH 2)$ holds, and extending $h$ to a negative, continuous function on $\cal X \backslash \cal V$, we inductively build a sequence $\{w_i\}$ with $w_0 = 0$, and a sequence $\eta_i$ with $\eta_0 \equiv 0$, such that setting
$$
H_i = F \cap E_\xi^{\eta_i},
$$
it holds
\begin{equation}\label{conditions_w_i_coupledgrad}
\begin{array}{ll}
(o) &  \eta_i \in C_c(\cal K), \quad 0 \le \eta_i \le \eta_{i+1}; \\
(a) &  w_i \in H_{i}(\cal X\backslash \cal V), \quad w_{i} = (w_i)_* = 0 \quad \text{ on } \, \partial \cal V; \\
(b) &  w_i \geq -i \quad \text{ on } \, \cal X\backslash \cal V, \quad w_i = -i \quad \text{outside a compact set $\cal C_{i}$ containing $\cal D_{i}$;} \\
(c) &  \left( 1- 2^{-i-1}\right)h < w_{i+1} \le w_i \le 0 \quad \text{on } \, \cal X \backslash \cal K \ \ \text{and} \ \ \|w_{i+1}-w_i\|_{L^{\infty}(\cal D_{i}\backslash \cal K)} \leq \frac{\varepsilon}{2^{i}}.
\end{array}
\end{equation}
Since $\eta_i \equiv 0$ out of $\cal K$, $\cal X \backslash \cal K$ is isometric to $X \backslash K$ and $H_i = F \cap E_\xi$ on $\cal X \backslash \cal K$, we deduce $w_i \in (F\cap E_{\xi})(X\backslash K)$ for each $i$, and $\{w_i\}$ converges locally uniformly on $X\backslash K$ to a function $w$ which is the desired Khas'minskii potential.

Suppose that we have constructed $w_i, \eta_i$ as above. The induction closely follows in steps 3 to 7 in the proof of Theorem \ref{teo_main}, the only difference being that to guarantee the solvability of the obstacle problem 
\begin{equation}\label{dirichlet_teo_main_gradient}
\begin{cases}
u_{j,k} \in H_{i+1}^{g_{j,k}}(\overline{\cal D_{j}\backslash \cal V}), \quad (-u_{j,k})^* \in \widetilde{H_{i+1}^{g_{j,k}}}(\overline{\cal D_{j}\backslash \cal V}),\\
u_{j,k} = (u_{j,k})_* = 0 \qquad \mbox{on} \ \ \partial \cal V ,\\
u_{j,k} = (u_{j,k})_* = -i-1 \qquad \mbox{on} \ \ \partial\cal D_{j},
\end{cases}
\end{equation}
we need to find $\eta_{i+1} \ge \eta_i$ large enough that $\partial \cal V$ has good $H_{i+1}$-barriers. Observe that the boundary $\partial \cal D_{j}$ can be treated verbatim as in Theorem~\ref{teo_main}. The assumption $(\HH 2)$ together with Proposition \ref{prop_F_convex_global} give us, in a neighborhood $\cal U$ of $\partial \cal V$, an $F$-barrier $\beta_c$ at height $0$, whenever $c>0$ is large, satisfying
$$ \beta_c \in \overrightarrow{F}(\cal U), \ \ \text{ and } \ \ \beta_c = c\rho \ \ \text { on } \ \ \cal U. $$
Reasoning as in Theorem \ref{teo_main}, we fix $c$ large enough such that 
\begin{align*}
&\beta_c \le w \ \ \text{ on } \ \ \overline{\cal U}, \ \ \beta_c < -i-1 \ \ \text{ on } \ \ \partial \cal U \cap \cal D_{j},  \\
&\text{and } \ \ C \doteq \sup_{\{\beta_c \geq -i-2\}}\vert \di \beta_c \vert < + \infty .
\end{align*}
Once this is done, we can define $\eta_{i+1}$ to be $\eta_{i+1} = \max\{2\|\eta_i\|_{L^\infty},C\} $ on $ \overline{\cal U}$ and extend it smoothly on $\cal X$ in such a way that $\eta_{i+1} \geq \eta_i$ and $\eta_{i+1}$ vanishes outside $\cal K$. Then, setting 
\begin{equation*}
\beta \doteq
\begin{cases}
\max\{-i-1,\beta_c\} & \text{ on} \ \ \cal U , \\
-i-1 & \text{ on} \ \ \cal D_{j}\backslash (\cal V\cup \cal U),
\end{cases}
\end{equation*}
$\beta$ is $H_{i+1}$-subharmonic and belongs to Perron's class 
$$
\big\{u\in H_{i+1}^{g_{j,k}}(\overline{\cal D_{j}\backslash \cal V}) \  :   \ u_{|_{\partial\cal D_{j}}}\leq -i-1, \ \ u_{|_{\partial \cal V}}\leq 0\big\}.
$$
The existence of $u_{j,k}$ solving \eqref{dirichlet_teo_main_gradient} then follows as in Theorem \ref{teo_main}. The remaining steps to prove $(A) \Rightarrow (K)$ follow verbatim, once we observe that $H_i \subset H_{i+1}$.
\end{proof}

\begin{remark}
\emph{In a similar way, Theorem \ref{teo_main_semH1} can be stated for subequations coupled with the eikonal by assuming that the strong maximum principle holds for $\widetilde{F^0}\cup \widetilde{E_\xi}$.
}
\end{remark}

\section{How the Ahlfors property depends on \texorpdfstring{$f$ and $\xi$}{}}\label{sec_dependenceonf}
Given a subequation $F_f$ locally jet-equivalent to an example $\mathbb{F}_f$ in $(\EE 2), \ldots, (\EE 8)$ and an eikonal $E_\xi$, in this section we investigate the dependence of the Ahlfors property on the pair $(f,\xi)$. First, observe that $\mathbb{F}_f$ in $(\EE 2),\ldots,(\EE 6)$ can be written in the following general form:
\begin{equation}\label{classicoesempio}
\mathbb{F}_f \doteq \big\{\mathscr{F}(p,A) \ge f(r) \big\}, \qquad \widetilde{\mathbb F_f} \doteq \big\{ \widetilde{\mathscr{F}}(p,A) \ge \widetilde f(r) \big\}, 
\end{equation}
for suitable $\mathscr{F} : \mathbf{J^2} \ra \R$ and $f \in C(\R)$ non-decreasing, where 
$$
\widetilde{\mathscr{F}}(p,A) \doteq -\mathscr{F}(-p,-A),\qquad \widetilde f(r) \doteq -f(-r). 
$$
Examples $(\EE 7)$ and $(\EE 8)$ differ from the prototype \eqref{classicoesempio} by small topological adjustments.

Hereafter, we consider the following pair of mutually exclusive conditions:
\begin{equation*}\label{ipo_f_F1eF1p}
\begin{array}{ll}
(f1) &  f(0)=0, \quad f(r)< 0 \ \text{ for } \, r<0, \quad \text{or}\\[0.2cm]
(f1') & f(r)\equiv 0 \ \text{ on some interval } \, (-\mu,0), \ \mu>0.
\end{array}
\end{equation*}
\begin{proposition}\label{prop_equivalenceahlfors}
Let $F_f$ be locally jet-equivalent to a subequation described by \eqref{classicoesempio}, for some $\mathscr{F} : \mathbf{J^2} \ra \R$ and $f \in C(\R)$ non-decreasing. Then, the following assertions are equivalent:
\begin{itemize}
\item[$(i)$] $\widetilde{F_f}$ has the Ahlfors property for some $f$ satisfying $(f1)$;
\item[$(ii)$] $\widetilde{F_f}$ has the Ahlfors property for each $f$ satisfying $(f1)$;
\item[$(iii)$] $\widetilde{F_f}$ has the Ahlfors property whenever $f \equiv -\eps$, for each fixed constant $\eps>0$;
\end{itemize}
as well as the following ones:
\begin{itemize}
\item[$(i')$] $\widetilde{F_f}$ has the Ahlfors property for some $f$ satisfying $(f1')$;
\item[$(ii')$] $\widetilde{F_f}$ has the Ahlfors property for each $f$ satisfying $(f1')$.
\end{itemize}
\end{proposition}

\begin{remark}
\emph{The theorem still holds when the universal subequation is $(\EE 7)$ or $(\EE 8)$, with minor changes.
}
\end{remark}

\begin{proof}
We begin supposing that $f$ satisfies $(f1)$. For notational convenience, we set $g(t) \doteq \widetilde{f}(t)$ and write $\widetilde{F}_g$ instead of $\widetilde{F_f}$: 
$$
\widetilde{\mathbb F}_g = \big\{ \widetilde{\mathscr{F}}(p,A) \ge g(r)\big\}, \qquad \text{and} \qquad (g1) \  :  \ g(0)=0, \ g>0 \ \text{ on } \, \R^+.
$$
Observe that $(ii) \Rightarrow (i)$ is obvious.\\[0.1cm]
$(i) \Rightarrow (ii)$. Let $\bar g$ be a function satisfying $(g1)$ for which $\widetilde{F}_{\bar g}$ has the Ahlfors property. Suppose, by contradiction, that there exist an open set $U$ with non-empty boundary and a bounded function $u \in H_g(\overline{U})$, $H_g \doteq \widetilde{F}_g \cup \{r \le 0\}$, which is positive somewhere and satisfies
$$
\sup_{\partial U} u^+ < \sup_{\overline{U}} u.
$$ 
Using $(g1)$ and the continuity of $g,\bar g$, we can fix $c> \sup_{\partial U} u^+$ close enough to $u_\infty \doteq \sup_{\overline{U}} u$ such that
\begin{equation}\label{facil}
\min_{[c,u_\infty]} g \ge \max_{[0,u_\infty-c]} \bar g.
\end{equation}
Let $x$ satisfy $u(x)>c$, and let $\Psi$ be a local jet-equivalence around $x$ satisfying $\Psi(F_g) = \mathbb{F}_g$. Pick a test function $\phi$ for $u-c$ and observe that, since $\Psi$ is a jet-equivalence, if $(r,p,A) \doteq \Psi(J^2_x\phi)$ then $(r+c,p,A) \equiv \Psi\big(J^2_x(\phi+c)\big)$. Using that $\phi+c$ is a test for $u$, from \eqref{facil} we obtain
$$
\widetilde{\mathscr{F}}\big(p, A\big) \ge g\big(r+c\big) \ge \bar g(r),
$$
thus $J^2_x\phi \in \Psi^{-1}(\mathbb{F}_{\bar g}) = F_{\bar g}$. Therefore, $w \doteq \max\{u-c,0\}$ is bounded above, non-constant and satisfies $w \in H_{\bar g}(\overline{U})$, and
$$
0 = \sup_{\partial U} w^+ < \sup_{\overline{U}} w,
$$
contradicting the Ahlfors property for $\widetilde{F}_{\bar g}$.\\[0.1cm]
$(ii) \Rightarrow (iii)$.\\ 
If $(iii)$ is not satisfied for some $U\subset X$, $\eps>0$ and $u \in H_{\eps}(\overline{U})$, it is enough to choose $g \in C(\R)$ of type $(g1)$ such that $g \le \eps$ on $\R$. It is immediate to see that $u \in H_g(\overline{U})$, contradicting $(ii)$.\\[0.1cm]
$(iii) \Rightarrow (ii)$.\\
If, by contradiction, there exist $g \in C(\R)$ of type $(g1)$, $U \subset X$ and $u \in H_g(\overline{U})$ contradicting $(A)$, fix $c \in (\sup_{\partial U} u^+, \sup_U u)$ and let $\eps \in (0, g(c)]$. As above, a direct check shows that $w \doteq \max\{u-c,0\}$ satisfies 
$$
w \in H_{\eps}(\overline{U}), \qquad 0 = \sup_{\partial U} w^+ < \sup_{\overline{U}} w,  
$$
contradicting $(iii)$.\\[0.1cm]
When $f$ is of type $(f1')$, the proof of $(i') \Leftrightarrow (ii')$ needs just minor changes, and is left to the reader.
\end{proof}

In an entirely analogous way, one can prove the following for the subequation $F_f \cap E_\xi$.

\begin{proposition}\label{prop_equivalenceahlfors_withgradient}
Let $F_f$ be locally jet-equivalent to a subequation described by \eqref{classicoesempio}, for some $\mathscr{F} : \mathbf{J^2} \ra \R$ and $f \in C(\R)$ non-decreasing. Consider
$$
E_\xi = \overline{\big\{ |p| < \xi(r)\big\}}, \quad \text{ for $\xi$ satisfying $(\xi 1)$ in \eqref{def_f1xi1}.}
$$
Then, the following assertions are equivalent:
\begin{itemize}
\item[$(i)$] $\widetilde{F_f}\cup \widetilde{E_\xi}$ has the Ahlfors property for some $(f,\xi)$ satisfying $(f1+\xi 1)$;
\item[$(ii)$] $\widetilde{F_f} \cup \widetilde{E_\xi}$ has the Ahlfors property for each $(f,\xi)$ satisfying $(f1+\xi 1)$;
\item[$(iii)$] $\widetilde{F_f}\cup \widetilde{E_\xi}$ has the Ahlfors property for each $f \equiv -\eps$, $\xi \equiv \eps$ and each fixed $\eps>0$.
\end{itemize}
Furthermore, the following conditions are equivalent:
\begin{itemize}
\item[$(i)$] $\widetilde{F_f}\cup \widetilde{E_\xi}$ has the Ahlfors property for some $(f,\xi)$ satisfying $(f1'+\xi 1)$;
\item[$(ii)$] $\widetilde{F_f} \cup \widetilde{E_\xi}$ has the Ahlfors property for each $(f,\xi)$ satisfying $(f1'+\xi 1)$.
\end{itemize}
\end{proposition}

\begin{proof}[Sketch of Proof]
In $(i) \Rightarrow (ii)$, having assumed the Ahlfors property for some $(\bar g,\bar \xi)$ and having supposed its failure for a pair $(g,\xi)$, the only difference is that $c< u_\infty$ is required to satisfy, instead of \eqref{facil}, 
$$
\min_{[c,u_\infty]} g \ge \max_{[0,u_\infty-c]} \bar g, \qquad \min_{[-u_\infty, -c]} \xi \ge \max_{[-u_\infty+c,0]}\bar \xi.
$$
The second condition is granted by $\bar \xi(0)=0$ in $(\xi 1)$. The rest follows verbatim. 
\end{proof}
\section{The main theorem for examples \texorpdfstring{$(\EE 2), \ldots, (\EE 8)$}{}}\label{sec_main_examples}
Here, we combine Theorem \ref{teo_main}, Theorem \ref{teo_main_withgradient} and Propositions \ref{prop_equivalenceahlfors}, \ref{prop_equivalenceahlfors_withgradient} in Section \ref{sec_dependenceonf} in the specific case of Examples $(\EE 2), \ldots, (\EE 8)$. We focus on the case when $(f,\xi)$, in the definition of $(\EE 1), \ldots, (\EE 8)$, satisfies $(f1 + \xi 1)$ in~\eqref{def_f1xi1}. 

For $0 \le \eta \in C(X)$, define $E_\xi^\eta$ as in \eqref{def_exieta}, and let $F_f$ be locally jet-equi\-valent to any of examples $(\EE 2), \ldots, (\EE 6)$ or be the universal Riemannian subequations in $(\EE 7)$ or in $(\EE 8)$. We know from Lemma \ref{lem_FcapE_subequation} that $F_f \cap E_\xi^\eta$ is a subequation. The independence of the Ahlfors property from the specific pair $(f,\xi)$ satisfying $(f1 + \xi 1)$ allows us to ``play" with $f,\xi$ to prove the equivalence stated in Theorem \ref{cor_bonito}: in particular, it is enough to address the comparison properties of $F_f$ and $F_f \cap E_\xi^\eta$ under the assumptions of the next proposition, whose proof is reported in the Appendix \ref{appendix_1}.
\begin{proposition}\label{comparison_examples}
Fix $(f,\xi)$ satisfying $(f1 + \xi 1)$, and assume that
$$
\text{$f$ is strictly increasing on $\R$ $\quad$ and $\quad$ $\xi$ is strictly decreasing on $\R$}. 
$$
\begin{itemize}
\item[$i)$] If $F_f \subset J^2(X)$ is locally jet-equivalent to one of the examples in $(\EE 2), \ldots, (\EE 6)$ via locally Lipschitz bundle maps, then the bounded comparison holds both for $F_f$ and for $F_f \cap E_\xi$. 
\item[$ii)$] If $F_f$ is the universal, quasilinear subequation in $(\EE 7)$ with eigenvalues $\{\lambda_j(t)\}$ in \eqref{ipo_a}, then the bounded comparison holds in the following cases: 
\begin{itemize}
\item[-] for $F_f$, provided that $\lambda_j(t) \in L^\infty(\R_0^+)$ for $j \in \{1,2\}$; 
\item[-] for $F_f \cap E_\xi$, provided that $\lambda_j(t) \in L^\infty_{\mathrm{loc}}(\R_0^+)$ for $j \in \{1,2\}$. 
\end{itemize}
\item[$iii)$] If $F_f$ is the universal subequation in $(\EE 8)$, then the bounded comparison holds both for $F_f$ and for $F_f \cap E_\xi$. 
\end{itemize}
Furthermore, for each of $i), \ldots, iii)$, comparison also holds for the obstacle subequations $F_f^0$ and $F_f^0 \cap E_\xi$.
\end{proposition}

We are ready to prove Theorem \ref{cor_bonito}, which we rewrite below for convenience.

\begin{theorem}\label{teo_main_esempi}
Fix $f$ satisfying $(f1)$ and $\xi$ satisfying $(\xi 1)$. 
\begin{itemize}
\item[$i)$] If $F_f \subset J^2(X)$ is locally jet-equivalent to one of the examples in $(\EE 2), \ldots, (\EE 6)$ via locally Lipschitz bundle maps, then AK-duality holds for $F_f$ and for $F_f \cap E_\xi$. 
\item[$ii)$] If $F_f$ is the universal, quasilinear subequation in $(\EE 7)$, with eigenvalues $\{\lambda_j(t)\}$ in \eqref{ipo_a}, then 
\begin{itemize}
\item[-] AK-duality holds for $F_f$ provided that $\lambda_j(t) \!\in\! L^\infty(\R_0^+)$ for $j \!\in\! \{1,2\}$; 
\item[-] AK-duality holds for $F_f \cap E_\xi$ provided that $\lambda_j(t) \in L^\infty_{\mathrm{loc}}(\R_0^+)$ for $j \in \{1,2\}$. 
\end{itemize}
\item[$iii)$] If $F_f$ is the universal subequation in $(\EE 8)$, then AK-duality holds for $F_f$ and for $F_f \cap E_\xi$. 
\end{itemize}
\end{theorem}
\begin{proof}
We first prove AK-duality for $F_f$. As seen at the beginning of Section~\ref{sec_teoprinci}, since $f$ satisfies $(f1)$ then assumptions $(\HH1), (\HH2)$ and $(\HH4)$ hold when $F_f$ is one of examples $(\EE 2), \ldots, (\EE 8)$. Following the proof of Theorem \ref{teo_main}, we therefore get the implications $(A) \Leftrightarrow (L)$, $(A) \Rightarrow (K) \Rightarrow (K_{w})$. In order to obtain the last implication $(K_w) \Rightarrow (A)$ we need comparison for $F_f^0$. When $f$ is strictly increasing, comparison is given for all examples $(\EE 2), \ldots, (\EE 8)$ by Proposition \ref{comparison_examples}. Therefore, applying Theorem \ref{teo_main} we deduce the desired duality. 
%
Now, for general $f$ satisfying $(f1)$, we choose $\bar{f}$ strictly increasing satisfying 
$$
\bar{f} \le f \ \ \text{ on } \ \ (-\infty,0) \ \ \text{and} \ \ \bar{f}(0)=0.
$$
Clearly, the weak Khas'minskii property for $F_f$ implies that for $F_{\bar{f}}$, which by the first part of the proof ensures the Ahlfors property for $\widetilde{F_{\bar{f}}}$. The Ahlfors property for $\widetilde{F_f}$ now follows by Proposition \ref{prop_equivalenceahlfors}.

Duality for $F_f \cap E_{\xi}^{\eta}$ follows the same path, we just observe that Proposition \ref{comparison_examples} holds when $f$ is strictly increasing and $\xi$ is strictly decreasing. The conclusion for general $(f,\xi)$ follows by using Proposition \ref{prop_equivalenceahlfors_withgradient}.
\end{proof}

\section{Immersions and submersions}\label{sec_immesub}

Let $\sigma : X^m \ra Y^n$ be a smooth map. Given a subequation $F \subset J^2(Y)$, we have an induced subset 
$$
H \doteq \overline{\sigma^*F}, \quad \text{where} \quad (\sigma^*F)_x = \Big\{ J^2_x( \phi \circ \sigma) \  :  \ J^2_{\sigma(x)} \phi \in F\Big\}. 
$$
Clearly, $H$ satisfies\footnote{The topological condition is more delicate, but for the present section we do not need it.} $(P)+(N)$. We investigate the following question: 
\begin{quote}
\emph{if $\widetilde{F}$ satisfies the Ahlfors property on $Y$, is it true that so does $\widetilde{H}$ on $X$?}
\end{quote}
In many instances it happens that $H \equiv J^2(X)$ and the problem is uninteresting. However, for $F$ belonging to some relevant classes the induced subset $H$ is non-trivial, and in fact contained in a subequation in the same class as $F$. This is so if $\sigma$ is an isometric immersion or a Riemannian submersion, and when the subequation is $F_f \cap E_\xi$, with $F_f$ as in examples $(\EE 3), (\EE 5)$ (with the restriction $k \le m$) or their complex analogues in $(\EE 6)$.


The investigation of the question above is considerably simpler if we make use of AK-duality and study when the (weak) Khas'minskii property is preserved for the induced subequation. The sought is then achieved by transplanting the Khas'minskii potentials to submanifolds, respectively to total spaces of Riemannian submersions, following an idea in \cite{bessapiccione, brandaooliveira}. 

\subsection*{Immersions}

Suppose that $\sigma : X^m \ra Y^n$ is an isometric immersion, let $\nabla, \bar \nabla$ denote the Levi-Civita connections of $X$ and $Y$, and let $\II$ be the second fundamental form. Our first result is the following

\begin{theorem}\label{teo_immersions}
Let $\sigma : X^m \ra Y^n$ be a proper isometric immersion with bounded second fundamental form, and let $F_f$ be one of the universal subequations in examples $(\EE 3)$, $(\EE 5)$ (with $k \le m$) or their complex analogues in $(\EE 6)$. Assume that $f$ satisfies $(f1)$, and let $E_\xi$ with $\xi$ satisfying $(\xi 1)$. Then,  
\begin{align*}
&\text{$\widetilde{F_f} \cup \widetilde{E_\xi}$ has the Ahlfors property on $Y$}\\
&\Longrightarrow \text{$\widetilde{F_f} \cup \widetilde{E_\xi}$ has the Ahlfors property on $X$}.
\end{align*}
\end{theorem}

\begin{remark}\label{rem_opportuno}
\emph{For $(\EE 5)$ with $k \le m$ then, in place of $\|\II\|_\infty < +\infty$, it is enough to require that
$$
\sup \Big\{ \big|\tr_{\mathcal{V}}\II(x)\big| \  :  \  x \in X, \ \mathcal{V} \le T_xX \ \text{ $k$-dimensional}\Big\} < +\infty.
$$
In particular, if $k=m$ it is enough that $X$ has bounded mean curvature. 
}
\end{remark}

\begin{proof}[Proof of Theorem \ref{teo_immersions}] By the duality Theorem \ref{cor_bonito}, $F_f \cap E_\xi$ has the Khas'minskii property. Take any Khas'minskii potential $\bar w \!\in\! (F_f \!\cap\! E_\xi)(Y \backslash K)$, $K$ compact. We claim that there exists $g$, independent of $\bar w$ and satisfying the assumptions in $(f1)$, such that 
\begin{equation}\label{pullback_khasm}
w \doteq \bar w \circ \sigma \in (F_g \cap E_\xi)(X \backslash C), \qquad C \doteq \sigma^{-1}(K).
\end{equation}
Since $\sigma$ is proper, $w$ is an exhaustion and $C$ is compact. Because of the arbitrariness of $K$, it is clear that letting $\bar w$ vary among all possible Khas'minskii potentials the family of induced $w$ guarantee the validity of the Khas'minskii property for $F_g \cap E_\xi$ on $X$; AK-duality again, and the independence of $(f,\xi)$ satisfying $(f1+\xi 1)$, ensures the desired Ahlfors property.

To show \eqref{pullback_khasm}, we follow the method in \cite[Thm. 6.6]{HL_restriction}. We consider the real case in $(\EE 3)$, $(\EE 5)$, the complex analogues in $(\EE 6)$ being similar. Let $\phi \in C^2(U)$ be a test for $w$ at some point $q_0 \in U$. We fix the index convention $1 \le i,j \le m$, $m+1 \le \alpha, \beta \le n$ and we choose an open neighborhood $U_0 \subset X$ of $q_0$ which is embedded in $Y$. Let $z=(x,y) \in U_0 \times B_\delta(0)$ be Fermi coordinates around $q_0$, $z(q_0) = (0,0)$, $x = \{x^i\}$, $y = \{y^\alpha\}$, with $\partial_i \doteq \partial/\partial x^i$ tangent and $\partial_\alpha \doteq \partial/\partial y^ \alpha$ normal to $U_0$. By construction, up to restricting to a smaller subset $\Omega = U \times B_\delta(0)$ around $(q_0,0)$, the squared distance function from $U_0$ is given by
$$
|y|^2 = \sum_{\alpha} (y^\alpha)^2,
$$
and, by Gauss lemma,  
\begin{equation}\label{gausslemma}
g_{i\alpha} = 0, \quad |\di y^\alpha| = 1 \qquad \text{for each } \, i,\alpha.
\end{equation}
Up to modifying $\phi$ to fourth order, we can assume that $w -\phi$ has a strict maximum at $q_0$. Set $\bar \phi(x,y) \doteq \phi(x)$ and define $\bar v \doteq \bar w - \bar \phi$ on $\Omega$. For $\eps>0$, let $z_\eps = (x_\eps, y_\eps)$ be a maximum point of 
$$
\bar v_\eps(x,y) \doteq \bar v(x,y) - |y|^2/\eps
$$
on $\overline{\Omega}$. Then, it is straightforward to verify that, as $\eps \ra 0$,
$$
M_\eps \doteq \bar v_\eps(x_\eps, y_\eps) \downarrow 0, \quad (x_\eps,y_\eps) \rightarrow (0, 0), \quad |y_\eps|^2/\eps \ra 0.
$$
Since, by construction, $\bar \phi_\eps \doteq \bar \phi + \frac{1}{\eps}|y|^2+ M_\eps$ is a test function for $\bar w$ at $z_\eps$, we have
\begin{equation}\label{condizzzion}
J^2_{z_\eps} \bar \phi_\eps = \left( \phi(x_\eps) + M_\eps,  (\partial_j \bar \phi) \di x^j + \frac{2}{\eps}y^\alpha \di y^\alpha,  \bar \nabla^2 \bar \phi + \frac{1}{\eps} \bar \nabla^2 |y|^2\right) \in F_{z_\eps}.
\end{equation}
From $\bar w \in E_\xi$ and \eqref{gausslemma}, we deduce 
\begin{equation*}
|\di \bar \phi|^2 + \frac{4|y|^2}{\eps^2} \le \xi^2\big( \phi(x_\eps) + M_\eps\big).
\end{equation*}
We treat $(\EE 5)$, $(\EE 3)$ being analogous. Fix a $k$-dimensional subspace of the span of $\{\partial_j\}$ at $z=0$ which, without loss of generality, we can assume to be $\langle \partial_1,\ldots, \partial_k\rangle(0)$, and consider the subbundle $\mathcal{V} \doteq \langle \partial_1,\ldots, \partial_k\rangle(z)$. Then, if $h_{ij}$ is the metric restricted to $\mathcal{V}$, 
\begin{align}\label{bellalista}
&\quad\ \disp \tr_{\mathcal{V}}\big( \bar \nabla^2 \bar\phi_\eps\big)(z_\eps) = \disp h^{ij} \bar \phi_{ij} + \frac{1}{\eps} h^{ij} (|y|^2)_{ij} \\
\notag & =  \big[h^{ij}\big(\partial^2_{ij} \phi + \bar\Gamma^k_{ij}\partial_k\phi\big)\big](z_\eps) + \frac{2y_\eps^\alpha}{\eps} \big[h^{ij} \bar \Gamma_{ij}^\alpha\big](z_\eps) \\
\notag & \le  \big[h^{ij}\big(\partial^2_{ij} \phi + \bar \Gamma^k_{ij}\partial_k\phi\big)\big](z_\eps) + \xi\big(\phi(x_\eps) + M_\eps\big) \big[|h^{ij} \bar \Gamma_{ij}^\alpha|\big](z_\eps) \\
\notag & \le  \big[h^{ij}\big(\partial^2_{ij} \phi + \Gamma^k_{ij}\partial_k\phi\big)\big](0) + \xi\big(\phi(0)\big) \big[|h^{ij} \bar \Gamma_{ij}^\alpha|\big](0) + o(1) \\
\notag & \le  \tr_{\mathcal{V}}\big( \bar \nabla^2 \phi\big)(0) + \xi\big(\phi(0)\big) |\tr_{\mathcal{V}} \II|(0) + o(1)
\end{align}
as $\eps \ra 0$. Because of \eqref{condizzzion} and the min-max characterization of eigenvalues, there exists an absolute constant $c>0$ such that 
\begin{align*}
\disp \tr_{\mathcal{V}}\big(\bar \nabla^2 \phi\big)(0) & \ge \disp \disp \tr_{\mathcal{V}}\big( \bar \nabla^2  \bar\phi_\eps\big)(z_\eps) - \xi\big(\phi(0)\big) |\tr_{\mathcal{V}} \II|(0) + o(1) \\
& \ge  f\big( \bar \phi(z_\eps) + M_\eps \big) - c\xi\big(\phi(0)\big)\|\II\|_\infty + o(1) \\
& \ge  f\big( \phi(0) \big) - c\xi\big(\phi(0)\big)\|\II\|_\infty + o(1) 
\doteq g\big(\phi(0)\big) + o(1)
\end{align*}
as $\eps \ra 0$. Passing to the limit as $\eps \ra 0$ we get $J^2_{q_0}\phi \in F_g \cap E_\xi$, as claimed.
\end{proof}

If, instead of $\widetilde{F_f} \cap \widetilde{E_\xi}$, we investigate the Ahlfors property just for $\widetilde{F_f}$, the situation is much more rigid because the last term in the second line of \eqref{bellalista} is controlled as $\eps \ra 0$ if and only if $\tr_{\mathcal{V}} \II(0) = 0$. Therefore, according to the observation in Remark \eqref{rem_opportuno}, in this case the Ahlfors property for $\widetilde{F_f}$ is inherited by proper submanifolds $X$ provided that $\tr_{\mathcal{V}} \II(x) = 0$ for each $x \in X$ and $k$-dimensional $\mathcal{V} \le T_xX$. An interesting case is when $k=m$, for which we have the following

\begin{corollary}\label{cor_immersionsLaplacian_weak}
Let $\sigma : X^m \ra Y^n$ be a proper minimal immersion, and suppose that $Y^n$ has the Ahlfors property for the subequation
\begin{equation*}\label{ambiente}
\big\{ \lambda_{n-m+1} + \cdots + \lambda_n(A) \ge -f(-r)\big\},
\end{equation*}
for some $f$ satisfying $(f1)$. Then, $X$ has the viscosity, weak Laplacian principle (equivalently, it is stochastically complete). 
\end{corollary}

The equivalence between the weak Laplacian principle and the stochastic completeness is discussed in Remark \ref{prop_weak_classicalvisco}.

\subsection*{A sufficient condition}

We now focus on conditions for the validity of the Ahlfors property for the operator $\widetilde{F_f}$ in Example $(\EE 5)$. Our purpose is to prove a generalized version of Proposition \ref{cor_sufficientiSMP}, Theorem \ref{prop_sufficientiFk} below. We recall Definition \ref{def_medioricci} of the $k$-th Ricci curvature. We begin with the following comparison theorem for the partial trace of $\nabla^2 \rho$: although it might be known to experts, we include a quick proof since we found no precise reference.

\begin{proposition}\label{prop_medioricci}
Let $X^m$ be a complete manifold, fix an origin $o$ and let $\rho(x) = \mathrm{dist}(x,o)$. Assume that 
\begin{equation}\label{ine_riccimedio_0}
\Ricc^{(k)}_x(\nabla \rho) \ge - G^2\big(\rho(x)\big) \qquad \forall \, x \not \in \cut(o),
\end{equation}
for some $k \in \{2, \ldots, m-2\}$ and some $G \in C(\R^+_0)$. Let $g \in C^2(\R^+_0)$ solve
\begin{equation}\label{eq_g}
\begin{cases}
g'' -G^2 g \ge 0 \qquad \text{on } \, \R^+, \\
g(0)=0, \quad g'(0) \ge 1.
\end{cases}
\end{equation}
If $g' > 0$ on $\R^+$, then 
\begin{equation}\label{hessrho_medioricci_2}
\lambda_{m-k}(\nabla^2 \rho) + \cdots + \lambda_m(\nabla^2 \rho) \le (k+1) \frac{g'(\rho)}{g(\rho)}
\end{equation}
in the viscosity sense on $X \backslash \{o\}$. 


\end{proposition}

\begin{remark}
\emph{The extreme cases $k=1$ and $k=m-1$ correspond, respectively, to the classical Hessian and Laplacian comparison theorems, for which we refer to \cite[Sect. 2]{prs} and \cite[Sect. 1.2]{bmr2} in our needed generality. However, the conclusion in these cases are stronger than the corresponding in \eqref{hessrho_medioricci_2}: more precisely, $g' > 0$ on $\R^+$ is not needed and 
\begin{equation}\label{maxprinci_HessLapla}
\begin{aligned}
\text{if } \, k = 1, & \quad \text{then} \quad  \lambda_m(\nabla^2 \rho) \le \frac{g'(\rho)}{g(\rho)}, \\
\text{if } \, k=m-1, & \quad \text{then} \quad \tr(\nabla^2 \rho) \le (m-1)\frac{g'(\rho)}{g(\rho)}.
\end{aligned}
\end{equation}
The reason of the ``shift" between \eqref{hessrho_medioricci_2} and \eqref{maxprinci_HessLapla} will be apparent in the proof below.
}
\end{remark}

\begin{proof}
We use the approach to comparison theorems via Riccati equations, due to J. Eschenburg and E. Heintze (see \cite{prs, bmr2}). Let $R$ be the curvature tensor of $X$. Fix $x \not \in \cut(o)$ and pick the unique minimizing, unit speed geodesic $\gamma : [0, \rho(x)]\ra X$ from $o$ to $x$. Let $\{E_\alpha\}$, $2 \le \alpha,\beta \le m$ be a parallel orthonormal basis of $\gamma'^\perp$ along $\gamma$, and let $B_{\alpha\beta} = \nabla^2 \rho(E_\alpha, E_\beta)$. Differentiating twice the identity $|\nabla \rho|^2 =1$ and recalling Ricci commutation rules, the matrix $B : (0,\rho(x)] \ra \mathrm{Sym}^2(\R^{m-1})$, $B = (B_{\alpha\beta})$ satisfies 
\begin{equation}\label{riccati}
\begin{cases}
B' + B^2 + R_\gamma = 0 &  \text{on } \, (0, \rho(x)], \\
B(t) = \frac{1}{t}I + o(1) & \text{as } \, t \ra 0^+,
\end{cases}
\end{equation}
where $(R_\gamma)_{\alpha\beta} = R(\gamma', E_\alpha, \gamma', E_\beta)$. Let $\mathcal{V}$ be a $k$-dimensional subspace of $(\gamma')^\perp$ at $x$, spanned by an orthonormal basis $\{V_j\}$, and extend each $V_j$ by parallel translation along $\gamma$. From \eqref{ine_riccimedio_0} we obtain
$$
\sum_{j=1}^k \langle R_\gamma V_j, V_j \rangle \ge k \Ricc^{(k)}(\gamma') \ge -k G^2(t).
$$
Tracing \eqref{riccati} on $\mathcal{V}$ and using the last inequality we deduce that the function 
$$
\theta(t) \doteq \frac{1}{k}\sum_{j=1}^k \langle B V_j, V_j \rangle(t)
$$
satisfies
$$
\begin{cases}
\disp \theta'(t) + k^{-1}\sum_{j=1}^k \langle B^2V_j,V_j \rangle - G^2(t) \le 0 & \quad \text{on } \, (0, \rho(x)], \\
\theta(t) = \frac{1}{t} + o(1) & \quad \text{as } t \ra 0^+.
\end{cases}
$$
Denote with $\pi = TX \ra \mathcal{V}$ the orthogonal projection onto $\mathcal{V}$, and let $\bar B = \pi \circ B_{|\mathcal{V}} : \mathcal{V} \ra \mathcal{V}$. Using Newton's inequality, we get
$$
\sum_{j=1}^k \langle B^2V_j,V_j \rangle \ge \sum_{j=1}^k \langle \bar B^2 V_j,V_j \rangle \ge k \theta^2,
$$
hence 
\begin{equation}\label{inericcati}
\begin{cases}
\disp \theta' + \theta^2 - G^2(t) \le 0 & \quad \text{on } \, (0, r(x)], \\
\theta(t) = \frac{1}{t}+ o(1) & \quad \text{as } \, t \ra 0^+.
\end{cases}
\end{equation}
By Sturm comparison (see \cite[Thm. 1.9]{bmr2}), it is sufficient to prove \eqref{hessrho_medioricci_2} under the assumption that $g$ satisfies \eqref{eq_g} with the equality sign and $g'(0)=1$. Since $\bar \theta(t) = g'(t)/g(t)$ solves
$$
 \begin{cases}
\disp \bar \theta' + \bar \theta^2 - G^2(t) = 0 &  \text{on } \, (0, r(x)], \\
 \bar \theta(t) = \frac{1}{t}+ o(1) &  \text{as } \, t \ra 0^+, 
\end{cases}
$$
the Riccati comparison theorem implies $\theta(t) \le \bar \theta(t)$, and because of the arbitrariness of $\mathcal{V}$ and the min-max characterization of eigenvalues, 
\begin{equation}\label{compa_medioricci}
\lambda_{m-k+1}(B) + \cdots + \lambda_m(B) \le k \frac{g'(\rho)}{g(\rho)} \qquad \text{for } x \not \in \cut(o).
\end{equation}
Identity $\nabla^2 \rho(\nabla \rho, \cdot) =0$ shows that the eigenvalues of $\nabla^2 \rho$ are $0$ and the eigenvalues of $B$. Let $\mathcal{V}$ be a $(k+1)$-dimensional subspace of $T_xX$  spanned by an orthonormal basis $\{v_0, \ldots, v_k\}$. We can choose the basis in such a way that
$$
v_0 = (\cos \psi)\nabla \rho + (\sin \psi) e_0, \qquad v_j = e_j \qquad \text{for } \, 1 \le j \le k,
$$
and $\{e_j\}_{j=0}^k$ is an orthonormal set in $\nabla \rho^\perp$. Therefore, at $x$,
\begin{align}\label{traccia_hesscomp}
\disp \tr_{\mathcal{V}}(\nabla^2 \rho) & = \disp (\sin^2\psi) \nabla^2 \rho(e_0,e_0) + \sum_{j=1}^k \nabla^2  \rho(e_j,e_j) \\
\notag & = \disp (\sin^2 \psi)\!\left(\sum_{j=0}^k \nabla^2 \rho(e_j,e_j)\!\right)\! + (\cos^2 \psi) \!\left(\sum_{j=1}^k \nabla^2 \rho(e_j,e_j)\!\right)\!.
\end{align}
The space $\mathcal{W}$ generated by $\{e_j\}$ is orthogonal to $\nabla \rho$, and from implication 
$$
\Ricc^{(k)}(\nabla \rho) \ge -G(\rho) \quad \Longrightarrow \quad \Ricc^{(k+1)}(\nabla \rho) \ge -G(\rho)
$$
we can apply \eqref{compa_medioricci} both for $k$ and for $(k+1)$ to infer from \eqref{traccia_hesscomp} the inequality
\begin{equation*}
\tr_{\mathcal{V}}(\nabla^2 \rho) \le (\sin^2 \psi) (k+1)\frac{g'(\rho)}{g(\rho)} + (\cos^2 \psi) k \frac{g'(\rho)}{g(\rho)} \le (k+1)\frac{g'(\rho)}{g(\rho)}.
\end{equation*}
Note that, in the last step, we used $g' > 0$. This concludes the proof of \eqref{hessrho_medioricci_2} for $x \not \in \cut(o)$. To prove that \eqref{hessrho_medioricci_2} holds in the viscosity sense on $X \backslash \{o\}$, we rely as usual on Calabi's trick: briefly, for $x \in \cut(o)$ consider a minimizing, unit speed geodesic $\gamma$ from $o$ to $x$. Given a small $\eps >0$, we set $o_\eps = \gamma(\eps)$ and $\rho_\eps = \eps +\mathrm{dist}(o_\eps, \cdot)$. Then $\rho_\eps$ is smooth around $x$, $\rho_\eps \ge \rho$ on $X$ and equality holds at $x$. Furthermore, repeating the above estimates with $\rho_\eps$ replacing $\rho$, and noting that $G_\eps(t) = G(t-\eps)$ and $\nabla \rho_\eps= \nabla \rho$ along $\gamma$, we deduce at the point $x$ the inequality
\begin{equation*}
\lambda_{m-k}(\nabla^2 \rho_\eps) + \cdots + \lambda_m(\nabla^2 \rho_\eps) \le (k+1) \frac{g_\eps'(\rho_\eps)}{g_\eps(\rho_\eps)} = (k+1) \frac{g_\eps'(\rho)}{g_\eps(\rho)},
\end{equation*}
where $g_\eps$ solves \eqref{eq_g} with $G_\eps$ in place of $G$, with the equality sign and with $g_\eps'(0)=1$. Note that $g_\eps \ra g$ in $C^1_{\mathrm{loc}}(\R^+)$, since we assumed that $g$ satisfies \eqref{eq_g} with equalities, and therefore $g_\eps'>0$ on $(0, \rho(x)]$ for small enough $\eps$. To conclude, using that any test function $\phi$ touching $\rho$ from below at $x$ also touches $\rho_\eps$ from below, from $\nabla^2 \phi \le \nabla^2 \rho_\eps$ at $x$ we get  
$$
\lambda_{m-k}(\nabla^2 \phi) + \cdots + \lambda_m(\nabla^2 \phi) \le (k+1) \frac{g_\eps'(\rho)}{g_\eps(\rho)} \ra (k+1) \frac{g'(\rho)}{g(\rho)}
$$
as $\eps \ra 0$. This shows \eqref{hessrho_medioricci_2} in the viscosity sense.
\end{proof}

We are ready to prove the following criterion. 

\begin{theorem}\label{prop_sufficientiFk}
Let $Y^n$ be a complete manifold, fix an origin $o$ and let $\rho(x) = \mathrm{dist}(x,o)$. Assume that 
\begin{equation}\label{ine_riccimedio}
\Ricc^{(k)}_x(\nabla \rho) \ge - G^2\big(\rho(x)\big) \qquad \forall \, x \not \in \cut(o),
\end{equation}
for some $k \in \{2, \ldots, n-1\}$ and some $G$ satisfying
$$
0 < G \in C^1(\R^+_0), \qquad G' \ge 0, \qquad G^{-1} \not \in L^1(+\infty).
$$
Then, $Y$ has the Ahlfors property for $\widetilde{F_f} \cap \widetilde{E_\xi}$, with 
\begin{equation}\label{eq_Fk_medioricci}
F_f= \big\{ \lambda_1(A) + \cdots + \lambda_{k+1}(A) \ge f(r)\big\}, 
\end{equation}
for each $(f,\xi)$ satisfying $(f1 + \xi 1)$.\\
Moreover, if $\sigma : X^m \ra Y^n$ is a proper isometric immersion, $k+1\le m \le n-1$ and the eigenvalues $\mu_1 \le \cdots \le \mu_m$ of the second fundamental form $\II$ satisfy 
\begin{equation}\label{ipo_tracepatrial_medioricci}
\max\Big\{ \big|\mu_1 + \cdots + \mu_{k+1}\big|, \big|\mu_{m-k} + \cdots + \mu_m\big|\Big\} \le C G(\rho \circ \sigma)\quad \text{on } \, X,
\end{equation}
for some constant $C>0$, then the Ahlfors property for $\widetilde{F_f} \cap \widetilde{E_\xi}$ holds on $X$ with $F_f$ in \eqref{eq_Fk_medioricci} and each $(f,\xi)$ satisfying $(f1 + \xi 1)$. In particular, if $m=k+1$ and the mean curvature satisfies 
$$
\big|H\big| \le C G(\rho \circ \sigma),
$$
then $X$ has the viscosity, strong Laplacian principle.
\end{theorem}

\begin{remark}\label{rem_extreme}
\emph{The extreme case $k=1$ can be handled with straightforward modifications in the proof below. Briefly, if \eqref{ine_riccimedio} holds with $k=1$ then, using the Hessian comparison theorem instead of Proposition \ref{prop_medioricci}, the Ahlfors property holds for $\widetilde{F_f} \cup \widetilde{E_\xi}$ with $F_f = \{ \lambda_1(A) \ge f(r)\}$. In other words, the viscosity, Hessian principle holds. The second part of the theorem is unchanged, that is, each proper immersion $\sigma : X^m \ra Y^n$ satisfying \eqref{ipo_tracepatrial_medioricci} with $k=1$ has the Ahlfors property for $\widetilde{F_f} \cup \widetilde{E_\xi}$, with 
$$
F_f = \big\{ \lambda_1(A) + \lambda_2(A) \ge f(r)\big\}.
$$
Furthermore, if instead of \eqref{ipo_tracepatrial_medioricci} we require
\begin{equation}\label{bbb}
|\II| \le C G(\rho \circ \sigma), 
\end{equation}
then the Hessian principle transplants to $X$. However, this last conclusion is immediate from the first part since, by Gauss equation, up to a constant a submanifold satisfying \eqref{bbb} inherits from the ambient space the bound \eqref{ine_riccimedio} with $k=1$. In this respect, see also Examples 1.13 and 1.14 in \cite{prsmemoirs}.
}
\end{remark}


\begin{proof}
Set 
$$
\eta(t) = - \int_0^{t} \frac{\di s}{G(s)}, \qquad w(x)=\eta\big(\rho(x)\big). 
$$
Then, $0 \ge w(x) \ra -\infty$ as $x$ diverges. From $G' \ge 0$ we deduce $\eta'' \ge 0$, and thus for $x \not \in \cut(o)$
\begin{equation}\label{eq_nabladiw_medioricci}
\nabla^2 w = \eta'' \di \rho \otimes \di \rho + \eta' \nabla^2 \rho \ge \eta' \nabla^2 \rho = -\frac{1}{G(\rho)} \nabla^2 \rho.
\end{equation}
Suppose that $k \in \{2, \ldots, n-2\}$. In our assumptions, Proposition \ref{prop_medioricci} ensures the validity of \eqref{hessrho_medioricci_2} on $Y \backslash \{o\}$ in the viscosity sense, in particular with $g$ solving 
$$
\begin{cases}
g'' - G^2 g =0 \quad \text{on } \, \R^+, \\
g(0)=0, \quad g'(0)=1
\end{cases}
$$
(note that $G \ge 0$ implies $g'>0$ on $\R^+$). If $k = n-1$, the same holds with the second in \eqref{maxprinci_HessLapla} replacing \eqref{hessrho_medioricci_2}. In both of the cases, the function $\theta = g'/g$ satisfies the Riccati equation $\theta' + \theta^2 = G^2$ on $\R^+$, from which we deduce that $\theta \le G$ whenever $\theta' \ge 0$. Since $G' \ge 0$ and consequently $G \ge G(0)>0$, it is easy to infer the existence of a constant $C_1>0$ such that $\theta \le C_1G$ on $[1,+\infty)$. Using \eqref{eq_nabladiw_medioricci} we get (say, for $k \in \{2, \ldots, n-2\}$):
\begin{align}\label{ineq_w_madioricci}
& \disp \lambda_1(\nabla^2 w)+ \cdots + \lambda_{k+1}(\nabla^2 w)\\
\notag &\ge  \disp G(\rho)^{-1} \Big[ \lambda_1(-\nabla^2 \rho)+ \cdots + \lambda_{k+1}(-\nabla^2 \rho)\Big] \\
\notag & \ge  \disp -\frac{(k+1) \theta(\rho)}{G(\rho)} \ge -(k+1) C_1,
\end{align}
at points of $Y \backslash B_1(o)$ where $\rho$ is smooth. Since $\eta' \le 0$, we can still use Calabi's trick as in Proposition \ref{prop_medioricci} to check that the same inequality holds in the viscosity sense on the entire $Y \backslash B_1(o)$. Moreover, $|\nabla w| \le 1/G(0) \doteq C_2$ on $Y$.   Fix two linear functions $f,\xi$ satisfying $(f1+\xi 1)$ and such that, $f(r) \le -(k+1)C_1$, $\xi(r) \ge C_2$ for $r < \eta(1)$. Then, by construction $w \in (F_f \cap E_\xi)\big( Y \backslash B_1(o)\big)$ with $F_f$ as in \eqref{eq_Fk_medioricci}. Because of the linearity of $(f,\xi)$, the family $\{\delta w\}$ for constant $\delta \in (0,1)$ is a family of Khas'minskii potentials for $F_f \cap E_\xi$, and the desired Ahlfors property for $\widetilde{F_f} \cup \widetilde{E_\xi}$ follows by AK-duality and by Proposition \ref{prop_equivalenceahlfors_withgradient}.\\
To prove the last part of the theorem, let $\sigma : X^m \ra Y^n$ be a proper isometric immersion. If the right-hand side of \eqref{ipo_tracepatrial_medioricci} were constant, we could directly apply Theorem \ref{teo_immersions} and Remark \ref{rem_opportuno} to deduce that $X$ has the Ahlfors property for $\widetilde{F_f} \cup \widetilde{E_\xi}$. However, in the above assumptions it is more convenient to look again at the equation satisfied by the transplanted function $\bar w = w \circ \sigma$. If $\rho$ is smooth around $\sigma(x)$, tracing the identity
$$
\bar \nabla^2 \bar w(v,v) = \nabla^2 w\big(\sigma_* v, \sigma_* v) + \langle \nabla w, \II(v,v) \rangle
$$
on a $(k+1)$-dimensional subspace $\mathcal{V}$, and using \eqref{ipo_tracepatrial_medioricci} and \eqref{ineq_w_madioricci}, we deduce 
\begin{align}\label{rrrmmm}
	\tr_{\mathcal{V}}\big(\bar \nabla^2 \bar w\big) &\ge \tr_{\sigma_*\mathcal{V}}\big(\nabla^2 w\big) - |\nabla w| \big| \tr_{\mathcal{V}}\II\big|\\
	\notag &\ge - (k+1) C_1 -\frac{|\nabla \rho|}{G(\rho)} CG(\rho) = -(k+1)C_1- C.
\end{align}
If $\rho$ is not smooth around $\sigma(x)$, we use again Calabi's trick to $\rho$. Clearly, $|\bar \nabla \bar w| \le |\nabla w| \le C_2$. By min-max, and suitably changing the linear function $f$, the family $\{\delta \bar w\}$ still guarantees the Khas'minskii property for $F_f \cap E_\xi$ on $X$, as required. If $m = k+1$, the viscosity, strong Laplacian principle follows from Theorem \ref{teo_Laplacian_intro}. 
\end{proof}

\begin{proof}[Proof of Proposition \ref{cor_sufficientiSMP}]
Cases (i) and (ii) follow by using Theorem \ref{prop_sufficientiFk} with the modifications in Remark \ref{rem_extreme} to, respectively, $k=1$ and $k=n-1$. Case (iii) follows from the same proposition applied with $m = k+1$, either directly (if $m \ge 3$) or again with the aid of Remark \ref{rem_extreme} (if $m=2$).
\end{proof}

\subsection*{Submersions}

We briefly recall some general facts. Let $\pi : X^m \ra Y^n$ be a submersion between Riemannian manifolds, and set $X_y \doteq \pi^{-1}\{y\}$ to denote the fiber over a point $y \in Y$. We recall that $\pi$ is called a Riemannian submersion if, for each $x \in X$, $\pi_{*,x}$ is an isometry when restricted to the orthogonal complement of $T_x X_{\pi(x)}$ in $T_x X$, called the horizontal subspace at $x$. Let $\mathcal{D}$ be the horizontal distribution, that is, $\mathcal{D}_x \doteq T_x X_{\pi(x)}^ \perp$. Let $\nabla, \bar \nabla$ be the Levi-Civita connections on $Y$ and $X$, respectively. For each vector field $V \in TX$, denote with $V^h$ and $V^v$ its projections on the horizontal and vertical subspaces, respectively. Given $W \in TY$, there exists a unique, smooth horizontal vector field $\bar W$ such that $\pi_*\bar W = W$, called the horizontal lift of $W$. Using the explicit formula for the Levi-Civita connection one checks that, if $V,W \in TY$ have horizontal lifts $\bar V, \bar W$,
\begin{equation*}
\bar \nabla_{\bar V} \bar W = \overline{ \nabla_V W} + \mathcal{A}(\bar V, \bar W),
\end{equation*}
with 
$$
\mathcal{A} : \cal D \times \cal D \ra \cal D^\perp, \qquad \cal A(\bar V, \bar W) = \frac{1}{2}[\bar V, \bar W]^v
$$
being the integrability tensor for the distribution $\cal D$ (see \cite{petersen}). For $y \in Y$, let $\II_y : \mathcal{D}^\perp \times \mathcal{D}^\perp \ra \mathcal{D}$ be the second fundamental form of the fiber $X_y$.\\
We are ready to state the following result, for which we feel convenient to redefine 
\begin{align}\label{exempliE3E5}
&F_f^k = \Big\{ \lambda_k(A) \ge f(r) \Big\} \quad \text{for $(\EE 3)$,} \\
\notag &F_f^k = \Big\{ \lambda_1(A) + \cdots + \lambda_k(A) \ge f(r)\Big\} \quad \text{for $(\EE 5)$},
\end{align}
to make explicit the dependence on $k$.

\begin{theorem}\label{teo_submersions}
Let $\pi : X^m \ra Y^n$ be a Riemannian submersion with compact fibers. Suppose that 
\begin{equation}\label{condi_subme}
\sup_{y \in Y} \|\II_y\|_\infty + \|A\|_\infty < +\infty.
\end{equation}
Fix $k \in \{1,\ldots, n\}$, and consider one of the subequations in \eqref{exempliE3E5} or their complex analogues. Assume that $f$ satisfies $(f1)$, and let $E_\xi$ with $\xi$ satisfying $(\xi 1)$. Then,  
\begin{itemize}
\item[$(i)$] In $(\EE 3)$, 
\begin{align*}
&\text{$\widetilde{F_f^k} \cup \widetilde{E_\xi}$ has the Ahlfors property on $Y$}\\
&\qquad \Longleftrightarrow \quad \text{$\widetilde{F_f^k} \cup \widetilde{E_\xi}$ has the Ahlfors property on $X$.}
\end{align*}
\item[$(ii)$] In $(\EE 5)$, 
\begin{align*}
&\text{$\widetilde{F_f^k} \cup \widetilde{E_\xi}$ has the Ahlfors property on $Y$} \\
&\qquad \Longleftrightarrow \quad \text{$\widetilde{F_f^{k+m-n}} \cup \widetilde{E_\xi}$ has the Ahlfors property on $X$}.
\end{align*}
\end{itemize}
For implication $\Leftarrow$, fibers do not need to be compact.
\end{theorem}

\begin{proof}
As for immersions, the simple idea for both implications is to transplant the relevant functions from $Y$ to $X$. More precisely, setting $j=k$ for $(\EE 3)$ and $j=k+m-n$ for $(\EE 5)$,
\begin{itemize}
\item[$(\Leftarrow)$] if, by contradiction, the Ahlfors property on $Y$ does not hold, in view of Proposition \ref{teo_main_A_L} there exists $w \in (\widetilde{F_f^k} \cup \widetilde{E_\xi})(Y)$ bounded, non-negative and non-constant. We shall prove that $\bar w \doteq w \circ \pi$ is $(\widetilde{F_g^j} \cup \widetilde{E_\xi})$-subharmonic, for suitable $g$ satisfying $(f1)$, that contradicts the Ahlfors property on $X$;
\item[$(\Rightarrow)$] by AK-duality in Theorem \ref{cor_bonito}, $F_f^k \cap E_\xi$ has the Khas'minskii property on $Y$, and we can thus fix a Khas'minskii potential $w$, defined on $Y \backslash K$ and $(F_f^k \cap E_\xi)$-subharmonic. We shall prove that $\bar w \doteq w \circ \pi$ is $(F_g^j \cap E_\xi)$-subharmonic, for a suitable $g$ satisfying $(f1)$. Since fibers of $\pi$  are compact, $\bar w$ is an exhaustion, and the conclusion is then reached as in Theorem \ref{teo_immersions}.
\end{itemize}
In either cases, let $\phi$ be a test function for $\bar w$ at $x_0 \in X$. Without loss of generality, we can assume that $\phi$ is constant on fibers of $\pi$\footnote{Indeed, consider a local chart $U$ on $X$ with coordinates $x=(y,z)$ around $x_0=(y_0,z_0)$, such that fibers of $\pi$ are given by constant $y$. In such a chart, $\bar w$ just depends on $y$ and, since $\phi$ touches $\bar w$ from above at $x_0$, up to a slight perturbation we can assume that $z_0$ is a strict minimum for $\phi$ on $\{y=y_0\}$ and that $\partial^2_z \phi$ is positive definite at $x_0$. Up to shrinking $U$, by the implicit function theorem there exists a smooth map $y \mapsto z(y)$ such that $\partial_z \phi(y, z(y))=0$, and by continuity $z(y)$ is the unique minimum point of $\phi$ on the fiber over $y$ in $U$. We can therefore choose $\bar \phi(y) = \phi(y,z(y))$ as a test function in place of $\phi$, since by construction $\bar w \le \bar \phi \le \phi$ on $U$.}. Let $V^v$ and $V^h$ denote, respectively, a vertical and a horizontal vector field around $x_0$. To compute $\bar \nabla^2 \phi(V^v,V^v)$, let $i_{x_0} : X_{\pi(x_0)} \ra X$ be the inclusion, and let $D$ be the Riemannian connection of $X_{\pi(x_0)}$. Since $\phi$ is constant on fibers, we have
\begin{equation}\label{eqsubme}
\begin{aligned}
0 & =  \disp \langle \bar \nabla \phi, V^v\rangle \quad \text{around $x_0$, and} \\[0.2cm]
0 & =  \disp D^2(\phi \circ i_{x_0})(V^v,V^v) = \bar \nabla^2 \phi(V^v,V^v) + \langle \bar \nabla \phi, \II(V^v,V^v) \rangle \quad \text{at } \, x_0.
\end{aligned}
\end{equation}
In particular, $\bar \nabla \phi$ is horizontal in a neighborhood of $x_0$. As for $\bar \nabla^2 \phi(V^ h, V^v)$, at the point $x_0$ we have
\begin{align}\label{terceira_4}
	\bar \nabla^2 \phi(V^h,V^v) &\doteq \langle \bar \nabla_{V^h} \bar \nabla \phi, V^v \rangle \\
	\notag &= \langle \overline{ \nabla_{\pi_* V^h} \pi_*\bar \nabla \phi} + \cal A( V^h, \bar \nabla \phi), V^v \rangle = \langle \cal A( V^h, \bar \nabla \phi), V^v \rangle.
\end{align}
To compute $\bar \nabla^2 \phi(V^h, V^h)$, we take a section $\sigma : U \subset Y \ra X$ around $y_0 = \pi(x_0)$, that is, a smooth map satisfying $\pi \circ \sigma = \mathrm{id}_U$, with the properties $\sigma(y_0) = x_0$ and $\sigma_*(T_{y_0}Y) = \cal D_{x_0}$. Set $v^h = \pi_* V^h$. From $\pi_*(\sigma_*z^h) = z^h$ for each $z^h \in TY$ we deduce
\begin{equation}\label{simm}
\sigma_* z^h = Z^h + Z^v, 
\end{equation}
where $Z^h$ lifts $z^h$ and $Z^v$ is some vertical vector such that $Z^ v(x_0)=0$. By construction, note that 
\begin{equation}\label{gradient}
|\nabla(\phi \circ \sigma)| = |\bar \nabla \phi| \qquad \text{at } \, y_0.
\end{equation}
We compute
\begin{equation}\label{pripri}
\nabla^2 (\phi \circ \sigma)(v^h, v^h) = \bar \nabla^2 \phi( \sigma_* v^h, \sigma_* v^h) + \langle \bar \nabla \phi, \nabla^2 \sigma (v^h,v^h) \rangle, 
\end{equation}
where $\nabla^2 \sigma$ is the generalized second fundamental form of the map $\sigma$, defined by  
$$
\nabla^2 \sigma(v^h, v^h) \doteq \bar \nabla_{\sigma_* v^h}\big( \sigma_* v^h\big) - \sigma_*\big( \nabla_{v^h}v^h\big).
$$
Since $\bar \nabla \phi$ is horizontal at $x_0$ and $V^v(x_0)=0$, using \eqref{simm} we compute at $y_0$
$$
\nabla^2 \sigma(v^h,v^h)  =  \disp \bar \nabla_{V^h}\big(V^h + V^v\big) - \overline{ \nabla_{v^h}v^ h} = \cal A(V^h,V^h) + \bar \nabla_{V^h}V^v. 
$$
Note that $\cal A(V^h, V^h)$ is vertical, and also\footnote{For $Z^h$ horizontal, since $V^v(x_0)=0$ we get $\langle \bar \nabla_{V^h}V^v, Z^h\rangle = - \langle V^v, \bar \nabla_{V^h}Z^v\rangle =0$ at $x_0$.} $\bar \nabla_{V^h}V^v$ at $x_0$, hence
$$
\langle \bar \nabla \phi, \nabla^2 \sigma (v^h,v^h) \rangle = 0 \qquad \text{at } \, x_0.
$$
Inserting into \eqref{pripri} we deduce
\begin{equation}\label{segunda_4}
\nabla^2 (\phi \circ \sigma)(v^h, v^h) = \bar \nabla^2 \phi( V^h, V^h) \qquad \text{at } \, y_0.
\end{equation}
Summarizing, decomposing a unit vector $V \in T_{x_0}X$ into its horizontal and vertical part, by \eqref{eqsubme}, \eqref{terceira_4} and \eqref{segunda_4} we obtain 
\begin{align}\label{fundi}
\bar \nabla^2 \phi(V,V) & =  \disp \bar \nabla^2 \phi\big(V^h, V^h\big) + 2 \bar \nabla^2 \phi(V^h, V^v) + \bar \nabla^2 \phi(V^v, V^v) \\
\notag & =  \nabla^2 (\phi \circ \sigma)\big( \pi_* V^ h, \pi_* V^ h\big)\\
\notag &\quad + 2\langle \cal A(V^h, \bar \nabla \phi), V^v \rangle - \langle \bar \nabla \phi, \II(V^v,V^v) \rangle \\
\notag & \ge  \nabla^2 (\phi \circ \sigma)\big( \pi_* V^ h, \pi_* V^ h\big)\\
\notag &\quad - |\nabla (\phi \circ \sigma)|\Big(2|\cal A||V^h||V^v| + |\II(V^v, V^v)|\Big) .
\end{align}
Define 
$$
C_\infty \doteq \sup_{y \in Y} \|\II_y\|_\infty + \|A\|_\infty.
$$
If we assume that $J^2_{y_0}(\phi \circ \sigma) \in E_\xi$, from \eqref{fundi} we obtain
\begin{equation}\label{linda}
\bar \nabla^2 \phi + \xi(\phi)C_\infty \metric \ge \Pi^*(\nabla^2 \phi)\Pi \qquad \text{as a quadratic form,}
\end{equation}
where $\Pi, \Pi^*$ are the projection onto the horizontal subspace and its adjoint, and $\metric$ is the metric of $X$. Note that, because of \eqref{segunda_4}, the eigenvalues of $\Pi^*(\nabla^2 \phi)\Pi$ at $x_0$ are $\{ \lambda_1,\ldots, \lambda_n, 0, \ldots, 0\}$, where $\{\lambda_i\}$ are the ordered eigenvalues of $\nabla^2 (\phi \circ \sigma)$ at $y_0$.\\[0.2cm]
We first examine case $(\EE 5)$, implication $(\Rightarrow)$, so let us assume that $w \in (F^k_f \cap E_\xi)(Y\backslash K)$ is a Khas'minskii potential. Clearly, since $\phi \circ \sigma$ is a test for $w$ at $y_0$, from \eqref{gradient} we deduce $J^2_{x_0}\phi \in E_\xi$. Let $\cal V \le T_{x_0}X$ be a $j$-dimensional subspace, $j = m-n+k$. From \eqref{linda} and min-max we obtain
\begin{align}\label{linda_22}
	&\tr_{\cal V}\bar \nabla^2 \phi + (m-n+k)\xi(\phi)C_\infty\\
	\notag &\ge \sum \Big\{\text{first $(m-n+k)$-eigenvalues of $\Pi^*(\nabla^2 \phi)\Pi$}\Big\},
\end{align}
at $x_0$. By dimensional considerations this last is the sum of the first $l$ eigenvalues of $\nabla^2 (\phi \circ \sigma)$, for some $l \in \{k, \ldots, n\}$. Using that $\phi \circ \sigma$ is a test for $w$ at $y_0$,
$$
\lambda_1 + \cdots + \lambda_l \ge \frac{l}{k}(\lambda_1+ \cdots + \lambda_k) \ge \frac{l}{k}f(\phi) \ge \frac{n}{k}f(\phi), 
$$
where the last step follows from $\phi(x_0) \le 0$ and $(f1)$. Therefore, 
$$
\tr_{\cal V}\bar \nabla^2 \phi \ge \frac{n}{k}f(\phi) - (m-n+k)C_\infty \xi(\phi) \doteq g(\phi),
$$
where $g$ satisfies $(f1)$ and $J^2_{x_0}\phi \in F^j_g \cap E_\xi$, as required. For the reverse implication $(\Leftarrow)$, we assume that $w \in (\widetilde{F_f^k} \cup \widetilde{E_\xi})(Y)$ is bounded, non-negative and non-constant. In particular, the $2$-jet of $\phi \circ \sigma$ satisfies
$$
|\nabla(\phi\circ \sigma)| \ge \xi(-\phi) \qquad \text{or} \qquad \big\{ \lambda_{n-k+1} + \cdots + \lambda_n \ge -f(-\phi) \big\},
$$
for some pair $(f,\xi)$ satisfying $(f1)$, $(\xi 1)$. By Proposition \ref{prop_equivalenceahlfors_withgradient}, without loss of generality we can assume that the pair $(f,\xi)$ satisfies  
\begin{equation}\label{extrass}
f(s) + (m-n+k) C_\infty \xi(s) < 0 \quad \text{and is non-decreasing for $s<0$}.
\end{equation}
We claim that 
\begin{equation}\label{claimboa_prim}
J^2_{x_0}\phi \in \widetilde{F^k_g} \cup \widetilde{E_\xi}. \qquad \text{for } \quad g(s) = f(s) + (m-n+k) C_\infty \xi(s). 
\end{equation}
Observe that $g$ satisfies the assumptions in $(f1)$ because of \eqref{extrass}.\\
By \eqref{gradient}, if $J^2_{y_0}(\phi \circ \sigma) \in \widetilde{E_\xi}$ then also $J^2_{x_0}\phi \in \widetilde{E_\xi}$ and we are done. Otherwise, $J^2_{y_0}(\phi \circ \sigma) \in \widetilde{F_f^k} \cap \big\{ |p| \le \xi(-\phi)\big\}$. Consider the horizontal lift $\cal V$ of the span of the eigenspaces corresponding to the biggest $k$-eigenvalues $\{ \lambda_{n-k+1}, \ldots, \lambda_n\}$ of $\nabla^2 (\phi \circ \sigma)$. Using that $\vert \bar \nabla \phi\vert \leq \xi(-\phi)$, we can trace \eqref{linda} on $\cal V$ and deduce
$$
\tr_{\cal V}\bar \nabla^2 \phi + (m-n+k) C_\infty \xi(-\phi) \ge \lambda_{n-k+1} + \cdots + \lambda_{n} \ge -f(-\phi).
$$
Hence the sum of the last $k$ eigenvalues of $\bar \nabla^2 \phi(x_0)$ is at least 
$$
-f(-\phi) - (m-n+k)C_\infty \xi(-\phi) = -g(-\phi),
$$
that is, $J^2_{x_0}\phi \in \widetilde{F_g^j} \subset \widetilde{F^j_g} \cup \widetilde{E_\xi}$, as claimed. \\[0.2cm]
Case $(\EE 3)$ is dealt with similarly: for implication $(\Rightarrow)$, let $w \!\in\! (F_f^k \!\cap\! E_\xi)(Y \backslash K)$ be a Khas'minskii potential; we claim that 
\begin{equation}\label{claimboa}
J^2_{x_0}\phi \in F^k_g \cap E_\xi \qquad \text{for } \quad g(s) \doteq f(s) - C_\infty \xi(s). 
\end{equation}
Otherwise, by min-max we could take a $k$-dimensional subspace $\cal V$ where $\bar \nabla^2 \phi < g(\phi) \metric$. Evaluating \eqref{linda} on a unit vector $V \in \cal V$ would yield
$$
\nabla^2 (\phi \circ \sigma)(\pi_* V, \pi_* V) \le \bar \nabla^2 \phi(V,V) + C_\infty \xi(\phi) < g(\phi) + C_\infty \xi (\phi) = f(\phi) \le 0.
$$
By the arbitrariness of $V$, $\pi_* : \cal V \ra TY$ has trivial Kernel and thus $\pi_* \cal V$ is $k$-dimensional. From 
$$
\nabla^2 (\phi \circ \sigma)\left( \frac{\pi_* V}{|\pi_* V|}, \frac{\pi_* V}{|\pi_* V|}\right) \le \nabla^2  (\phi \circ \sigma) (\pi_* V, \pi_* V) < f(\phi)
$$
we obtain $\nabla^2 (\phi \circ \sigma) < f(\phi) ( \, , \, )$ on $\pi_*\cal V$, contradicting $J^2_{y_0}(\phi \circ \sigma)\in F_f^k$.\\
For $(\Leftarrow)$ we assume $w \in \widetilde{F_f^k} \cup \widetilde{E_\xi}$ and claim that
\begin{equation}\label{claimboa_2}
J^2_{x_0}\phi \in \widetilde{F^k_g} \cup \widetilde{E_\xi} \qquad \text{for } \quad g(s) \doteq f(s) + C_\infty \xi(s). 
\end{equation}
As before, we can assume that the pair $(f,\xi)$ is chosen in such a way that $g$ satisfies the assumptions in $(f1)$, and by \eqref{gradient} we can also restrict to the case $J^2_{y_0}(\phi \circ \sigma) \in \widetilde{F_f^k} \cap \big\{ |p| \le \xi(-\phi)\big\}$. If by contradiction we suppose that $J^2_{x_0}\phi \not \in \widetilde{F_g}$, we can select an $(m-k)$-dimensional $\cal V \le T_{x_0}X$ with
$$
\bar \nabla^2 \phi(V,V) < -g(-\phi) \quad \text{for each} \quad V \in \cal V, \ |V|=1.
$$
The intersection of $\cal V$ with the horizontal subspace has dimension at least $(n-k)$, call it $\bar{\cal V}$. For each $V \in \bar{\cal V}$, we have 
\begin{equation}\label{alhn}
\nabla^2 (\phi \circ \sigma)\big( \pi_*V,\pi_*V\big) < -g(-\phi) + C_\infty \xi(-\phi) = -f(-\phi),
\end{equation}
whence $\lambda_{n-k}(\nabla^2 (\phi \circ \sigma)) < -f(-\phi)$ by min-max, contradiction.
\end{proof}

When $k=n$ in $(ii)$ of Theorem \ref{teo_submersions}, condition \eqref{condi_subme} can be considerably weakened, and there is no need to bound $\cal A$. Indeed, following the above proof we obtain the next mild improvement of a result in \cite{bessapiccione, brandaooliveira}.

\begin{corollary} \label{cor_subme}
Let $\pi : X^m \ra Y^n$ be a Riemannian submersion with compact fibers, and denote with $H_y$ be the mean curvature of the fiber $X_y$. 
\begin{itemize}
\item[1)] If there exists $C>0$ such that $\|H_y\|_\infty \le C$ for each $y \in Y$, then
$$
\text{$Y$ has the viscosity, strong Laplacian principle if and only if so does $X$.}
$$
\item[2)] If fibers are minimal, then 
$$
\text{$Y$ has the viscosity, weak Laplacian principle if and only if so does $X$.}
$$
\end{itemize}
\end{corollary}

\section{Ahlfors property for \texorpdfstring{$\widetilde E$}{} and Ekeland maximum principle}

We are going to prove Theorem \ref{teo_ekeland_intro} in the introduction, which we rewrite for the convenience of the reader.

\begin{theorem}\label{teo_ekeland}
Let $X$ be a Riemannian manifold. Then, the following statements are equivalent:
\begin{itemize}
\item[(1)] $X$ is complete.
\item[(2)] the dual eikonal $\widetilde E= \{|p|\ge 1\}$ has the Ahlfors property (viscosity, Ekeland principle).
\item[(3)] the eikonal $E=\{|p| \le 1\}$ has the Khas'minskii property.
\item[(4)] the infinity Laplacian $F_\infty \doteq \overline{\{A(p,p)>0\}}$ has the Ahlfors property.
\item[(5)] $F_\infty$ has the Liouville property.
\item[(6)] $F_\infty$ has the Khas'minskii property.
\item[(7)] $F_\infty$ the next strengthened Liouville  property:\vspace{1ex}
\begin{quote}
Any $F_\infty$-subharmonic $u \ge 0$ such that $|u(x)| = o\big(\varrho(x)\big)$ as $x$ diverges ($\varrho(x)$ the distance from a fixed origin) is constant.
\end{quote}
\end{itemize}
\end{theorem}


\begin{proof}
First, we prove that $(1)$ implies both $(3)$ and $(6)$. Fix a pair $(K,h)$ and choose a point $o \in K$. Since $X$ is complete, $-\varrho(x) \doteq -\dist(x,o)$ is an exhaustion on $M \backslash \{o\}$ which is both $E$ and $F_\infty$-subharmonic. This fact is well known, but we give a quick proof of it via Calabi's trick: for fixed $x_0\neq o$, $\gamma : [0,r(x_0)] \ra X$ a unit speed minimizing geodesic from $o$ to $x_0$, and $\delta \in (0, r(x_0))$, the function $\varrho_\delta(x) \doteq \delta + \dist\big(x, \gamma(\delta)\big)$ is smooth in a neighborhood $U$ around $x_0$ and $\varrho \le \varrho_\delta$ by the triangle inequality, with equality holding at $x_0$. Furthermore, $|\nabla \varrho_\delta|=1$ on $U$, hence $u$ is $F_\infty$-harmonic on $U$ by differentiating. Since any test function $\phi$ touching $-\varrho$ from above at $x_0$ is also a test for $-\varrho_\delta$, we get $J^2_{x_0}\phi \in (F_\infty \cap E)_{x_0}$. Now, using the properties of $h$, we can fix a function $g \in C^\infty(\R)$ such that 
\begin{align*}
&0<g'<1, \ g'' \ge 0 \ \text{ on } \, R, \qquad g(t) \ra -\infty \ \text{ as } \, t \ra -\infty, \\
&0 \ge g(-t) \ge \max_{\partial B_t\backslash K} h \ \text{ for } \, t > 0. 
\end{align*}
%
Concluding, $w(x) \doteq g(-\varrho(x))$ is both $E$ and $F_\infty$-subharmonic, and satisfies all the assumptions to be a Khas'minskii potential for $(K,h)$.\\[0.1cm] 
$(2) \Leftrightarrow (3)$.\\ 
Apply Theorem \ref{teo_main_withgradient} with the choice $F \equiv J^2(X)$. Clearly, $F \cap E^\eta=E^\eta$ is a subequation, $(\HH 1),(\HH 2)$ are immediate and $(\HH 3')$ holds\footnote{For $u \in E(\Omega)$, $\delta u \in E^\str(\Omega)$ when $\delta \in (0,1)$.}.\\[0.1cm]
$(3) \Rightarrow (1)$.\\
Fix a pair $(K,h)$ and a Khas'minskii potential $w$: 
$$
w \in E(X \backslash K), \qquad h \le w \le 0 \quad \text{on } \, X\backslash K, \qquad w(x) \ra -\infty \ \text{ as $x$ diverges.}
$$  
Let $\gamma$ be a unit speed geodesic issuing from some point of $X$, parametrized on a maximal interval $[0,T)$. We shall prove that $T=+\infty$. Suppose, by contradiction, that $T< +\infty$. By ODE theory, $\gamma$ leaves $K$ for $t$ large, hence we can suppose, up to cutting a piece of $\gamma$ and translating the parameter, that $\gamma([0,T)) \cap K = \emptyset$. The map $\gamma$ induces a map 
\begin{align*}
&\gamma^* : J^2(X\backslash K) \ra J^2([0,T)), \quad \text{given by} \\
&\gamma^*(x,u,\di u, \nabla^2 u) = \big(t,u \circ \gamma,(u\circ \gamma)',(u \circ \gamma)''\big),
\end{align*}
and we can consider the pull-back subequation $H = \overline{\gamma^*E}$, which coincides with the Eikonal subequation on $[0,T)$. By the restriction theorem (\cite{HL_restriction}, Thm. 8.2), the function $u \doteq w \circ \gamma$ satisfies $u \in E([0,T))$. Using that $u \le 0$ and $T< +\infty$, we can choose a line of type $l(t) = -2t + C$ ($C$ large enough) disjoint from the graph of $u$. Then, reducing $C$ up to a first contact point $t_0$ between $l$ and the graph of $u$, we would contradict $J^2_{t_0} l \in E$.\\[0.1cm]
$(4) \Leftrightarrow (5)$.\\
It is a direct consequence of Proposition \ref{teo_main_A_L}.\\[0.1cm]
$(6) \Rightarrow (5)$.\\
We follow the proof of $(K_\weak) \Rightarrow (A)$ in Theorem \ref{teo_main}, but with some differences. Assume, by contradiction, that there exists $u \ge 0$ non-constant, bounded and $F_\infty$-subharmonic, and choose a small, relatively compact set $K$ and $\eps>0$ such that $\max_K u + 2\eps < \sup_X u - 2\eps$. Fix $x_0 \in X \backslash K$ for which $u(x_0)> \sup_U u -\eps$, an open set $U$ with $K \cup \{x_0\} \Subset U$, and choose $h \in C(X\backslash K)$ satisfying  
$$
h <0 \quad \text{on  } \, X \backslash K, \qquad h \ge -\eps \ \text{ on } \, U \backslash K, \qquad h(x) \ra -\infty \ \text{ as $x$ diverges.}
$$
Let $w$ be an $F_\infty$-subharmonic, Khas'minskii potential for $(K,h)$, and let $\Omega$ be large enough that $w \le -\sup_U u$ on $X \backslash \overline{\Omega}$. We compare
$$
w - \eps \in F_\infty(\overline{\Omega \backslash K}), \qquad u - \max_K u \in F_\infty(\overline{\Omega \backslash K}),
$$
with the aid of Theorem \ref{compa_inftyLaplacian}: from $w- \eps + u - \max_K u \le 0$ on $\partial \Omega \cup \partial K$ we deduce $w-\eps +u - \max_K u \le 0$ on $\Omega \backslash K$. However,
\begin{align*}
	w(x_0) - \eps + u(x_0) - \max_K u &> -\eps -\eps + \sup_U u - \eps - \max_K u \\
	&= \sup_U u - \max_K u - 3\eps > \eps,
\end{align*}
a contradiction.\\[0.1cm]
$(4) \Rightarrow (1)$.\\
Fix a relatively compact, open set $K$ and a smooth exhaustion $\{\Omega_j\} \uparrow X$ with $K \Subset \Omega_j$ for each $j$. Let $u_j$ solve 
\begin{equation}\label{equation}
\begin{cases}
u_j \ \text{ is $F_\infty$-harmonic on $\Omega_j \backslash K$,} \\
u_j = 0 \quad \text{on } \, \partial K, \quad u_j =1 \quad \text{on } \, \partial \Omega_j.
\end{cases}
\end{equation}
The existence of a (unique) solution $u_j \in \lip(\Omega_j \backslash K)$ follows by combining, with minor modification, results in \cite{juutinen, crandall_visit}\footnote{More precisely, the proof of the existence of $u$ can be obtained as follows: one begins by considering the existence problem for an absolutely minimizing Lipschitz (AML) extension $u \in \lip(\Omega_j \backslash K)$ of the boundary datum $\beta = 0$ on $\partial K$, $1$ on $\partial \Omega_j$ (see \cite{crandall_visit} for definitions). Existence is guaranteed in \cite{juutinen}. Next, the reader can follow the proof of \cite[Thm. 2.1]{crandall_visit} to show that an AML extension $u$ is $F_\infty$-harmonic. The argument, which uses the equivalence between $u \in $AML and $u$ satisfying the comparison with cones property, does not use specific properties of $\R^m$, and is therefore applicable to the manifold setting with few modifications.}. Moreover, 
\begin{equation}\label{problema}
\|\di u_j\|_\infty = \min\Big\{ \|\di v\|_{\infty} \ : \ v \in \lip(\Omega_j \backslash K), \ v=u \ \text{ on } \, \partial \Omega_j \cup \partial K \Big\}
\end{equation}
and $0 \le u_j \le 1$ in view of the comparison Theorem \ref{compa_inftyLaplacian}. Extending $u_j$ with one outside of $\Omega_j$, $\{u_j\}$ is a sequence of uniformly Lipschitz functions with non-increasing Lipschitz constant, and thus it subconverges locally uniformly to a Lipschitz limit $u_\infty \ge 0$. By Proposition \ref{prop_basicheF}, $u_\infty$ is $F_\infty$-harmonic and, since $u_\infty=0$ on $\partial K$, $u_\infty=0$ on $X \backslash K$ because of the Ahlfors property. Now, pick a point $o \in K$ and a unit speed geodesic $\gamma : [0,T) \ra X$. If, by contradiction, $T<+\infty$, then as before we can suppose that $\gamma((0,T)) \subset X \backslash K$ and $\gamma(0)\in \partial K$. Define the functions $w_j = u_j \circ \gamma$, and note that $w_j(0)=0$ and $w_j =1$ after some $T_j<T$. Integrating, $1/T \le \|w_j'\|_\infty \le C$ on $[0,T)$, for each $j$. However, $w_j \ra 0$ locally uniformly, a contradiction.\\[0.1cm]
$(7) \Rightarrow (5)$. Obvious.\\[0.1cm]
$(1) \Rightarrow (7)$. We proceed as in $(6)\Rightarrow (5)$, observing that the distance function $-r$ from a point $o \in K$ is $F_\infty$-subharmonic: we use $-\delta r$ instead of $w$, for $\delta$ small enough that $-\delta r(x_0) \ge -\eps$. Once $\delta$ is fixed, the sublinear growth in $(7)$ guarantees that $-\delta r \le \sup_{U} u$ on $X \backslash \Omega$, for $\Omega$ large. The conclusion is reached by applying the comparison for $F_\infty$ as in $(6)\Rightarrow (5)$.
\end{proof}

\begin{remark}
\emph{The characterization $(1) \Leftrightarrow (4)$ has been inspired by Theorems 2.28 and 2.29 in \cite{pigolasetti_ensaio}, stating that $X$ is complete if and only if
\begin{equation}\label{capac_infty}
\mathrm{cap}_\infty(K) \doteq \inf\Big\{ \|\di u\|_\infty, \  : \ u \in \lip_c(X), \ u=1 \ \text{ on } \, K\Big\}
\end{equation}
is zero for each compact set $K$. Since solutions $u_j$ of \eqref{equation} satisfy \eqref{problema}, if $\mathrm{cap}_\infty(K)=0$ then automatically $u_\infty=0$. Indeed, 
$$
\begin{array}{ll}
\text{from $u_j \!\ra\! u$ locally uniformly,} &   \disp \|\di u_\infty\|_\infty \!\le\! \liminf_{j \ra +\infty} \|\di u_j\|_\infty =\!\! \lim_{j \ra +\infty} \|\di u_j \|_\infty, \\[0.3cm]
\text{directly by \eqref{problema} and \eqref{capac_infty},} &   \disp \mathrm{cap}_\infty(K) \equiv \lim_{j \ra +\infty} \|\di u_j\|_\infty.
\end{array}
$$
Therefore, the completeness of $X$ follows from the argument at the end of $(4)\Rightarrow (1)$, giving an alternative proof of the ``if part" of the result in \cite{pigolasetti_ensaio}. 
}
\end{remark}

\section{Martingale completeness}
In this section, we consider the subequations describing the viscosity, weak and strong Hessian principles and prove Theorem \ref{teo_hessianmax_intro}. We recall that 
$$
E_\xi = \overline{\big\{ |p| < \xi(r) \big\}} \qquad \text{has dual} \qquad \widetilde{E_\xi} = \overline{\big\{ |p| > \xi(-r) \big\}},
$$
and that $(f1),(\xi 1)$ are defined in \eqref{def_f1xi1}.

\begin{theorem}\label{teo_hessianmax}
Consider the subequations 
$$
F = \{ \lambda_1(A) \ge -1\} \qquad \text{and} \qquad F_f = \{ \lambda_1(A) \ge f(r)\}, 
$$
for some $f \in C(\R)$ non-decreasing. Then, 
\begin{itemize}
\item[-] AK-duality holds both for $F_f$ and for $F_f\cap E_\xi$, for some (equivalently, each) $(f,\xi)$ satisfying $(f1+\xi 1)$. 
\end{itemize}
Moreover, the following properties are equivalent:
\begin{itemize}
\item[(1)] $X$ satisfies the viscosity, weak Hessian principle;
\item[(2)] $X$ satisfies the viscosity, weak Hessian principle for semiconcave  functions;\item[(3)] $\widetilde{F_f}$ has the Ahlfors property for some (each) $f$ of type $(f1)$;
\item[(4)] $X$ satisfies the viscosity, strong Hessian principle;
\item[(5)] $\widetilde{F_f} \cup \widetilde{E_\xi}$ has the Ahlfors property, for some (each) $(f,\xi)$ satisfying $(f1+\xi 1)$;
\item[(6)] $F_f \cap E_\xi$ has the Khas'minskii property with $C^\infty$ potentials, for some (each) $(f,\xi)$ satisfying $(f1+\xi 1)$.
\end{itemize}
In particular, each of $(1), \ldots, (6)$ implies that $X$ is geodesically complete and martingale complete.
\end{theorem}

\begin{remark}
\emph{It is worth to recall that M. Emery in \cite[Prop. 5.36]{emery} proved that if $X$ is martingale complete, then $X$ is necessarily (geodesically) complete. Theorem \ref{teo_hessianmax} improves on \cite{prs_overview}, where partial versions of some of the implications were shown. More precisely, 
\begin{itemize}
\item[-] if $X$ has the $C^2$, weak Hessian principle, then $X$ is non-extendible (that is, not isometric to a proper, open subset of a larger manifold), a condition weaker than the geodesic completeness. 
\item[-] If $\{ \lambda_1(A) \ge f(r)\}$ has the Khas'minskii property with a $C^2$ potential, then $X$ is complete.
\end{itemize}
}
\end{remark}


\begin{proof}[Proof of Theorem \ref{teo_hessianmax}] First, AK-duality holds for $F_f$ and $F_f \cap E_\xi$ as a particular case of Theorem \ref{cor_bonito}. The equivalences $(1) \Leftrightarrow (3)$ and $(4) \Leftrightarrow (5)$, as well as $(5) \Rightarrow (3)$, are consequences of Propositions \ref{prop_equivalenceahlfors}, \ref{prop_equivalenceahlfors_withgradient} and the rescaling properties of $F$, $E$, and $(1) \Rightarrow (2)$ is obvious. By AK-duality, $(6)$ implies both $(3)$ and $(5)$. We are left to prove $(2) \Rightarrow (3)$, and $(1) \Rightarrow (6)$ for each $(f,\xi)$ enjoying $(f1+\xi 1)$.\\[0.2cm]
$(1) \Rightarrow (6)$.\\
We split the proof into the following steps:
\begin{itemize}
\item[$(\alpha)$] If $(1)$ holds, then for each $\eps>0$ the Khas'minskii property holds for $F_f \cap E_\xi$ with $f = -\eps$, $\xi = \eps$, and with $C^\infty$ potentials; \vspace{0.1cm}
\item[$(\beta)$] If $(\alpha)$ holds, then $F_f \cap E_\xi$ has the Khas'minskii property for each $(f,\xi)$ enjoying $(f1+\xi 1)$, with $C^\infty$ potentials.
\end{itemize}
We first prove $(\beta)$: fix $(f,\xi)$ and a pair $(K,h)$. By $(\alpha)$, for each $\eps>0$ we can choose a smooth Khas'minskii potential for $(K,h/2)$: 
\begin{equation}\label{potK}
\begin{array}{l}
\disp h/2 \le w \le 0 \quad \text{on } \, X \backslash K, \qquad w(x) \ra -\infty \quad \text{as $x$ diverges,} \\[0.2cm]
\disp |\nabla w| \le \eps, \qquad \nabla^2 w \ge -\eps \metric \qquad \text{on } \, X \backslash \overline{K}.
\end{array}
\end{equation}
Fix $\delta>0$ such that $w-\delta > h$, which is possible since $h -w \le h/2$ has a negative maximum on $X \backslash K$. Next, up to rescaling $w$ with a small, positive constant, we can reduce the value of $\eps$ in \eqref{potK} as small as we wish, still keeping the validity of $w-\delta > h$. In particular, because of $(f1),(\xi 1)$ we can reduce $\eps$ up to satisfy  
\begin{equation*}
-\eps \ge f(-\delta), \qquad \eps \le \xi(-\delta).
\end{equation*}
Setting $w_\delta \doteq w-\delta$ and recalling that $f$ is non-decreasing and $\xi$ is non-increasing, we obtain
\begin{align*}
&\disp h < w_\delta \le -\delta \quad \text{on } \, X \backslash K, \qquad w_\delta(x) \ra -\infty \quad \text{as $x$ diverges,} \\
&\disp |\nabla w_\delta| \le \eps \le \xi(-\delta) \le \xi(w_\delta), \\
&\nabla^2 w_\delta \ge -\eps \metric \ge f(-\delta)\metric \ge f(w_\delta) \metric \qquad \text{on } \, X \backslash \overline{K}.
\end{align*}
Hence, $w_\delta$ is the smooth Khas'minskii potential for $F_f \cap E_\xi$ and the pair $(K,h)$.

To conclude, we prove $(\alpha)$. By $(1)$, Proposition \ref{prop_equivalenceahlfors} and the rescaling properties of $F = \{ \lambda_1(A) \ge -1\}$, the Ahlfors property holds for each $\widetilde{F_f}$ with $f$ enjoying $(f1)$ and $f > -1$. Fix one such $f$. By AK-duality, for a pair $(K,h)$ we can choose a small geodesic ball $B \Subset K$ and take an $F_f$-subharmonic Khas'minskii potential $w_0$ for $(B,h/2)$, which because of $f > -1$ satisfies
\begin{equation}\label{potK_weak}
\begin{aligned}
&\disp h/2 \le w_0 \le 0 \quad \text{on } \, X \backslash B, \\
&w_0(x) \ra -\infty \quad \text{as $x$ diverges,} \qquad w_0 \in F(X \backslash B).
\end{aligned}
\end{equation}
The idea is to approximate $w_0$ uniformly with a semiconvex, $C^\infty$-function $\bar w$, obtain from $\bar w$ a smooth function $w$ which has bounded gradient, and modify $w$ to get the desired Khas'minskii potential.

For each $x_0 \in X\backslash \overline{B}$, take a small neighborhood $U \subset X \backslash \overline{B}$ around $x_0$ such that $\nabla^2 \varrho_{x_0}^2 > 3/2 \metric$ on $\overline{U}$ ($\varrho_{x_0}$ the distance from $x_0$). Using \eqref{potK_weak},  $w_0+\varrho_{x_0}^2$ has positive Hessian in viscosity sense on $U$, hence it is convex\footnote{We recall that $u \in\USC(X)$ is called convex if its restriction to geodesics is a convex function on $\R$. The proof that a viscosity solution $u$ of $\nabla^2 u \ge 0$ is convex can be found in \cite[Prop. 2.6]{HL_primo}, and also as a consequence of the restriction Theorem 8.2 in \cite{HL_restriction}.}. In particular, $w_0 \in C(X \backslash \overline{B})$ and, according to the terminology in \cite{greene_wu}, $w_0$ is $(-1)$-convex. Proposition 2.2 in \cite{greene_wu} guarantees that $w_0$ can be uniformly approximated on $X \backslash \overline{B}$ by smooth, $(-1)$-convex functions: up to translating downwards, we can therefore pick $\bar w \in C^\infty(X \backslash K)$ satisfying
\begin{equation}\label{ipobarw}
\begin{aligned}
&h \le \bar w < 0 \quad \text{on } X \backslash K, \qquad \bar w(x) \ra -\infty \ \ \text{ as $x$ diverges,} \\
&\nabla^2 \bar w > -\metric \quad \text{ on } \, X \backslash K. 
\end{aligned}
\end{equation}
Extend $\bar w, h$ smoothly on the entire $X$ in such a way that $h \le \bar w \le 0$ on $X$, and define $g \doteq -\bar w+1$. Then, for some constant $\lambda>0$,
\begin{equation}\label{bonita_ODE!!}
\begin{aligned}
&1 \le g \le 1 + |h| \quad \text{on } \, X, \qquad g(x) \ra +\infty \quad \text{if $x$ diverges,} \\
&\nabla^2 g \le \metric \le g \metric \quad \text{on } \, X.  
\end{aligned}
\end{equation}
We claim that $\log g$ has bounded gradient: more precisely, we want to prove that $|\nabla g| \le g$ on $X$. To this end, for $x \in X$ such that $|\nabla g(x)| \neq 0$, we consider a maximal flow curve $\gamma$ for $\nabla g$ passing through $x$ at time $0$. Since flow lines do not intersect, $|\nabla g| \neq 0$ on the entire $\gamma$ and we can reparametrize $\gamma$ according to the arclength $s \in (a,b)$, with $\gamma(0)=x$. We claim that $a = -\infty$. Otherwise, by ODE theory the curve $\gamma$ would diverge (i.e. leave every compact set) as $s \ra a^+$, but this is impossible since the second condition in \eqref{bonita_ODE!!} implies that $\gamma((a,0])$ is confined into the compact set $\{ g \le g(x)\}$. Define $u \doteq g \circ \gamma$. By \eqref{bonita_ODE!!}, 
$$
u'(s) = \big|\nabla g(\gamma(s))\big| >0, \qquad u''(s) \le u(s) \quad \text{on } \, (-\infty, b),
$$
and $u \ra +\infty$ as $s \ra b$. For each $t < 0$ fixed, we define
$$
\phi_t(s) \!=\! \cosh\big(s\!-\!t \big) \!+\! \frac{u'(t)}{u(t)} \sinh\big(s\!-\!t \big), \quad \text{which solves } \ \phi_t''(s) = \phi_t(s) \quad \text{on } \, \R.
$$
Integrating $(u'\phi_t - u\phi_t')' \le 0$ on $[t,0]$ and using the positivity of $u,\phi_t$ there, we get
\begin{align}\label{good}
	\frac{|\nabla g|}{g}(x) = \frac{u'(0)}{u(0)} &\le \frac{\phi_t'(0)}{\phi_t(0)} = \frac{\sinh(- t)+ u'(t)/u(t) \cosh(-t)}{\cosh(- t)+ u'(t)/u(t) \sinh(-t)}\\
	\notag &\le \tanh(- t) + \frac{u'(t)}{u(t)}. 
\end{align}
Since $u$ is increasing and bounded from below as $s \ra -\infty$, there exists a sequence $\{t_j\} \ra -\infty$ with $u'(t_j) \ra 0$ as $j \ra +\infty$. Evaluating \eqref{good} on $t_j$, letting $j \ra +\infty$ and recalling that $u \ge 1$, we deduce $|\nabla g| \le g$ at $x$, as claimed.

To conclude we note that, for $\eps \in (0,1)$, $w = -\eps \log g$ satisfies
\begin{align*}
	&h \le w \le 0 \quad \text{on } \, X\backslash K, \quad w(x) \ra -\infty \ \  \text{ as $x$ diverges,}  \\
	&|\nabla w| \le \eps, \quad \nabla^2 w \ge - \eps \metric \ \ { on } \ \ X,
\end{align*}
concluding the proof of $(\alpha)$.\\[0.2cm]
$(2) \Rightarrow (3)$.\\
Suppose, by contradiction, that the Ahlfors property does not hold for $\widetilde{F_f} = \{ \lambda_m(A) \ge -f(-r)\}$, for some $f$ satisfying $(f1)$. In view of Proposition~\ref{prop_equivalenceahlfors}, we can suppose that $f$ is strictly increasing and that $f \ge -1$. By Proposition \ref{teo_main_A_L} the Liouville property is violated, and we can take $u \in \widetilde{F_f}(X)$ non-constant, bounded and non-negative. Up to fixing a level $c \in (\inf_X u,\sup_X u)$ and replacing $u$ with $\max\{u-c,0\}$, we can also assume that $u$ is zero on some small, open ball $B$. Our aim is to produce a non-constant function $v \ge u$ which is $\widetilde{F_f}$-harmonic and non-constant. The construction is close to the one used in Theorem \ref{teo_main}: to begin with, by gluing inside of the ball $B$ we produce the extended manifold $\cal X$ and the sets $\cal V, \cal K$ as in Step 1 of the proof of Theorem \ref{teo_main}, and we extend $F$ and $u$ in the obvious way. Since $\partial \cal V$ is a convex Euclidean sphere, it is $F_f$-convex at height zero by Proposition \ref{prop_condition_F_convex}. Let $\cal D_j \uparrow \cal X$ be a smooth exhaustion  containing $\overline{\cal V}$. Define $u_\infty \doteq \sup_{\cal X} u$, and consider the Dirichlet problem
\begin{equation}\label{obstacolino}
 \begin{cases}
v_j \ \text{ is $\widetilde{F_f}$-harmonic on $\cal D_j \backslash \cal V$,} \\
v_j = 0 \ \  \text{ on } \, \partial \cal V, \quad v_j = u_\infty \ \ \text{ on } \, \partial \cal D_j.
\end{cases}
\end{equation}
To check that the problem has a solution, as usual we consider the Perron's class and envelope

$$
\begin{array}{l}
\disp \mathcal{F}_j = \Big\{v\in \widetilde{F_f}(\overline{\cal D_j \backslash \cal V}) \  :  \ v \le 0 \ \text{ on } \, \partial \cal V, \quad v \le u_\infty \ \text{ on } \, \partial \cal D_j\Big\}, \\[0.3cm]
v_j(x) = \sup\big\{v(x): v\in \mathcal{F}_j\big\}. 
\end{array}
$$
Observe that $u \in \mathcal{F}_j$. As for boundary convexity, from $\{ \tr(A) \ge -mf(-r)\} \subset \widetilde{F_f}$ we deduce that the boundary of each smooth open set in $\cal X$ is $\widetilde{F_f}$-convex (since it is so for $\{\tr(A) \ge -mf(-r)\}$, see, e.g. Subsection 7.4 in \cite{HL_existence}). Therefore, $\partial \cal V$ is both $F_f$ and $\widetilde{F_f}$-convex. Using that the function $-u_\infty$ is strictly $F_f$-subharmonic, by weak comparison $v_j \le u_\infty$. Hence, $\partial \cal D_j$ has good barriers from below (by $\widetilde{F_f}$-convexity) and from above ($u_\infty$). Being $f$ strictly increasing, strict approximation and thus full comparison holds for $F_f$, and we are in the position to follow the steps in Theorem \ref{obstaclethm} (or in \cite[Thm. 12.4]{HL_dir}) to ensure the existence of a solution to \eqref{obstacolino}. Now, from 
$$
- v_j \in F_f(\cal D_j\backslash \cal V), \qquad F_f = \big\{ \lambda_1(A) \ge f(r) \big\} \subset \big\{ \lambda_1(A) \ge -1 \big\},
$$
we obtain that $-v_j$ is $(-1)$-convex for each $j$, hence it is locally Lipschitz. Because of comparison,  $u \le v_{j+1} \le v_j \le u_\infty$ for each $j$, thus the Lipschitz bounds are locally uniform (see \cite[Sect. 6.3]{evansgariepy})\footnote{The result is stated for $\R^m$, but can easily be modified for small, regular balls in a Riemannian manifold.}. Consequently, $\{v_j\}$ converges locally uniformly to some $v \ge u$ which is $\widetilde{F_f}$-harmonic on $\cal X \backslash \cal V$ (by Proposition \ref{prop_basicheF}). In particular, $v$ is semiconcave. Since $v=0$ on $\partial \cal V$, $v$ is not constant and thus it does not attain a positive local maximum (Lemma \ref{lem_maxprinc}). In particular, $\max_{\overline{\cal K} \backslash \cal V} v < \sup_{\cal X\backslash \cal V} v$. Having fixed
$$
c \in \left(\max_{\overline{\cal K} \backslash \cal V} v, \sup_{\cal X\backslash \cal V}v\right), 
$$
the function $v_c \doteq \max\{ v-c,0\}$ is $\widetilde{F_f}$-subharmonic, semiconcave on the set $\cal U \doteq \{v>c\}$, non-constant, and zero on $\partial \cal U$. Concluding, from $\cal U \subset \cal X\backslash \cal K$, we can transplant $\cal U$ and $v_c$ on $X \backslash K$ to contradict the Ahlfors property in $(2)$, as claimed.

Having shown the equivalence between $(1), \ldots, (6)$, we observe that the Ahlfors property for $\widetilde{F_f} \cup \widetilde{E_\xi}$ with $\xi \le 1$ implies that for $\widetilde{E}$, hence $X$ is geodesically complete by Theorem \ref{teo_ekeland}. Since $(5)$ implies the existence of a $C^2$ Khas'minskii potential in \eqref{khasmi_Hessian_original} (just choose $f,\xi$ bounded), the martingale completeness of $X$ follows via Emery's theorem in \cite[Prop. 5.37]{emery}.
\end{proof}

\begin{remark}\label{prop_weak_classicalvisco}
\emph{The argument in $(2) \Rightarrow (3)$ can be applied to the subequation $F_f = \{ \tr(A) \ge f(r)\}$ in place of $\{ \lambda_1(A) \ge f(r)\}$: in view of elliptic estimates for semilinear equations, for smooth $f$ the resulting functions $v_j$ are $C^2$ and converge to a $C^2$-function $v$ that, by construction, contradicts the Ahlfors property for $\widetilde{F_f}$. In particular, this shows equivalence between the viscosity, weak Laplacian principle and the classical formulation of it for $C^2$ functions in \cite{prsmemoirs}, that is, the stochastic completeness of $X$.
}
\end{remark}

\appendix

\section[Remarks on comparison for quasilinear\\ \mbox{}\hspace{4.7em} equations]{Remarks on comparison for quasilinear equations}\label{appendix_1}

In this section, we prove Proposition \ref{comparison_examples} and give some remarks on comparison for quasilinear equations. First we need to fix some notation. Let $\varrho$ denote the distance function in $X$ and, for a pair of points $x, y \in X$ with $\varrho(x,y)< \min\{\mathrm{inj}(x),\mathrm{inj}( y)\}$ ($\mathrm{inj}(z)$ the injectivity radius at $z$), let $\mathcal{L}: T_{x}X \ra T_{y}X$ be the parallel translation along the unique segment from $x$ to $y$. Set $\cal L^*$ to denote the induced pull-back on covariant tensors, whence for instance $\cal L^*A(V,W) \doteq A(\cal LV, \cal LW)$ for $A \in \Sym^2(T_yX)$. Define $\varrho_x(y) \doteq \varrho(x,y)$ for $x$ fixed, and similarly for $\varrho_y$.

\begin{proposition}
Fix $(f,\xi)$ satisfying $(f1 + \xi 1)$, and assume that
$$
\text{$f$ is strictly increasing on $\R$ $\quad$ and $\quad$ $\xi$ is strictly decreasing on $\R$}. 
$$
\begin{itemize}
\item[$i)$] If $F_f \subset J^2(X)$ is locally jet-equivalent to one of the examples in $(\EE 2), \ldots, (\EE 6)$ via locally Lipschitz bundle maps, then the bounded comparison holds for $F_f$ and for $F_f \cap E_\xi$. \vspace{0.1cm}
\item[$ii)$] If $F_f$ is the universal, quasilinear subequation in $(\EE 7)$ with eigenvalues $\{\lambda_j(t)\}$ in \eqref{ipo_a}, then the bounded comparison holds in the following cases: 
\begin{itemize}
\item[-] for $F_f$, provided that $\lambda_j(t) \in L^\infty(\R_0^+)$ for $j \in \{1,2\}$; 
\item[-] for $F_f \cap E_\xi$, provided that $\lambda_j(t) \in L^\infty_{\mathrm{loc}}(\R_0^+)$ for $j \in \{1,2\}$. 
\end{itemize}
\item[$iii)$] If $F_f$ is the universal subequation in $(\EE 8)$, then the bounded comparison holds for $F_f$ and for $F_f \cap E_\xi$. 
\end{itemize}
Furthermore, for each of $i), \ldots, iii)$, comparison also holds for the obstacle subequations $F_f^0$ and $F_f^0 \cap E_\xi$.
\end{proposition}


\begin{proof}
In the assumptions of $i)$, by Theorem \ref{teo_importante!!} and Proposition \ref{lem_weakcomp_congradient} both $F_f$ and $F_f\cap E_\xi^\eta$ satisfy the weak comparison. By Proposition \ref{prop_unifcontinuous}, $\mathscr{F}$ is uniformly continuous, and therefore Theorem \ref{thm_comparison_examples} and Proposition \ref{thm_comparison_examples_withgradient} ensure the strict approximation property for $F_f$ and $F_f \cap E_\xi^\eta$. Hence, the result follows from Theorem \ref{localwcestrict} and Lemma \ref{wcobstacle}.

On the other hand, to reach the result for $ii)$ and $iii)$, comparison is more conveniently checked via the Riemannian theorem on sums in \cite{azagraferrerasanz}. Let $K \subset X$ be compact, $\mathbb{F} \subset J^2(X)$ be a subequation, $u \in \mathbb{F}(K)$, $v \in \widetilde{\mathbb{F}}(K)$, with $\vert u\vert, \vert v \vert \leq R$. Suppose that $u+v \le 0$ on $\partial K$ but $\max_K(u+v) = c_0>0$. Let $-\kappa \le 0$ be a lower bound for the sectional curvatures of planes over points of $K$. Then, by the Riemannian theorem on sums, there exist sequences of points $(x_\alpha, y_\alpha) \in \inte K \times \inte K$ with the following properties:
\begin{itemize}
\item[(i)] $(x_\alpha,y_\alpha) \ra (\bar x, \bar x)$ as $\alpha \ra +\infty$, for some maximum point $\bar x \in \inte K$ of $u+v$; 
\item[(ii)] $\alpha \varrho(x_\alpha,y_\alpha)^2 \ra 0$ as $\alpha \ra +\infty$; 
\item[(iii)] $u(x_\alpha) + v(y_\alpha) \doteq c_\alpha \downarrow c_0$; 
\item[(iv)] there exist $A_\alpha \in \Sym^2(T_{x_\alpha}X), B_\alpha \in \Sym^2(T_{y_\alpha}X)$ such that 
$$
\begin{array}{lcl}
(r_\alpha, p_\alpha, A_\alpha) \doteq \big(u(x_\alpha), \alpha \varrho_{x_\alpha} \di \varrho_{x_\alpha}, A_\alpha\big) \in \mathbb{F}_{x_\alpha}, \\[0.2cm]
(s_\alpha, q_\alpha, B_\alpha) \doteq \big(v(y_\alpha), \alpha \varrho_{y_\alpha} \di \varrho_{y_\alpha}, B_\alpha\big) \in \widetilde{\mathbb{F}}_{y_\alpha}, \\[0.2cm]
A_\alpha + \cal L_\alpha^*B_\alpha \le \kappa \alpha \varrho(x_\alpha,y_\alpha)^2 I,
\end{array}
$$
\end{itemize}
where $\cal L_\alpha$ is the parallel translation along the (unique, for $\alpha$ large) geodesic from $x_\alpha$ to $y_\alpha$.
Since $\cal L_\alpha$ is orthonormal, $\cal L_\alpha^*(\di \varrho_{y_\alpha}) = -\di \varrho_{x_\alpha}$ and $\mathbb{F}$ is universal, 
$$
\cal L_\alpha^*(s_\alpha, q_\alpha, B_\alpha) = (s_\alpha, -p_\alpha, \cal L_\alpha^* B_\alpha) = (-r_\alpha + c_\alpha, -p_\alpha, \cal L_\alpha^* B_\alpha) \in \widetilde{\mathbb{F}}_{x_\alpha}. 
$$
Now, suppose first that $\mathbb{F}$ is the subequation in $(\EE7)$. We can assume without loss of generality that $\vert p_\alpha\vert \not= 0$. By $(iii)$, the first in $(iv)$, and positivity and negativity properties,
\begin{align*}
\tr\big[T(p_\alpha)A_{\alpha}\big] & \ge  f(r_\alpha),\\
\tr\big[T(-p_\alpha)\left(\kappa \alpha \varrho(x_\alpha,y_\alpha)^2 I-A_{\alpha}\right)\big] & \ge  -f(r_\alpha - c_\alpha).
\end{align*}
Summing the above inequalities and using that $T(p)=T(-p)$ we obtain
\begin{equation}\label{ineq_k_appendix}
\kappa \alpha \varrho(x_\alpha,y_\alpha)^2 \tr\big[T(p_\alpha)I\big] \ge f(r_\alpha)-f(r_\alpha - c_\alpha).
\end{equation}	
Assuming that $f$ is strictly increasing and satisfies $(f1)$, since $c_{\alpha} \downarrow c_0 > 0$,
$$ \inf_{\vert r_\alpha\vert \leq R}\left(f(r_\alpha) - f(r_\alpha -c_\alpha)\right) > 0.$$
On the other hand, being $\lambda_j \in L^{\infty}(\R_0^+)$, by $(ii)$ the left side in \eqref{ineq_k_appendix} goes to zero. Contradiction. When we consider $\mathbb{F}\cap \mathbb{E}$ arguing as above we have
\begin{eqnarray*}
(r_\alpha, p_\alpha, A_\alpha) \doteq \big(u(x_\alpha), \alpha \varrho_{x_\alpha} \di \varrho_{x_\alpha}, A_\alpha\big) \in \left(\mathbb{F}\cap \mathbb{E}\right)_{x_\alpha}, \\
\cal L_\alpha^*(s_\alpha, q_\alpha, B_\alpha) = (s_\alpha, -p_\alpha, \cal L_\alpha^* B_\alpha) = (-r_\alpha + c_\alpha, -p_\alpha, \cal L_\alpha^* B_\alpha) \in (\widetilde{\mathbb{F}\cap\mathbb{E}})_{x_\alpha}.
\end{eqnarray*}
If for some $\alpha$, $(-r_\alpha + c_\alpha, -p_\alpha, \cal L_\alpha^* B_\alpha) \in \widetilde{\mathbb{E}}_{x_\alpha}$, then
$$ \xi(r_\alpha - c_\alpha) \leq \vert p_\alpha\vert \leq \xi(r_\alpha) ,$$
that contradicts the fact that $\xi$ is strictly decreasing and $c_\alpha > 0$. Therefore, $\vert p_\alpha\vert$ is uniformly bounded and in order to conclude from \eqref{ineq_k_appendix} we only need  $\lambda_j \in L_{\mathrm{loc}}^{\infty}(\R_0^+)$.

In a similar way, if $\mathbb{F}$ is the universal Riemannian subequation in $(\EE8)$, we can assume $\vert p_\alpha\vert \not= 0$ and compute
\begin{align*}
\vert p_\alpha\vert^{-2}A_{\alpha}(p_\alpha ,p_\alpha) & \ge  f(r_\alpha),\\
\vert -p_\alpha\vert^{-2}\left(\kappa \alpha \varrho(x_\alpha,y_\alpha)^2 I-A_{\alpha}\right)(p_\alpha ,p_\alpha) & \ge  f(r_\alpha).
\end{align*}
Summing up, we get
$$\kappa \alpha \varrho(x_\alpha,y_\alpha)^2 > f(r_\alpha) - f(r_\alpha -c_\alpha) \ge \inf_{\vert r_\alpha\vert \leq R}\left(f(r_\alpha) - f(r_\alpha -c_\alpha)\right) > 0,$$
and the contradiction follows by $(ii)$. The case $\mathbb{F}\cap \mathbb{E}$ follows easily as above.

Furthermore, the comparison also holds for the obstacle subequation $\mathbb{F}^0$ (resp. $(\mathbb{F\cap E})^0$) because being $0$ a $\widetilde{\mathbb{F}}$-subharmonic function (resp. $\widetilde{\mathbb{F\cap E}}$-subharmonic), if $v \in \widetilde{\mathbb{F}^0}$ (resp. $v \in \widetilde{(\mathbb{F\cap E})^0}$), then $\widetilde{v} \doteq \max\{v,0\} \in \widetilde{\mathbb{F}}$ (resp. $\widetilde v \in \widetilde{\mathbb{F\cap E}}$). 
\end{proof}

In view of \eqref{ineq_k_appendix}, the conditions on $\{\lambda_j(t)\}$ in $ii)$ can be removed when the lower bound on the sectional curvature of $X$ is $\kappa = 0$, and this strongly suggests that the restriction is merely technical, as discussed in the section ``open problems". In some further cases, one can get rid of the restriction on $\{\lambda_j(t)\}$ and include other relevant classes of operators. We briefly examine two of them.

\subsubsection*{$k$-Laplacian type operators}
We observe that the $k$-Laplace operator, $1 \le k < +\infty$, is excluded in $ii)$ (when there is no gradient control) and allowed just for $k \ge 2$ if coupled with $E_\xi$. However, in the first case, if $k>1$ one could still obtain the validity of the bounded comparison by showing the equivalence between viscosity and weak solutions of $\Delta_k u \ge f(u)$, for instance following the ideas in \cite{julinjuutinen}, and then resorting on the standard comparison for weak solutions. For more general classes of operators to which the above procedure might apply, see \cite{fangzhou}. It is worth to point out that a slightly less general version of the AK-duality has already been shown for $k$-Laplacian type operators in \cite{marivaltorta}, by using weak solutions from the very beginning. However, clearly the approach does not allow to cover the case when a gradient condition is included.
\subsubsection*{Quasilinear operators as in Serrin-Ivanov's papers \cite{serrin_barriers, ivanov}}
When $F$ is locally jet-equivalent to $(\EE 7)$, Euclidean representations of such examples in local charts introduce a dependence on coefficients that complicates the use of the theorem on sums even on $\R^m$. However, suppose that the Euclidean representation $\mathcal{F}$ on $V \subset \R^m$ is included in the classes studied in \cite{ivanov} (which improves on \cite{serrin_barriers}). Then, by \cite[Thm. 10.3]{ivanov} one can avail of the following properties:
\begin{itemize}
\item[1)] On small enough balls of $V$, the Dirichlet problem for $\cal F$-harmonics, $C^2$ solutions is solvable for any smooth boundary data $\varphi$. 
\item[2)] There exist uniform $C^{2,\alpha}$-estimates for $\cal F$-harmonics $u \in C^2(B)$, $B \subset V$, depending on their boundary value $\|\varphi\|_1$. 
\end{itemize}

In this case, comparison for viscosity solutions follows from the next proposition. Unfortunately, the very general class considered in \cite{ivanov} seems not to be invariant under local jet-equivalence, and therefore it is not sufficient to check that the universal model $\mathbb{F}$ in $(\EE 7)$ lies in the class: one needs to work with the local expression of $F$.

\begin{proposition}
Let $F_f$ be locally jet-equivalent to $(\EE 7)$ via locally Lipschitz bundle maps, and suppose that $f$ is strictly increasing. Assume properties $1)$ and $2)$. Then, $F_f$ has the bounded comparison. 
\end{proposition}
\begin{proof}
By contradiction, take $K \subset X$ compact and $u \in F_f(K)$, $v \in \widetilde{F_f}(K)$ with $u+v \le 0$ on $\partial K$ but $(u+v)(x_0)>0$. Let $\mathcal{F}$ be an Euclidean representation of $F_f$ in a chart $(U, \varphi)$ around $x_0$, fix a small Euclidean ball $B \Subset \varphi(U)$ around $y_0= \varphi(x_0)$ and set $\bar u \doteq u\circ \varphi \in \mathcal{F}(\overline{B})$, $\bar v \doteq v\circ \varphi \in \widetilde{\mathcal{F}}(\overline{B})$. Let $\{\psi_j\} \subset C^\infty(\partial B)$ be a decreasing sequence pointwise converging to $\bar u$ on $\partial B$, and by $1)$ consider the solution of the problem
$$
 \begin{cases}
\text{$w_j$ is $\cal F$-harmonic on $B$, and $w_j \in C^2(\overline{B})$,} \\
w_j = \psi_j \quad \text{on } \, \partial B .
\end{cases}
$$
Since $f$ is strictly increasing and $w_j$ is $C^2$, each $w_j$ can be uniformly approximated by $w_j -\eps \in \cal F^\str(\overline{B})$, hence comparison holds for $w_j$. In particular, $\{w_j\}$ is a decreasing sequence, $w_j \ge \bar u$ on $\overline{B}$ and 
$$
\max_{\overline{B}} (w_j + \bar v) = \max_{\partial B} (w_j + \bar v) = \max_{\partial B} (\psi_j + \bar v).
$$ 
By $2)$, $w_j \downarrow w$ locally in $C^{2,\alpha}(B)$ for some $w \in \cal F(\overline{B}) \cap C^2(B)$, and from $\bar u \le w_j \le \psi_j$ on $\partial B$ we deduce that $\bar u = w$ on $\partial B$. As a consequence, since 
$$
(w + \bar v)(y_0) \ge (\bar u + \bar v)(y_0) \ge \sup_{\partial B} (\bar u + \bar v) = \sup_{\partial B} (w +\bar v), 
$$
$w + \bar v$ attains a positive maximum at some interior point $y_1 \subset B$. Let $V \Subset B$ be a small ball centered at $y_1$, and let $\varrho$ be the distance function from $y_1$ in the Euclidean metric. Since $w \in C^2(\overline{V})$ and $f$ is strictly increasing, $w-c \in \widetilde{\mathcal{F}_f}^\str(\overline{V})$ for each $c>0$. Fix $c$ in such a way that $(w-c + \bar v)(y_1)>0$. Also by the $C^2$-regularity of $w$, there exists $\delta>0$ small enough such that $w-\delta \varrho^2 -c \in \widetilde{\mathcal{F}_f}^\str(\overline{V})$. Now,  
\begin{align*}
	(w-\delta \varrho -c + \bar v)(y_1) &= (w-c + \bar v)(y_1)\\
	&= \max_{\overline V}(w-c + \bar v) > \max_{\partial V}(w-\delta \varrho^2-c + \bar v).
\end{align*}
We can now translate $w-\delta \varrho^2 -c$ downward by an appropriate positive constant $c_2$ to contradict the validity of the weak comparison (Theorem~\ref{teo_importante!!}) on $V$.
\end{proof}

\subsubsection*{Acknowledgements}
The authors want to express their gratitude to the referee for his-her very careful reading of the manuscript, and for several suggestions that helped to improve both the results and the presentation. The second author thanks the Universidade Federal do Cear\'a and Universit\`a dell'Insubria - Como for the hospitality during the writing of this paper. A first version of this manuscript was written when the first author was part of the Departamento de Matem\'atica of the Universidade Federal do Cear\'a, that he thankfully acknowledge for the beautiful environment. He also wishes to thank the Scuola Normale Superiore and CNPq for financial support.

\end{document}